\documentclass[EJP]{ejpecp} 

\usepackage[utf8]{inputenc}
\hypersetup{
    colorlinks,
    linkcolor={red!50!black},
    citecolor={blue!50!black},
    urlcolor={blue!80!black}
}
\usepackage{bbm}
\usepackage{mathrsfs}  
\usepackage[normalem]{ulem}
\usepackage[dvipsnames]{xcolor}
\usepackage{cancel}
\usepackage{thmtools}

\SHORTTITLE{Large deviations for heavy-tailed Markov-additive processes}

\TITLE{Sample-path large deviations for a class of heavy-tailed Markov-additive processes
	\support{C.-H.R is supported by NSF Grant CMMI-2146530.}
} 

\AUTHORS{%
Bohan~Chen
	\footnote{Munich Re, Germany.
		\EMAIL{chenbohan1988@hotmail.com}}
	\and
Chang-Han~Rhee
 	\footnote{Northwestern University, United States of America.
		\EMAIL{chang-han.rhee@northwestern.edu}
	}
	\and 
Bert~Zwart
	\footnote{Centrum Wiskunde Informatica, The Netherlands. 
		\EMAIL{bert.zwart@cwi.nl} 
	}
}


\AMSSUBJ{60F10 ; 60G17 ; 60B10 ; 60J05 ; 60G70} 

\SUBMITTED{April 28, 2023} 
\ACCEPTED{} 

\VOLUME{0}
\YEAR{2020}
\PAPERNUM{0}
\DOI{10.1214/YY-TN}

\ABSTRACT{%
For a class of additive processes driven by the affine recursion $X_{n+1} = A_{n+1} X_n + B_{n+1}$, we develop a sample-path large deviations principle in the $M_1'$ topology on $\D [0,1]$. We allow $B_n$ to have both signs and focus on the case where Kesten's condition holds on $A_1$, leading to heavy-tailed distributions. The most likely paths in our large deviations results are step functions with both positive and negative jumps.
}

\newtheorem{result}[theorem]{Result}

\newcommand{\E}{\mathbf{E}}
\renewcommand{\P}{\mathbf{P}}
\newcommand{\R}{\mathbb{R}}
\newcommand{\Z}{\mathbb{Z}}
\newcommand{\1}{\mathbbm{1}}
\newcommand{\toas}{\xrightarrow{\text{a.s.}}}
\renewcommand{\varlimsup}{\limsup}
\renewcommand{\varliminf}{\liminf}

\usepackage[normalem]{ulem}	

\newcommand{\rvout}[1]{{\color{red}{\sout{#1}}}}
\renewcommand{\rvout}[1]{}

\newcommand{\disc}{\mathrm{Disc}}
\newcommand{\D}{\mathbb D}

\newif\ifnavigationlinks

\newcommand{\opacity}{30}

\usepackage[shortlabels]{enumitem}

\newlist{thmdependence}{itemize}{10}
\setlist[thmdependence]{nosep,label=-}

\newcommand{\thmtreenode}[5]{\item[#1] \linkdest{location, thm tree #3} {#2}~\ref{#3} \linktopf{#3} \thmsum{#4}{#5}}
\newcommand{\thmtreenodewopf}[5]{\item[#1] \linkdest{location, thm tree #3} {#2}~\ref{#3} \thmsum{#4}{#5}}
\newcommand{\thmtreeref}[2]{\item[\elsewhere] {{\hyperlink{location, thm tree #2}{\color{gray}#1}}}~\ref{#2}\thmsum{0.5}{}}

\ifnavigationlinks
    \newcommand{\linksinthm}[1]{\emph{\linkdest{location, #1}\linktopf{#1} \linktothmtree{location, thm tree #1} }}
    \newcommand{\linksinthmwopf}[1]{\emph{\linkdest{location, #1} \linktothmtree{location, thm tree #1} }}
    \newcommand{\linksinpf}[1]{\linkdest{location, proof of #1}\linktothm{#1} \linktothmtree{location, thm tree #1} }
    \newcommand{\notationdef}[2]{\linkdest{location, notation definition of #1}\hyperlink{location, notation index of #1}{#2}}
    \newcommand{\notationidx}[2]{\linkdest{location, notation index of #1}\hyperlink{location, notation definition of #1}{#2}}
\else
    \newcommand{\linksinthm}[1]{}
    \newcommand{\linksinthmwopf}[1]{}
    \newcommand{\linksinpf}[1]{}
    \newcommand{\notationdef}[2]{#2}
    \newcommand{\notationidx}[2]{}
\fi

\newcommand{\linktopf}[1]{\hyperlink{location, proof of #1}{\pflinksymbol}}
\newcommand{\linktothm}[1]{\hyperlink{location, #1}{\thmlinksymbol}}
\newcommand{\linktothmtree}[1]{\hyperlink{#1}{\thmtreelinksymbol}}
\newcommand{\thmlinksymbol}{{\tiny [Theorem]}}
\newcommand{\pflinksymbol}{{\tiny [Proof]}}
\newcommand{\thmtreelinksymbol}{{\tiny [ThmTree]}}

\newcommand{\complete}{{\color{black}\checkmark}}

\newcommand{\issue}{{\color{red}\checkmark}}
\newcommand{\elsewhere}{}

\newcommand{\CRA}[1]{{\color{red!\opacity}{[CR: #1]}}} 
\newcommand{\BZ}[1]{{\color{RoyalBlue}[BZ: #1]}} 
\newcommand{\BZA}[1]{{\color{RoyalBlue!\opacity}{[BZ: #1]}}} 

\newcommand{\thmsum}[2]{\quad{\color{gray}\begin{minipage}[t]{#1\linewidth}{#2}\vspace{0.5\baselineskip}\end{minipage}}}

\makeatletter
\newcommand{\linkdest}[1]{\Hy@raisedlink{\hypertarget{#1}{}}}
\makeatother

\begin{document}



\section{Introduction}

Let $\{X_n$, $n\geq 0\}$ be an affine recursion such that 
\begin{equation}\label{eq: affine transformation}
\notationdef{nota-X-n}{X_{n+1}} = A_{n+1} X_n  +B_{n+1}
\end{equation} 
for a sequence of i.i.d.\ $\R^2$-valued random vectors \notationdef{nota-A-n-B-n}{$(A_n, B_n)$}. 
The Markov chain driven by (\ref{eq: affine transformation}) has been studied extensively in the past several decades and continues to pose new research challenges. 
A classical result,  which can be found in  \cite{goldie1991} and  \cite{kesten1973} shows that under certain assumptions  (see Assumption~\ref{ass: regularity condition AR(1) process} below),
 the Markov chain $X_n$, $n\geq 0$ has a unique stationary distribution \notationdef{nota-pi}{$\pi$}, for which we have 
\begin{equation}\label{eq: Kesten and Goldie}
\pi(x,\infty) \sim c_+ x^{-\alpha}
\quad\mbox{and}\quad
\pi(-\infty,-x) \sim c_- x^{-\alpha},
\quad\mbox{as}\ x\to\infty,
\end{equation}
for some $c_-$, $c_+$ satisfying $c_- + c_+ > 0$; 
see the monograph  \cite{buraczewskidamekmikosch2016} for a recent comprehensive account.

Define $\bar X_n = \{\bar X_n(t),t\in[0,1]\}$, with 
\begin{equation}\label{eq: X_n bar introduction}
\notationdef{nota-X-bar-n}{\bar X_n(t)} =  \frac 1n \sum_{i=0}^{\lfloor nt \rfloor-1} X_i, \hspace{1cm} t\in [0,1].
\end{equation}
The focus of the present paper is on sample-path large deviations of the additive process $\bar X_n$, 
assuming the invariant distribution of $X_n$ has a heavy tail as in (\ref{eq: Kesten and Goldie}). 
The study of additive processes of the form (\ref{eq: X_n bar introduction}) is less well developed.
Classical theory initiated by  Donsker and Varadhan  (see, for example, \cite{donsker1975asymptotic, DV_3}) provides powerful tools designed to study large
deviations for additive functionals of light-tailed and geometrically
ergodic Markov chains. More recent contributions in this area include \cite{KM_2, kontoyiannis2003spectral}.  Analogs of these sample-path results in a heavy-tailed setting do not seem to be available.


A considerable body of theory has been developed to analyse exceedance probabilities for random walks with heavy-tailed step sizes. Let such a random walk be given by  $\{\hat S_n, n \geq 0\}$.
Early papers \cite{nagaev1969, nagaev1978} identified appropriate sequences $(x_n)$ for which 
\begin{equation}\label{eq: principle of a single big jump introduction}
\P ( \hat S_n/n > x_n ) = n \P ( \hat S_1 > x_n )(1+o(1)), 
\quad\mbox{as}\ n\to\infty,
\end{equation}
holds, depending on the tail behavior of the distribution of $\hat S_1$.
For a detailed description, we refer to e.g.\ \cite{borovkov2008}, \cite{denisov2008}, and \cite{fosskorshunovzachary2013}. 
When \eqref{eq: principle of a single big jump introduction} is valid, the so-called principle of a single big jump is said to hold. 
As a generalization of \eqref{eq: principle of a single big jump introduction}, a functional form has been derived in \cite{hult2005}, where random walks with i.i.d.\ multi-dimensional regularly varying (cf.\ Definition~1.1 of \cite{hult2005}) step sizes are considered. 

Several works have focused on the extension of (\ref{eq: principle of a single big jump introduction}) to more general processes where there is a certain dependence structure in the increments. 
Some key references are \cite{foss2007, hult2007, mikosch2000}, where stable processes, modulated processes, and stochastic differential equations are considered. 
%
%
Extensions to additive processes of the form considered in this paper have been provided in \cite{mikosch2013,mikosch2016}, which also consider more general examples of driving recursion $X_{n+1} = f_{n+1}(X_n)$. The principle of a single big jump is still valid, but an additional constant in the RHS of (\ref{eq: principle of a single big jump introduction}) can appear.

Extending the results of  \cite{mikosch2013,mikosch2016} to the sample-path level poses several phenomenological and technical challenges.
So far, all results cited center around the phenomenon where rare events are caused by a single big jump. However, not all rare events are caused by this relatively simple scenario, for early examples see \cite{fosskorshunov2012, zwart2004}. In a recent paper, 
\cite{rheeblanchetzwart2016} provides sample-path large deviations results for L\'evy processes and random walks with regularly varying increments, which deal with a general class of rare events that can especially be caused by multiple jumps. For further examples see \cite{chen2019}. However, the case studied here is considerably harder, as big jumps occur by a condensation phenomenon,
through the concatenation of many small jumps. In particular, a large value of the sample mean is not due to a single large value of the $A_n$ or $B_n$ but to large values of 
the products $A_1\cdots A_n$, see also \cite{buraczewski2013, collamore2007}. When studying sample-path large deviations, this phenomenon poses nontrivial technical requirements. In particular, an appropriate topology needs to be considered.

Our approach to overcome these challenges is as follows. 
We first proceed to identify a sequence of regeneration times $r_n$, $n\geq 1$ (see \cite{athreya1978}) and split the Markov chain into i.i.d.\ cycles. 
By aggregating the trajectory of $\bar X_n$ over regeneration cycles, we obtain a regenerative process with i.i.d.\ jump distributions and $r_n$, $n\geq1$ as renewals. 
Under a set of assumptions originating from \cite{goldie1991} and \cite{kesten1973}, we establish our first result, Theorem \ref{thm: tail estimates for the area under first return time/regeneration cycle}, which is that the ``area'' under a typical regeneration cycle, denoted by $\mathfrak R$ (see \eqref{eq: area under first return time} below), has an asymptotic power law. 
To be precise, we have 
\begin{equation}\label{eq: tail estimates introduction}
\P(\mathfrak R > x) \sim C_+ x^{-\alpha}
\quad\mbox{and}\quad
\P(\mathfrak R < -x) \sim C_- x^{-\alpha},
\quad\mbox{as}\ x\to\infty,
\end{equation}
for some constants $C_-$, $C_+$. 
This is related to a result of \cite{collamore2007} for the case where $X_n\geq 0$. Our argument is different, developed in a two-sided setting, and can be extended to more general recursions, cf.\ \cite{chenthesis}. 

Using the tail estimates (\ref{eq: tail estimates introduction}), we present in Sections~\ref{SECTION: ONE-SIDED LARGE DEVIATIONS} and \ref{SECTION: TWO-SIDED LARGE DEVIATIONS} large deviations results for $\bar X_n$ as in \eqref{eq: X_n bar introduction}, which constitutes the second major step in our approach. 
We achieve this by introducing a new asymptotic equivalence concept (see Lemma~\ref{lem: asymptotic equivalence} below), which, together with the decomposition in cycles, allows us to build a bridge between our problem and the one studied \cite{rheeblanchetzwart2016}. 
In the latter paper, the Skorokhod $J_1$ topology is used. 
However, showing that the residual process (i.e.\ the contribution of the cycle going on at the endpoint of our interval) is negligible in its contribution to $\P(\bar X_n\in E)$ is not straightforward, especially when the increments of $\bar X_n$ are dependent as in the current setting. 
To overcome this, we switch to a slightly weaker topology, namely the $M_1'$-topology on ${\mathbb D}[0,1]$ (as defined in \cite{Mihail2020}, see also Section~\ref{SECTION: ONE-SIDED LARGE DEVIATIONS} below), and derive asymptotic estimates of events involved with the ``area'' under the last ongoing cycle. This choice of topology is crucial as it allows many light-tailed jumps, occurring within a cycle, to merge into a single heavy-tailed jump.

Our main sample-path large deviations results are presented in Section~\ref{section: main results}. 
For the case where $B_n$ as in \eqref{eq: affine transformation} is nonnegative, our result establishes that 
\begin{equation}\label{eq: main result introduction}
C_{\mathcal J^*}(E^\circ) 
\leq 
\liminf_{n\to\infty}\frac{\P(\bar X_n\in E)}{(n\P(\mathfrak R>n))^{\mathcal J^*}}
\leq 
\limsup_{n\to\infty}\frac{\P(\bar X_n\in E)}{(n\P(\mathfrak R>n))^{\mathcal J^*}}
\leq 
C_{\mathcal J^*}(E^-).
\end{equation}
Precise details can be found in Section~\ref{SECTION: ONE-SIDED LARGE DEVIATIONS} below. 
At this moment, we just mention that $C_j$ is a measure on ${\mathbb D}[0,1]$ for each $j$, 
and $\mathcal J^*$ denotes the minimum number of jumps that are required for a nondecreasing, piecewise linear function with drift $\E B_1/(1-\E A_1)$ to be in the set $E$. 
In Section~\ref{SECTION: TWO-SIDED LARGE DEVIATIONS} we develop a two-sided version of this result. 

While we restrict to the case of affine recursions in (\ref{eq: affine transformation}), 
the methods developed in this paper can be extended to more general recursions of the form $X_{n+1} = f_{n+1}(X_n)$, in which $f_n(z)/z\rightarrow A_n$ as $z\rightarrow\infty$; we refer to \cite{chenthesis} for details. On the other hand, our methods require the assumption $A_n\geq 0$ in (\ref{eq: affine transformation}). If this assumption no longer holds, it is possible
to have big jumps of opposite sign in the same regeneration cycle, which requires a topology weaker than $M_1'$. Functional central limit theorems allowing $A_n$ to have both signs were recently derived in \cite{Basrak2018} using the $M_2$ topology.

This paper is organized as follows. 
In Section~\ref{SECTION: PRELIMINARIES}, we introduce some useful tools for future purposes. 
We present our main results in Section~\ref{section: main results}. 
Sections~\ref{section:proof_part1}--\ref{section: proofs of section 3.2 and 3.3} are devoted to the proofs.

\section{Preliminaries}\label{SECTION: PRELIMINARIES}
In this section, we recall and establish some preliminary results. All the proofs are deferred to Section~\ref{section:proof_part1}. 
We start by introducing a regularity condition. 

\begin{assumption}\label{ass: regularity condition AR(1) process}
 The  random vector $(A_1, B_1)$ satisfies 
\begin{enumerate}
\item $A_1 \geq 0$ a.s.\ and the law of $\log A_1$ 
conditioned on $\{A_1>0\}$ 
is nonarithmetic. 
\item There exists an $\notationdef{nota-alpha}{\alpha}\in(1,\infty)$ such that $\E A_1^\alpha = 1$,  $\E A_1^\alpha \log ^+ A_1 < \infty$ (where $\notationdef{nota-log-plus}{\log^+x} = \max\{\log x,0\}$), and $\E |B_1|^{\alpha+\epsilon}<\infty$ for some $\epsilon > 0$. 
\item $\P (A_1x+B_1=x)<1$ for every $x\in\mathbb R$.
\end{enumerate}
\end{assumption}

The  conditions in Assumption~\ref{ass: regularity condition AR(1) process} imply that $\E \log A_1<0$ and $\E \log^+ |B_1|<\infty$, and hence (see e.g.\ Theorem 2.1.3 of \cite{buraczewskidamekmikosch2016}), the Markov chain has a unique stationary distribution, denoted by $\pi$. Moreover, \cite{goldie1991} and \cite{kesten1973} showed there exist constants 
 $c_+$, $c_-$ satisfying $c_+ + c_- \in (0,\infty)$
such that
\begin{equation}  
\label{eq kestengoldie}
\pi(x,\infty)\sim c_+ x^{-\alpha} \quad\mbox{and}\quad \pi(-\infty,-x)\sim c_- x^{-\alpha},
\qquad\mbox{as}\ x\to\infty.
\end{equation}
A natural question is whether $c_+>0$ and/or $c_->0$ in our setting. In \cite{guivarch}, 
sufficient conditions for $c_+>0$ and/or $c_->0$ are developed using algebraic methods.

\subsection{Background from Markov chain theory}
We review some concepts from Markov chain theory. 
We begin by introducing two conditions
on general Markov chains. 
A Markov chain on some general state space $(\mathbb S,\mathcal S)$ with transition kernel $P$ satisfies a drift condition \eqref{condition: drift} if 
\[
\int_{\mathbb S} h(y) P(x,dy) \leq \kappa_1 h(x) + \kappa_2\1_{\mathcal C}(x), 
\tag{$\mathcal D$}\label{condition: drift}
\]
for some $\kappa_1\in(0,1)$ and $\kappa_2>0$,
where $h$ takes values in $[1,\infty)$, and $\mathcal C$ is a Borel subset of $\mathbb R$. 
Moreover, we say that a $\phi$-irreducible Markov chain on $(\mathbb S,\mathcal S)$ with transition kernel $P$ satisfies the minorization condition \eqref{condition: minorization} if 
\[
\theta \1_{\mathcal C_0}(x) \phi(E\cap E_0) \leq P(x,E),\qquad x\in\mathbb S, E\in\mathcal S,
\tag{$\mathcal M$}\label{condition: minorization}
\]
for some set $E_0\subseteq\mathbb S$, some set $\mathcal C_0$ with $\phi(\mathcal C_0)>0$, some constant $\theta>0$, and some probability measure $\phi$ on $(\mathbb S,\mathcal S)$. 

\begin{remark}\label{rmk: regeneration scheme}
If the minorization condition \eqref{condition: minorization} holds, then there exists a sequence of strictly increasing finite random times $r_n$, $n \geq 1$ such that $\{X_n\}_{n\geq0}$ regenerates at each $r_n$ w.r.t.\ $\phi$, that is, $\P(X_{r_i} \in E) = \phi(E\cap E_0)$ for each $i$.
In particular, $X_n$ attempts (independently of everything else) to regenerate each time it enters $\mathcal C_0$, and such attemps are successful w.p.\ $\theta$.

\end{remark}


Recall that we say that the Markov chain $\{X_n\}_{n\geq 0}$ is geometrically ergodic if there exists some number $\rho_0\in(0,1)$ such that 
$$
\| P^n(x,\cdot) - P_0(\cdot)\|_{TV} = o(\rho_0^n)
$$
as $n\to \infty$.

\begin{result}[Lemma 2.2.3, Proposition 2.2.4, Theorem 2.4.4 of \cite{buraczewskidamekmikosch2016}]
\label{res: path properties of AR(1) processes}
\linksinthmwopf{res: path properties of AR(1) processes}
Let $\{X_n\}_{n\geq0}$ be such that $X_{n+1} = A_{n+1} X_n + B_{n+1}$. 
Supposing that Assumption~\ref{ass: regularity condition AR(1) process} holds, we have that: 
\begin{enumerate}
\item For any given $\delta\in(0,\alpha)$, $\{X_n\}_{n\geq0}$ satisfies the drift condition \eqref{condition: drift} with $h(x) = 1+ |x|^\delta$ and $\mathcal C = [-M,M]$ for some constant $M\geq0$. 
\item $\{X_n\}_{n\geq0}$ is $\pi$-irreducible. 
\item $\{X_n\}_{n\geq0}$ is geometrically ergodic. 
\end{enumerate}
\end{result}


The regeneration scheme described in Remark~\ref{rmk: regeneration scheme} plays an important role in our analysis. Our next assumption guarantees the existence of the regeneration times.

\begin{assumption}\label{ass: minorization AR(1) process}
Condition~\eqref{condition: minorization} is satisfied with $\mathcal C_0 = [-d,d]$ for some \notationdef{nota-d}{$d>0$} such  that $[-d,d] \cap \mbox{supp}(\pi) \neq \emptyset$. 
\end{assumption}
For the rest of the paper, we use \notationdef{nota-r-n}{$r_n$} to denote the $n$\textsuperscript{th} regeneration time w.r.t.\ this particular $\mathcal C_0$ and $\phi$ in Assumption~\ref{ass: minorization AR(1) process}. Moreover, we assume $r_0=0$, so that $X_0 \in \mathcal C_0$.
For completeness we mention some sufficient conditions for Assumption \ref{ass: minorization AR(1) process} to hold in terms of the joint distribution of $(A_1, B_1)$, which is a minor extension of Lemma~2.2.3 of \cite{buraczewskidamekmikosch2016}. 
Let $\notationdef{nota-scr-B-r}{\mathscr B_r(x)} = \{ x' \colon |x-x'|<r \}$ for $x\in\mathbb R$ and $r>0$. 
\begin{proposition}
\label{prop: conditions for minorization}
\linksinthm{prop: conditions for minorization}
Assume that one of the following conditions hold. 
\begin{enumerate}
\item Let $B_1 \geq b$ a.s.\ for some $b>0$.  
Moreover, there exist intervals $I_1 = (a_1,a_2)\subseteq \R_+$, $I_2 = (b_0-\delta,b_0+\delta)$ for some $a_1<a_2$, $b_0$, $\delta>0$, a $\sigma$-finite measure $\nu_0$ with $b_0$ in the support of $\nu_0$, and a constant $c_0>0$ such that for any Borel sets $D_1$, $D_2\subseteq \mathbb R$,
\[
\P ((A_1,B_1)\in (D_1\times D_2)) \geq c_0 |D_1\cap I_1| \nu_0(D_2\cap I_2),
\]
where $|\cdot|$ denotes the Lebesgue measure on $\mathbb R$.
\item 
There exist intervals $I_1 = (a_0-\delta,a_0+\delta)\subseteq \R_+$, $I_2 = (b_1,b_2)$ for some $a_0$, $b_1<b_2$, $\delta>0$, a $\sigma$-finite measure $\nu_0$ with $a_0$ in the support of $\nu_0$, and a constant $c_0>0$ such that for any Borel sets $D_1$, $D_2\subseteq \mathbb R$,
\begin{equation}\label{eq:sufficient condition for minorization}
\P ((A_1,B_1)\in (D_1\times D_2)) \geq c_0 \nu_0(D_1\cap I_1) | D_2\cap I_2 |.
\end{equation}
\end{enumerate}
Then, for any $x_0 \in \R$, there exists $\epsilon = \epsilon(x_0)$, $\theta>0$, and an open interval $E_0$
such that 
\begin{equation}\label{eq: minorization}
\theta | E\cap E_0 | \leq P(x,E),\qquad x\in\mathscr B_\epsilon(x_0), E\in\mathcal B(\mathbb R). 
\end{equation}
\end{proposition}

Our next result implies the geometric decay of $\P (r_1>k)$ as $k\to\infty$. 

\begin{lemma}
\label{lem: exponential decay of regeneration time}
\linksinthm{lem: exponential decay of regeneration time}
Suppose that Assumptions~\ref{ass: regularity condition AR(1) process} and \ref{ass: minorization AR(1) process} hold. 
Let $\{r_n\}_{n\geq0}$ be the sequence of regeneration times associated with $\mathcal C_0$. 
Let $E_1$ be a bounded set. 
There exists $t>1$ such that 
\[
\sup_{x\in E_1} \E [ t^{r_1} \,|\, X_0=x ] < \infty.
\]
\end{lemma}

\subsection{A useful change of measure}
Another helpful tool in our analysis is the so-called $\alpha$-shifted change of measure (see e.g.\ \cite{collamore2016,collamore2013}). 
Let \notationdef{nota-nu}{$\nu$} denote the distribution of $(\log A_n,B_n)$ and define the $\alpha$-shifted measure $\nu^\alpha$ by
\[
\notationdef{nota-nu-alpha}{\nu^\alpha(E)} = \int_E e^{\alpha x} d\nu(x,y),\qquad E\in\mathfrak B(\mathbb R^2).
\]
Let $\mathcal L(\log A_n,B_n)$ denote the law of $(\log A_n,B_n)$. 
For a stopping time $T$, Let \notationdef{nota-scr-D-T-alpha}{$\mathscr D_T^\alpha$} be the dual change of measure such that, under $\mathscr D_T^\alpha$, 
\begin{equation}\label{eq: dual change of measure}
\mathcal L(\log A_n,B_n) 
= 
\begin{cases}
\nu^\alpha,\qquad&\mbox{for}\ n\leq T,\\[5pt]
\nu,             &\mbox{for}\ n> T.
\end{cases}
\end{equation}
Let \notationdef{nota-E-P-alpha-D}{$\P^\alpha$, $\P^{\mathscr D_T^\alpha}$, $\E^\alpha$ and $\E^{\mathscr D_T^\alpha}$} denote expectation and probability w.r.t.\ the $\alpha$-shifted measure $\nu^\alpha$ and the dual change of measure $\mathscr D_T^\alpha$, respectively. 
Defining 
\begin{equation}\label{eq: multiplicative random walk}
\notationdef{nota-S-n}{S_n} = \sum_{i=1}^n \log A_i, 
\end{equation}
we have the following result. 
 
\begin{result}[Lemma 5.3 of \cite{collamore2013}]\label{res: connection between original and dual chang of measures}
Let $T$ and $\tau$ be stopping times w.r.t.\ $\{X_n\}_{n\geq0}$, let $g\colon \mathbb R^\infty \to [0,\infty]$ be a deterministic function, and let $g_n$ denote its projection onto the first $n+1$ coordinates, i.e., $g_n(x_0,\ldots,x_n) = g(x_0,\ldots,x_n,0,0,\ldots)$. 
Then
\begin{align*}
\E [g_{\tau-1}(X_0,\ldots,X_{\tau-1})] 
&= 
\E^{\mathscr D_{T}^\alpha} \left[ g_{\tau-1}(X_0,\ldots,X_{\tau-1}) e^{-\alpha S_{T}} \1_{\{T<\tau\}} \right]\\ 
&\hspace{12.5pt}+ 
\E^{\mathscr D_{T}^\alpha} \left[ g_{\tau-1}(X_0,\ldots,X_{\tau-1}) e^{-\alpha S_\tau} \1_{\{T\geq\tau\}} \right]. 
\end{align*}
\end{result}
\begin{remark}\label{a-remark-on-dual-change-of-measure}
\linksinthmwopf{a-remark-on-dual-change-of-measure}
Note that by the same argument, if a random variable $R$ is measurable w.r.t.\ the stopped $\sigma$-algebra $\mathcal F_T$, then
$$\E [R] = \E^{\mathscr D_T^\alpha} [Re^{-\alpha S_T}] = \E^\alpha [R e^{-\alpha S_T}].$$
\end{remark}
Our analysis relies on the fact that the Markov chain $X_n$ is closely related to a multiplicative random walk, that is, 
\[
X_{n+1} \approx A_{n+1} X_n,\quad\mbox{for large}\ n.
\]
Roughly speaking, the process $X_n$ resembles a perturbation of a multiplicative random walk, in an asymptotic sense (for details see \cite{collamore2016,collamore2013}). 
Hence, it is natural to consider the ``discrepancy'' process between $X_n$ and $\prod_{i=1}^n A_i$, which is defined as
\begin{equation}\label{eq: discrepancy process}
\notationdef{nota-Z-n}{Z_n} = X_n e^{-S_n} = X_0 + \sum_{k=1}^n B_k e^{-S_k},\quad n \geq 0,
\end{equation} 
where $S_n$ is as in \eqref{eq: multiplicative random walk}. 
Under the $\alpha$-shifted measure, we have $\E^\alpha \log A_1 = \E A_1^\alpha \log A_1 > 0$ by Assumption~\ref{ass: regularity condition AR(1) process} and Theorem 2.4.4 of \cite{buraczewskidamekmikosch2016}.
Consequently, we have the following result.

\begin{lemma}\label{lem: asymptotics of X_n and Z_n}
Let Assumption~\ref{ass: regularity condition AR(1) process} hold. 
Under $\P^\alpha$, 
\begin{enumerate}
\item $|X_n| \uparrow \infty$ a.s.\ as $n\to\infty$. 
\item $Z_n\toas Z$ as $n\to\infty$, where $\notationdef{nota-Z}{Z} = X_0 + \sum_{k=1}^\infty B_k e^{-S_k}$.
\end{enumerate} 
\end{lemma}

\subsection{\texorpdfstring{$\mathbb M$}{M}-convergence}\label{section: M-convergence}
We briefly review the notion of $\mathbb M$-convergence \cite{lindskogresnickroy2014, rheeblanchetzwart2016}, and introduce a novel asymptotic equivalence concept.
Let $(\mathbb S, d)$ be a complete separable metric space, and $\mathscr S$ be the Borel $\sigma$-algebra on $\mathbb S$. 
Given a closed subset $\mathbb C$ of $\mathbb S$, let $\mathbb S\setminus\mathbb C$ be equipped with the relative topology as a subspace of $\mathbb S$, and consider the associated sub $\sigma$-algebra $\mathscr S_{\mathbb S\setminus\mathbb C} = \{ E\colon E\subseteq \mathbb S\setminus\mathbb C, E\in\mathscr S \}$ on it. 
Define $\notationdef{nota-mathbb-C-r}{\mathbb C^r} = \{ x\in\mathbb S\colon d(x,\mathbb C)<r \}$ for $r>0$, and let $\mathbb M(\mathbb S\setminus \mathbb C)$ be the class of measures defined on $\mathscr S_{\mathbb S\setminus\mathbb C}$ whose restrictions to $\mathbb S \setminus \mathbb C^r$ are finite for all $r > 0$. 
Topologize $\mathbb M(\mathbb S\setminus \mathbb C)$ with a sub-basis $\{ \{\nu\in \mathbb M(\mathbb S\setminus \mathbb C)\colon\nu(f)\in G\} \colon f\in \mathcal C_{\mathbb S\setminus\mathbb C}, G\ \mbox{open in}\ \R_+ \}$, 
where $\mathcal C_{\mathbb S\setminus\mathbb C}$ is the set of real-valued, nonnegative, bounded, continuous functions whose support is bounded away from $\mathbb C$ (i.e., $f(\mathbb C^r) = \{0\}$ for some $r>0$). 
A sequence of measures $\nu_n\in \mathbb M(\mathbb S\setminus \mathbb C)$ converges to $\nu\in \mathbb M(\mathbb S\setminus \mathbb C)$ if $\nu_n(f) \to \nu(f)$ for each $f\in \mathcal C_{\mathbb S\setminus\mathbb C}$. 
We say that a set $E_1 \subseteq \mathbb S$ is bounded away from another set $E_2 \subseteq \mathbb S$ if $\inf_{x\in E_1,y\in E_2} d(x,y) > 0$. 
The following characterization of $\mathbb M$-convergence can be considered as a generalization of the classical notion of weak convergence of measures, see e.g.\ \cite{billingsley2013}. 

\begin{result}[Theorem~2.1 of \cite{lindskogresnickroy2014}]
Let $\nu$, $\nu_n\in \mathbb M(\mathbb S\setminus \mathbb C)$. 
We have $\nu_n\to\nu$ in $\mathbb M(\mathbb S\setminus \mathbb C)$ as $n\to\infty$ if and only if 
\[
\varlimsup_{n\to\infty} \nu_n(F) \leq \nu(F)
\]
for all closed $F\in \mathscr S_{\mathbb S\setminus\mathbb C}$ bounded away from $\mathbb C$ and 
\[
\varliminf_{n\to\infty} \nu_n(G) \geq \nu(G)
\]
for all open $G\in \mathscr S_{\mathbb S\setminus\mathbb C}$ bounded away from $\mathbb C$.
\end{result}

We now introduce a new notion of equivalence between two families of random objects, which will prove to be useful in Section~\ref{section: proofs of section 3.2 and 3.3}. 
Let $\notationdef{nota-F-delta}{F_\delta} = \{ x\in\mathbb S\colon d(x,F)\leq \delta \}$ and $\notationdef{nota-G--delta}{G^{-\delta}} = ((G^c)_\delta)^c$. Note that when it comes to the fattening and shaving of sets, we denote open sets with a superscript and closed sets with a subscript.

\begin{definition}\label{def: asymptotic equivalence}
Suppose that $X_n$ and $Y_n$ are random elements taking values in a complete separable metric space $(\mathbb S, d)$. 
$Y_n$ is said to be asymptotically equivalent to $X_n$ with respect to $\epsilon_n$ and $\mathbb C$, if, for each $\delta>0$ and $\gamma > 0$, 
\begin{align*}
&\varlimsup_{n\to\infty} \epsilon_n^{-1} \P(X_n\in (\mathbb S\setminus\mathbb C)^{-\gamma}, d(X_n,Y_n)\geq \delta)\\
&=
\varlimsup_{n\to\infty} \epsilon_n^{-1} \P(Y_n\in (\mathbb S\setminus\mathbb C)^{-\gamma}, d(X_n,Y_n)\geq \delta)
=0. 
\end{align*}
\end{definition}

\begin{remark}
\label{remark: asymptotic equivalence}
Note that the asymptotic equivalence w.r.t. $\mathbb C$ implies the asymptotic equivalence w.r.t. $\mathbb C'$ if $\mathbb C \subseteq \mathbb C'$. 
In view of this, the strongest notion of asymptotic equivalence w.r.t. a given sequence $\epsilon_n$ is the one w.r.t. an empty set. 
In this case, the conditions for the asymptotic equivalence reduce to a simple condition: $\P ( d(X_n , Y_n ) \geq \delta) = o(\epsilon_n)$ for any $\delta > 0$. 
That special case of asymptotic equivalence has been introduced and applied in \cite{rheeblanchetzwart2016}. 
In our context, this simple condition suffices for the case of $B_1\geq 0$ in Section~\ref{SECTION: ONE-SIDED LARGE DEVIATIONS}; however, we have to work with the case that $\mathbb C$ is not an empty set when we deal with general $B_1$ in Section~\ref{SECTION: TWO-SIDED LARGE DEVIATIONS}. 
\end{remark}

The usefulness of this notion of equivalence comes from the following result. 

\begin{lemma}\label{lem: asymptotic equivalence}
Suppose that $\epsilon_n^{-1} \P(X_n\in\cdot) \to \nu(\cdot)$ in $\mathbb M(\mathbb S\setminus \mathbb C)$ for some sequence $\epsilon_n$ and a closed set $\mathbb C$. 
If $Y_n$ is asymptotically equivalent to $X_n$ with respect to $\epsilon_n$ and $\mathbb C$, then the law of $Y_n$ has the same normalized limit, i.e., $\epsilon_n^{-1} \P(Y_n\in\cdot) \to \nu(\cdot)$ in $\mathbb M(\mathbb S\setminus \mathbb C)$. 
\end{lemma}

\section{Main results}\label{section: main results}
This section is organized as follows. 
In Section~\ref{SECTION: TAIL ESTIMATES FOR THE AREA UNDER FIRST RETURN TIME/REGENERATION CYCYLE}, we analyze the tail estimates of the area under the first return time and regeneration cycle, which are needed to derive the sample-path large deviations of $\bar X_n$. In Section~\ref{SECTION: ONE-SIDED LARGE DEVIATIONS} we derive such results in the case  where $B_1\geq 0$. The two-sided case is more involved and is treated in Section~\ref{SECTION: TWO-SIDED LARGE DEVIATIONS}.

\subsection{Tail estimates on the area under the first return time/regeneration cycle}\label{SECTION: TAIL ESTIMATES FOR THE AREA UNDER FIRST RETURN TIME/REGENERATION CYCYLE}
Let 
\begin{equation}\label{eq: first return time}
\notationdef{nota-tau-d}{\tau_d} = \inf\{n\geq 1\colon |X_n| \leq d\}
\end{equation}
denote the first return time of $X_n$ to the set $[-d,d]$, where $d$ is such that $[-d,d] \cap \mbox{supp}(\pi) \neq \emptyset$. 
Recall that $\{r_n\}_{n\geq0}$ is the sequence of regeneration times of $\{X_n\}_{n\geq0}$. 
We denote the area under the first return time and the regeneration cycle by
\begin{equation}\label{eq: area under first return time}
\notationdef{nota-frak-B}{\mathfrak B} = \sum_{n=0}^{\tau_d-1} X_n
\quad\mbox{and}\quad
\notationdef{nota-frak-R}{\mathfrak R} = \sum_{n=0}^{r_1-1} X_n,
\end{equation}
respectively. 
Recall that $Z = X_0 + \sum_{k=1}^\infty B_k e^{-S_k}$. 
Finally, note that considering $B_i \equiv 1$, there exists a constant \notationdef{nota-C-infty}{$C_\infty$} given in \cite{goldie1991}
that satisfies
\begin{equation}\label{eq: tail asymptotics on the stationary distribution of AR(1) process}
\P \left( \sum_{k=0}^\infty e^{S_k} > u \right) \sim C_\infty u^{-\alpha}.
\end{equation}

\begin{theorem}\label{thm: tail estimates for the area under first return time/regeneration cycle}
Suppose that Assumptions~\ref{ass: regularity condition AR(1) process} and \ref{ass: minorization AR(1) process} hold.
\begin{enumerate}
\item \label{part 1 of thm: tail estimates for the area under first return time/regeneration cycle}
\linksinthm{part 1 of thm: tail estimates for the area under first return time/regeneration cycle}
We have
\begin{align*}
&\lim_{u\to\infty} u^\alpha \P(\mathfrak B>u) = C_\infty \E^\alpha [(Z^+)^\alpha \1_{\{\tau_d=\infty\}}]\\
\mbox{and}\quad
&\lim_{u\to\infty} u^\alpha \P(\mathfrak B<-u) = C_\infty \E^\alpha [(Z^-)^\alpha \1_{\{\tau_d=\infty\}}].
\end{align*}
\item 
\linksinthm{part 2 of thm: tail estimates for the area under first return time/regeneration cycle}
In addition 
\begin{align*}
\lim_{u\to\infty} u^\alpha \P(\mathfrak R>u) = C_+
\quad\mbox{and}\quad
\lim_{u\to\infty} u^\alpha \P(\mathfrak R<-u) = C_-,
\end{align*}
where 
$
\notationdef{nota-C-+}{C_+} = C_\infty \E^\alpha [(Z^+)^\alpha \1_{\{r_1=\infty\}}]
$
and
$
\notationdef{nota-C--}{C_-} = C_\infty \E^\alpha [(Z^-)^\alpha \1_{\{r_1=\infty\}}]
$.
\end{enumerate}
\end{theorem}

Like in the classical estimates \eqref{eq kestengoldie}, it is natural to ask when $C_+, C_- \in (0,\infty)$. 
A proof of finiteness of $\E^\alpha [|Z|^\alpha]$ is obtained as byproduct of the proof of Lemma 
\ref{lem: almost sure convergence of G+ and G-}. $C_\infty \in (0,\infty)$ by specializing \eqref{eq kestengoldie} to the case $A_1\geq 0$ and $B_1\equiv 1$. 
If  $B_1$ is nonnegative and $\P(A_1=B_1=0)=0$, then $Z>0$ $\P^\alpha$-a.s.
Since also $\P^\alpha( r_1=\infty) > 0$,  $C_+>0$. 

When $B_1$ can take both signs, the situation is much more delicate and we sketch how one can deal with this issue. One way is to derive sufficient conditions for the support of $Z$ under $\P^\alpha$ to be the entire real line, from which strict positivity of both $C_+$ and $C_-$ can be inferred. Such a sufficient condition can be derived from a careful inspection of the proof of Theorem 2.5.5 (1) of \cite{buraczewskidamekmikosch2016}  (which is a result due to \cite{guivarch}). For example, if the support of $(A_1,B_1)$ includes points 
$(a,b), (a_1, b_1), (a_2, b_2)$ such that $a<1, a_1,a_2>1$ and $b_1/(1-a_1)<b/(1-a)<b_2/(1-a_2)$ the support of $Z$ is the whole real line. 

\subsection{One-sided large deviations}\label{SECTION: ONE-SIDED LARGE DEVIATIONS}
We first consider the case where $B_1$ is nonnegative.
To deal with the dependence structure of the Markov chain within the regeneration cycle, we consider in this section the space $\mathbb D = \mathbb D [0,1]$, consisting of real-valued functions with domain $[0,1]$ which are right-continuous with left limits. We endow $\mathbb D$ with the  $M_1'$ topology. 
To describe this topology in detail, let $\xi \in \mathbb D$ and define the extended completed graph $\Gamma_\xi'$ of $\xi$ by 
\[
\notationdef{nota-Gamma-xi-prime}{\Gamma_\xi'} = \{(x,t)\in\mathbb R\times[0,1] \colon x \in [\xi(t^-)\wedge\xi(t),\xi(t^-)\vee\xi(t)] \},
\]
where $\xi(0^-)=0$. 
Define an order on the graph $\Gamma_\xi'$ by saying that $(x_1,t_1)\leq (x_2,t_2)$ if either (i) $t_1<t_2$ or (ii) $t_1=t_2$ and $|\xi(t_1^-)-x_1| \leq |\xi(t_2^-)-x_2|$. 
We say that a mapping $(u,s):[0,1]\to \Gamma_\xi'$ is a parametric representation of $\xi$ if $r \mapsto(u(r),s(r))$ is continuous and nondecreasing. 
Let \notationdef{nota-Pi-prime}{$\Pi'(\xi)$} be the set of all parametric representations of $\xi\in \mathbb D$.
For any $\xi_1,\xi_2\in\mathbb D$, the $M_1'$ metric is defined by
\[
\notationdef{nota-d-M-1-prime}{d_{M_1'}}(\xi_1,\xi_2) 
= 
\inf_{\substack{(u_i,s_i)\in\Pi'(\xi_i)\\i\in\{1,2\}}} \| u_1-u_2 \|_\infty \vee \| s_1-s_2 \|_\infty.
\]
For the rest of the paper, we consider the topology w.r.t.\ this metric, unless specified otherwise. 

For the one-sided large deviations result, we need the following elements. 
We say that a function $\xi\in\mathbb D$ is \textit{piecewise constant}, if there exist finitely many time points $t_i$ such that $0=t_0<t_1<\cdots< t_m=1$ and $\xi$ is constant on the intervals $[t_{i-1},t_i)$ for all $1\leq i\leq m$. 
For $\xi\in\mathbb D$, define the set of discontinuities of $\xi$ by
\begin{equation}\label{eq: discontinuities}
\notationdef{nota-disc}{\disc(\xi)} = \{ t\in[0,1] \colon \xi(t) \neq \xi(t^-) \}, 
\end{equation}
where $\xi(0^-)=0$. 
For each integer $j$, define
\[
\notationdef{nota-D-underbar-lessthan-j}{\underline{\mathbb D}_{< j}} = \{ \xi\in\mathbb D \colon \xi\ \mbox{piecewise constant and nondecreasing},\ |\mbox{Disc}(\xi)| < j 
\}.
\]
For $z\in\R$ and each integer $j$, define
\begin{equation}\label{eq: D_<j^z}
\notationdef{nota-D-underbar-lessthan-j-z}{\underline{\mathbb D}_{< j}^z} = \{\xi\in\mathbb D \colon \xi = z \cdot id + \xi', \xi'\in\underline{\mathbb D}_{< j} \}.
\end{equation} 
For each constant $\gamma>1$, let $\notationdef{nota-nu-gamma}{\nu_\gamma(x,\infty)} = x^{-\gamma}$, and let $\notationdef{nota-nu-gamma-j}{\nu_\gamma^j}$ denote the restriction (to $\notationdef{nota-R-plus-j-downarrow}{\R_+^{j\downarrow}} = \{ x\in\R^j\colon x_1 \geq \cdots \geq x_j>0 \}$) of the $j$-fold product measure of $\nu_\gamma$. 
Let \notationdef{nota-C-0-z}{$C_0^z$} be the Dirac measure concentrated on the linear function $zt$. 
For $j\geq1$, define a sequence of Borel measures $C_j^z \in \mathbb M(\mathbb D\setminus\underline{\mathbb D}_{< j})$ concentrated on $\underline{\mathbb D}_{< j+1}$ as
\begin{equation}\label{eq: C_j^z}
\notationdef{nota-C-j-z}{C_j^z (\,\cdot\,)} = 
\E \left[ \nu_\alpha^j \left\{ x\in(0,\infty)^j\colon z \cdot id + \sum_{i=1}^j x_i\mathbbm 1_{U_i} \in \,\cdot \right\} \right],
\end{equation}
where $\alpha$ is as in Assumption~\ref{ass: regularity condition AR(1) process} and the random variables $U_i$, $i\geq 1$, are independent and uniform distributed on $[0,1]$. 
For $E\subseteq \mathbb D$ and $z\in\mathbb R$, define 
\begin{equation}\label{eq: J_z(E)}
\notationdef{nota-mathcal-J-z-uparrow}{\mathcal J_z^{\uparrow}(E)} = \inf \{ j \colon E \cap \underline{\mathbb D}_{< j+1}^z \neq \emptyset \}.
\end{equation}
Setting \linkdest{nota-mu} $\mu = \int x \pi (dx) = \E B_1/(1-\E A_1)$, we state below the main theorem for the one-sided case.  
Recall $C_+$ defined in Theorem~\ref{thm: tail estimates for the area under first return time/regeneration cycle}.
As kindly pointed out by a referee, 
if $B_1 \geq 0$ a.s., then thanks to Assumption 2.1 $\P(B_1 = 0) < 1$ and $C_+$
must be strictly positive, due to  (2.12) in \cite{goldie1991}.

\begin{theorem}
\label{thm: one-sided large deviations}
\linksinthm{thm: one-sided large deviations}
Suppose that Assumptions~\ref{ass: regularity condition AR(1) process} and \ref{ass: minorization AR(1) process} hold. 
Moreover, let $B_1\geq 0$. 
\begin{enumerate}
\item For each $j\geq 1$, 
\[
n^{j(\alpha-1)} \P(\bar X_n\in\,\cdot\,)\to (C_+ \E r_1)^j C_j^{\mu}(\,\cdot\,),
\]
in $\mathbb M(\mathbb D\setminus\underline{\mathbb D}_{< j}^\mu)$ as $n\to\infty$. 
\item Let $E$ be measurable. 
If $\mathcal J_\mu^{\uparrow}(E)<\infty$ and $E$ is bounded away from $\underline{\mathbb D}_{< \mathcal J_\mu^{\uparrow}(E)}$, then 
\begin{align*}
\varliminf_{n\to\infty}\frac{\P(\bar X_n\in E)}{n^{-\mathcal J_\mu^{\uparrow}(E)(\alpha-1)}}
&\geq 
(C_+ \E r_1)^{\mathcal J_\mu^{\uparrow}(E)} C_{\mathcal J_\mu^{\uparrow}(E)}^{\mu}(E^\circ);\\
\varlimsup_{n\to\infty}\frac{\P(\bar X_n\in E)}{n^{-\mathcal J_\mu^{\uparrow}(E)(\alpha-1)}}
&\leq 
(C_+ \E r_1)^{\mathcal J_\mu^{\uparrow}(E)} C_{\mathcal J_\mu^{\uparrow}(E)}^{\mu}(E^-).
\end{align*}
\end{enumerate}
\end{theorem}

\subsection{Two-sided large deviations}\label{SECTION: TWO-SIDED LARGE DEVIATIONS}
Similarly as in Section~\ref{SECTION: ONE-SIDED LARGE DEVIATIONS}, we need the following elements. 
Define the set of step functions with less than $j$ discontinuities by
\[
\notationdef{nota-D-underbar-ll-j}{\underline{\mathbb D}_{\ll j}}
= 
\{ \xi\in\mathbb D \colon \xi\ \mbox{piecewise constant},\, |\mbox{Disc}(\xi)| < j \}, \quad\mbox{for}\ j\geq0.
\]
For $z\in\R$, define 
\begin{equation}\label{eq: D_<<j^z}
\notationdef{nota-D-underbar-ll-j-z}{\underline{\mathbb D}_{\ll j}^z} = \{\xi\in\mathbb D \colon \xi = z \cdot id + \xi',\  \xi'\in\underline{\mathbb D}_{\ll j} \},\quad\mbox{for}\ j\geq0.
\end{equation} 
Let \notationdef{nota-C-00-z}{$C_{0,0}^z$} be the Dirac measure concentrated on the linear function $zt$. 
For each $(j,k)\in\mathbb Z_+^2\setminus\{(0,0)\}$, define a measure $C_{j,k}^z$ as
\begin{equation}\label{eq: C_j,k^z}
\notationdef{nota-C-j-k-z}{C_{j,k}^z (\,\cdot\,)} = 
\E \left[ \nu_\alpha^{j+k} \left\{ (x,y)\in(0,\infty)^{j+k} \colon z \cdot id + \sum_{i=1}^j x_i\mathbbm 1_{U_i} - \sum_{i=1}^k y_i\mathbbm 1_{V_i} \in \,\cdot \right\} \right],
\end{equation}
where 
 $U_i$, $V_i$ are independent and uniform distributed on $[0,1]$. 
For $E\subseteq \mathbb D$ and $z\in\mathbb R$, define 
\begin{equation}\label{eq: J_z'(E)}
\notationdef{nota-cal-J-z}{\mathcal J_z(E)} = \inf \{ j \colon E \cap \underline{\mathbb D}_{\ll j+1}^z \neq \emptyset \}.
\end{equation}
Recalling $\mu = \E B_1/(1-\E A_1)$, we now state our main theorem for the two-sided case.  

\begin{theorem}\label{thm:two-sided-large-deviations}
\linksinthm{thm:two-sided-large-deviations}
Suppose that Assumptions~\ref{ass: regularity condition AR(1) process}~and~\ref{ass: minorization AR(1) process} hold. 
Let $\E |B_1|^m<\infty$ for every $m\in\mathbb Z_+$. 
Moreover, let $C_+$, $C_-$ be as in Theorem~\ref{thm: tail estimates for the area under first return time/regeneration cycle} such that $C_+C_->0$. 
\begin{enumerate}
\item For each $j\geq 1$,
\[
n^{j(\alpha-1)} \P(\bar X_n\in\,\cdot\,)
\to 
(\E r_1)^j \sum\nolimits_{(l,m)\in I_{= j}} (C_+)^l (C_-)^m C_{l,m}^{\mu}(\,\cdot\,),
\]
in $\mathbb M(\mathbb D\setminus\underline{\mathbb D}_{\ll j}^\mu)$ as $n\to\infty$, 
where $\notationdef{nota-I-equal-j}{I_{= j}} = \{ (l,m)\in\mathbb Z_+^2 \colon l+m=j \}$. 
\item Let $E$ be  measurable. 
If $\mathcal J_\mu(E)<\infty$ and $E$ is bounded away from $\underline{\mathbb D}_{\ll \mathcal J_\mu(E)}$, then 
\begin{align*}
\varliminf_{n\to\infty}\frac{\P(\bar X_n\in E)}{n^{-\mathcal J_\mu(E)(\alpha-1)}} 
&\geq 
(\E r_1)^{\mathcal J_\mu(E)} \sum_{(l,m)\in I_{=\mathcal J_\mu(E)}} (C_+)^l (C_-)^m C_{l,m}^{\mu}(E^\circ);\\ 
\varlimsup_{n\to\infty}\frac{\P(\bar X_n\in E)}{n^{-\mathcal J_\mu(E)(\alpha-1)}}
&\leq 
(\E r_1)^{\mathcal J_\mu(E)} \sum_{(l,m)\in I_{=\mathcal J_\mu(E)}} (C_+)^l (C_-)^m C_{l,m}^{\mu}(E^-),
\end{align*}
where the summations are over all $(l,m)$ that belong to the set $I_{= \mathcal J_\mu(E)}$. 
\end{enumerate}
\end{theorem}

\section{Proofs of Section~\ref{SECTION: PRELIMINARIES}}\label{section:proof_part1}


\begin{proof}[Proof of Proposition~\ref{prop: conditions for minorization}]
\linksinpf{prop: conditions for minorization}
Part 1) and Part 2) for the case  $x_0 \neq 0$ are in \cite[page 22]{buraczewskidamekmikosch2016}.
Hence, we focus on showing part 2) for the case $x_0=0$. 

Note that for any Borel set $E$, \eqref{eq:sufficient condition for minorization} implies that 
\begin{align*}
P(x,E) 
&=
\E [ \1_{\{A_1 x + B_1 \in E\}} ]
\geq 
c_0 \int_{a_0-\epsilon}^{a_0+\epsilon} \int_{I_2} \1_{\{ax + b \in E\}}\, db\, \nu_0(da)
 \\&
= 
c_0 \int_{a_0-\epsilon}^{a_0+\epsilon} \int \1_{\{z - ax \in I_2\}} \1_{\{z \in E\}}\, dz\, \nu_0(da). 
\end{align*}
Let 
\[
E_0 = 
\begin{cases}
\big(b_1 + \epsilon(a_0+\epsilon),\ b_2 - \epsilon(a_0 + \epsilon)\big) & \text{if } x_0 = 0
\\
\big(b_1 + (x_0+\epsilon)(a_0+\epsilon),\ b_2 + (x_0- \epsilon)(a_0 - \epsilon)\big) & \text{if } x_0 > 0
\\
\big(b_1 + (x_0+\epsilon)(a_0-\epsilon),\ b_2 + (x_0- \epsilon)(a_0 + \epsilon)\big) & \text{if } x_0 < 0
\end{cases},
\] 
and pick an $\epsilon>0$ sufficiently small so that $E_0$ is nonempty and $\epsilon < |x_0| \wedge a_0$.
Note that if $x \in \mathscr B_\epsilon(x_0)$, $z\in E_0$, and $a\in (a_0-\epsilon, a_0+\epsilon)$, then 
$z\in E_0$ implies 
$z-ax \in I_2$;  
that is, $\1_{\{z\in E_0\}} \leq \1_{\{ z-ax \in I_2\}}$. Therefore, we have that for all $x\in \mathscr B_\epsilon(x_0)$,
\begin{align*}
P(x,E) 
&\geq 
c_0 \int_{a_0 - \epsilon}^{a_0 + \epsilon} \int \1_{\{z \in E_0\}} \1_{\{z \in E\}}\,dz\,\nu_0(da)
\geq
c_0 \nu_0((a_0-\epsilon, a_0 + \epsilon)) | E\cap E_0 |.
\end{align*}
The constant $\theta = c_0 \nu_0((a_0-\epsilon, a_0 + \epsilon))$ is strictly positive since $a_0$ belongs to the support of $\nu_0$. 
%
\end{proof}

\begin{proof}[Proof of Lemma~\ref{lem: exponential decay of regeneration time}]
\linksinpf{lem: exponential decay of regeneration time}
By Theorem 15.2.6 of \cite{meyn2009} and Result~\ref{res: path properties of AR(1) processes}, any bounded set is $h$-geometrically regular with $h(x) = |x|^\delta+1$, $\delta\in(0,\alpha)$. 
Thus, from the definition of $h$-geometrical regularity (cf.\ page 373 of \cite{meyn2009}), there exists $t>1$ such that \\
$\sup\nolimits_{x\in E_1} \E [ \sum\nolimits_{k=0}^{\tau_{d}-1} h(X_k) t^k \,|\, X_0=x ] < \infty$ 
and 
$\sup\nolimits_{x\in \mathcal C_0} \E [ \sum\nolimits_{k=0}^{\tau_{d}-1} h(X_k) t^k \,|\, X_0=x ] < \infty
$.
Since $h\geq1$,
\begin{equation}\label{eq: proposition 2.1 eq01}
\chi_0(t) \triangleq \sup_{x\in E_1} \E [t^{\tau_{d}} \,|\, X_0=x ]
<
\sup_{x\in E_1} \E \left[\left. \sum\nolimits_{k=0}^{\tau_{d}-1} h(X_k) t^k \,\right|\, X_0=x \right]
<
\infty.
\end{equation}
Likewise,
$
\chi_1(t) \triangleq \sup\nolimits_{x\in \mathcal C_0} \E [t^{\tau_{d}} \,|\, X_0=x ] 
<\infty
$.
Note that for any $s\in(1,t)$, by Jensen's inequality, we get
\begin{align*}
\chi_1(s) &= \sup\nolimits_{x\in \mathcal C_0} \E [s^{\tau_{d}} \,|\, X_0=x ] 
=\sup\nolimits_{x\in \mathcal C_0} \E [t^{\frac{\log s}{\log t}\tau_{d}} \,|\, X_0=x ] 
\\
&
\leq \sup\nolimits_{x\in \mathcal C_0} \E [t^{\tau_{d}} \,|\, X_0=x ]^{\frac{\log s}{\log t}} = \chi_1(t)^{\frac{\log s}{\log t}}\to 1
\end{align*}
as $s\to 1$. 
Now let $t>1$ be sufficiently close to 1 so that $(1-\theta) \chi_1(t)<1$.
From the regeneration scheme as described in Remark~\ref{rmk: regeneration scheme}, we obtain
\begin{equation}\label{eq: proposition 2.1 eq03}
\sup_{x\in E_1} \E [ t^{r_1} \,|\, X_0=x ] \leq \chi_0(t) \left( \theta + \sum_{n=1}^\infty \theta(1-\theta)^n (\chi_1(t))^n \right)<\infty.
\end{equation}
\end{proof}

\begin{proof}[Proof of Lemma~\ref{lem: asymptotics of X_n and Z_n}]
The second statement follows from Theorem 2.1.3 of \cite{buraczewskidamekmikosch2016} since  the random walk $-S_n$ has a negative drift under $\P^\alpha$ and $\E^\alpha [\log B_1] <\infty.$ To prove the first statement, we begin by using a similar argument, invoking
 Assumption~\ref{ass: regularity condition AR(1) process}, to conclude that the random variable $\sum_{i=1}^\infty |B_i| e^{-S_i}$ is a.s. finite. Consequently, we can lower bound
\[
|X_n| 
= 
e^{S_n} \left| X_0 + \sum_{i=1}^n B_i e^{-S_i} \right| 
\geq 
e^{S_n} \left( | X_0 | - \sum_{i=1}^n |B_i| e^{-S_i} \right)
\geq 
e^{S_n} \left( | X_0 | - \sum_{i=1}^\infty |B_i| e^{-S_i} \right).
\]
and we see that $|X_n|\rightarrow\infty$ on the event $\{\sum_{i=1}^\infty |B_i| e^{-S_i} < M; X_0> 2M\}$ since $S_n\rightarrow\infty$ under $\P^\alpha$. Hence, we can conclude that
\[
\P^\alpha ( |X_n| \geq M,\ \mbox{for all}\ n\geq1 \,|\, |X_0| \geq 2M  ) > 0. 
\]
Combining this with the fact that the set  $[M,\infty)$ is attainable by $\{|X_n|\}_{n\geq0}$ for sufficiently large $M$ (by Assumption~\ref{ass: regularity condition AR(1) process}), 
Theorem~8.3.6 of \cite{meyn2009} completes the proof. 
\end{proof}

\begin{proof}[Proof of Lemma~\ref{lem: asymptotic equivalence}]
Let $G$ be an open set bounded away from $\mathbb C$ so that $G \subseteq (\mathbb S\setminus\mathbb C)^{-\gamma}$ for some $\gamma>0$.
For a given $\delta>0$, due to the assumed asymptotic equivalence, $\P(X_n \in \mathbb (S\setminus\mathbb C)^{-\gamma}, d(X_n,Y_n) \geq \delta) = o(\epsilon_n)$.
Therefore,
\begin{align*}
\varliminf_{n\to\infty} \epsilon_n^{-1} \P(Y_n\in G)
&\geq \notag
\varliminf_{n\to\infty} \epsilon_n^{-1} \P(X_n\in G^{-\delta},d(X_n,Y_n)<\delta)\\
&= \notag
\varliminf_{n\to\infty} \epsilon_n^{-1} \{ \P(X_n\in G^{-\delta}) - \P(X_n\in G^{-\delta},d(X_n,Y_n) \geq \delta) \}\\
&\geq 
\varliminf_{n\to\infty} \epsilon_n^{-1} \{ \P(X_n\in G^{-\delta})-
\P(X_n\in (\mathbb S\setminus\mathbb C)^{-\gamma},d(X_n,Y_n) \geq \delta) \}\\
&= \notag
\varliminf_{n\to\infty} \epsilon_n^{-1} \P(X_n\in G^{-\delta}) 
\geq 
\nu(G^{-\delta}).
\end{align*}
Since $G$ is an open set, $G = \bigcup_{\delta>0} G^{-\delta}$. 
Due to the continuity of measures, $\lim_{\delta\to 0} \nu ( G^{-\delta} ) = \nu(G)$, and hence, we arrive at the lower bound 
\[
\varliminf_{n\to\infty} \epsilon_n^{-1} \P(Y_n\in G) \geq \nu(G)
\]
by taking $\delta\to0$. 
Now, turning to the upper bound, consider a closed set $F$ bounded away from $\mathbb C$ so that $F \subseteq (\mathbb S\setminus\mathbb C)^{-\gamma}$ for some $\gamma>0$. 
Given a $\delta>0$, by the asymptotic equivalence assumption, $\P(Y_n \in F, d(X_n,Y_n) \geq \delta) \leq \P(Y_n \in (\mathbb S\setminus\mathbb C)^{-\gamma}, d(X_n,Y_n) \geq \delta) = o(\epsilon_n)$. Therefore,
\begin{align*}
\varlimsup_{n\to\infty} \epsilon_n^{-1} \P(Y_n\in F)
&=
\varlimsup_{n\to\infty} \epsilon_n^{-1} \{ \P(Y_n\in F,d(X_n,Y_n)<\delta) + \P(Y_n\in F,d(X_n,Y_n)\geq \delta) \}\\
&\leq
\varlimsup_{n\to\infty} \epsilon_n^{-1} \{ \P(X_n\in F_\delta) + \P(Y_n\in F,d(X_n,Y_n)\geq \delta) \}\\
&= 
\varlimsup_{n\to\infty} \epsilon_n^{-1} \P(X_n\in F_\delta)
\leq 
\nu(F_\delta)
\end{align*}
as long as $\delta$ is small enough so that $F_\delta$ is bounded away from $\mathbb C$. 
Note that $\{F_\delta\}$ is a decreasing sequence of sets, $F = \bigcap_{\delta>0} F_\delta$ (since $F$ is closed), 
and $\nu \in \mathbb M (\mathbb S\setminus\mathbb C)$ (and hence $\nu$ is a finite measure on $\mathbb S\setminus\mathbb C^r$ for some $r > 0$ such that $F_\delta \subseteq \mathbb S\setminus\mathbb C^r$ for some $\delta > 0$). 
Due to the continuity (from above) of finite measures, $\lim_{\delta\to0} \nu(F_\delta) = \nu(F)$. 
Therefore, we arrive at the upper bound
$
\varlimsup_{n\to\infty} \epsilon_n^{-1} \P(Y_n\in F) \leq \nu(F)
$
by taking $\delta\to0$. 
\end{proof}

\section{Proofs of Section~\ref{SECTION: TAIL ESTIMATES FOR THE AREA UNDER FIRST RETURN TIME/REGENERATION CYCYLE}}\label{section: proofs of section 3.1}
This section provides the proof of Theorem~\ref{thm: tail estimates for the area under first return time/regeneration cycle}. 
Before turning to technical details, we briefly describe our strategy for proving the tail asymptotics of $\mathfrak B$. A similar idea is behind the proof for $\mathfrak R$. 
Let
\begin{equation}\label{eq: T(u) and K_beta^gamma(u)}
\notationdef{nota-T}{T(u)} = \inf \{ n\geq0 \colon |X_n| > u \}\qquad\text{and}\qquad
\notationdef{nota-K-beta-gamma}{K_\beta^\gamma(u)} = \inf \{ n>T(u^\beta)\colon |X_n| \leq u^\gamma \}
\end{equation}
where $0<\notationdef{nota-gamma}{\gamma} < \notationdef{nota-beta}{\beta} < 1$.
We can then write  
\begin{equation}\label{eq: decomposition of area under first return time}
\mathfrak B =  \sum_{n=0}^{T(u^\beta)-1} X_n + \sum_{n=T(u^\beta)}^{K_\beta^\gamma(u)-1} X_n + \sum_{n=K_\beta^\gamma(u)}^{\tau_d-1} X_n.
\end{equation}
We will choose $\beta$ close enough to $1$ and $\gamma$ far enough from $1$ so that $\beta + \gamma > 1$ and we can find $\rho \in (\gamma, \beta)$ such that $\beta  - \gamma + \rho> 1$.
The proof of Theorem~\ref{thm: tail estimates for the area under first return time/regeneration cycle}~(1) is based on the following steps. 
\begin{itemize}
\item On the event $\{T(u^\beta)<\tau_d\}$, the first and the last term on the right hand side (r.h.s.) of \eqref{eq: decomposition of area under first return time} are negligible in contributing to the tail asymptotics. 
Proposition~\ref{prop: neglibigle parts in the area under first return time} below proves such properties. 
Lemma~\ref{lem: bound on moments of first return time} is useful in showing Proposition~\ref{prop: neglibigle parts in the area under first return time}. 
\item In view of the previous bullet, the second term on the r.h.s.\ of \eqref{eq: decomposition of area under first return time} plays the key role in $\P (\mathfrak B>u)$. 
Our analysis relies on the fact that the Markov chain $X_n$ resembles a multiplicative random walk in the corresponding regime. 
Proposition~\ref{prop: determined part in the area under first return time} below proves such asymptotics. 
Lemmas~\ref{lem: almost sure convergence of G+ and G-}, \ref{lem: integrability of Z bar} are helpful for proving Proposition~\ref{prop: determined part in the area under first return time}. 
\end{itemize}
Similarly, the proof of Theorem~\ref{thm: tail estimates for the area under first return time/regeneration cycle}~(2) hinges on Propositions~\ref{prop: negligible parts in the area under regeneration cycle} and \ref{prop: determined part in the area under regeneration cycle}, which play the similar roles as Proposition~\ref{prop: neglibigle parts in the area under first return time} and \ref{prop: determined part in the area under first return time}, respectively. 

\begin{proposition}
\label{prop: neglibigle parts in the area under first return time}
\linksinthm{prop: neglibigle parts in the area under first return time}
Suppose that Assumptions~\ref{ass: regularity condition AR(1) process} and \ref{ass: minorization AR(1) process} hold.
There exist a $\beta<1$ and $0<\gamma<\beta$ such that 
\[
\P \left( \left| \sum_{n=0}^{T(u^\beta)-1} X_n \right| > u, T(u^\beta) < \tau_d \right)
\quad \text{and}\quad 
\P \left( \left| \sum_{n=K_\beta^\gamma(u)}^{\tau_d-1} X_n \right| > u, T(u^\beta) < \tau_d \right)
\]
are of order $o( u^{-\alpha} )$ as $u\to\infty$.
\end{proposition}

\begin{proposition}
\label{prop: determined part in the area under first return time}
\linksinthm{prop: determined part in the area under first return time}
Suppose that Assumptions~\ref{ass: regularity condition AR(1) process} and \ref{ass: minorization AR(1) process} hold.
 There exist $0<\gamma<\beta<1$ (identical to those in Proposition \ref{prop: neglibigle parts in the area under first return time}) such that
\begin{align*}
&\lim_{u\to\infty} u^\alpha \P \left( \sum_{n=T(u^\beta)}^{K_\beta^\gamma(u)-1} X_n>u, T(u^\beta) < \tau_d \right) 
= 
C_\infty \E^\alpha [(Z^+)^\alpha \1_{\{\tau_d=\infty\}}]\\
\mbox{and}\quad
&\lim_{u\to\infty} u^\alpha \P \left( \sum_{n=T(u^\beta)}^{K_\beta^\gamma(u)-1} X_n < -u, T(u^\beta) < \tau_d \right) 
= 
C_\infty \E^\alpha [(Z^-)^\alpha \1_{\{\tau_d=\infty\}}].
\end{align*}
\end{proposition}

\begin{proof}[Proof of Theorem~\ref{thm: tail estimates for the area under first return time/regeneration cycle} (1)]
\linksinpf{part 1 of thm: tail estimates for the area under first return time/regeneration cycle}
Recalling $T(u^\beta) = \inf \{ n\geq0 \colon |X_n| > u^\beta \}$ for $\beta\in(0,1)$, write
\begin{align}
\P( \pm \mathfrak B>u)
&=
\P( \pm \mathfrak B>u,T(u^\beta)<\tau_d)
\label{eq: theorem 3.1 eq01}
+ 
\P( \pm\mathfrak B>u,T(u^\beta)\geq\tau_d). 
\end{align}
Since $\P (\tau_d > n)$ decays geometrically in $n$ since $|X_0| \leq d$, 
and $|X_n| \leq u^\beta$ for $n\leq \tau_d-1$ on $T(u^\beta)\geq\tau_d$, we have that 
\begin{align}
\notag 
\P (\pm\mathfrak B>u, T(u^\beta) \geq \tau_d)
&\leq \P (|\mathfrak B|>u, T(u^\beta) \geq \tau_d)
\leq  
\P \left( \sum_{n=0}^{\tau_d-1} |X_n| >u, T(u^\beta) \geq \tau_d \right) \\
&\leq 
\P (u^\beta \tau_d \geq u)
=
\P (\tau_d \geq u^{1-\beta})
= 
o(u^{-\alpha}).\label{eq: theorem 3.1 eq02}
\end{align}
Using \eqref{eq: theorem 3.1 eq01} and \eqref{eq: theorem 3.1 eq02}, we can focus on analyzing the first term on the r.h.s.\ of \eqref{eq: theorem 3.1 eq01}. 
For $0<\gamma<\beta<1$, recall 
$
K_\beta^\gamma(u) = \inf \{ n> T(u^\beta)\colon |X_n| \leq u^\gamma \}
$. 
Using the decomposition in \eqref{eq: decomposition of area under first return time} and Proposition~\ref{prop: neglibigle parts in the area under first return time}, we obtain that, for $\epsilon\in(0,1)$, 
\begin{align}
\P ( \mathfrak B>u, T(u^\beta) < \tau_d )
\nonumber&\leq 
\P \left( \sum_{n=T(u^\beta)}^{K_\beta^\gamma(u)-1} X_n>(1-\epsilon)u, T(u^\beta) < \tau_d \right)\nonumber
\\
&\hspace{12.5pt}+
\P \left( \left| \sum_{n=0}^{T(u^\beta)-1} X_n \right| > \frac{\epsilon u}{2}, T(u^\beta) < \tau_d \right)
+
\P \left( \left| \sum_{n=K_\beta^\gamma(u)}^{\tau_d-1} X_n \right| > \frac{\epsilon u}{2}, T(u^\beta) < \tau_d \right)
\nonumber
\\
&=
\P \left( \sum_{n=T(u^\beta)}^{K_\beta^\gamma(u)-1} X_n>(1-\epsilon)u, T(u^\beta) < \tau_d \right) + o(u^{-\alpha}),
\label{eq: theorem 3.1 eq03}
\end{align}
and
\begin{align}
\nonumber 
\P ( \mathfrak B>u, T(u^\beta) < \tau_d )
\nonumber&\geq 
\P \left( \sum_{n=T(u^\beta)}^{K_\beta^\gamma(u)-1} X_n>(1+\epsilon)u, T(u^\beta) < \tau_d \right)\\
&\hspace{12.5pt}-\P \left( \left| \sum_{n=0}^{T(u^\beta)-1} X_n \right| > \frac{\epsilon u}{2}, T(u^\beta) < \tau_d \right)
-\P \left( \left| \sum_{n=K_\beta^\gamma(u)}^{\tau_d-1} X_n \right| > \frac{\epsilon u}{2}, T(u^\beta) < \tau_d \right)\nonumber
\\
&=\P \left( \sum_{n=T(u^\beta)}^{K_\beta^\gamma(u)-1} X_n>(1+\epsilon)u, T(u^\beta) < \tau_d \right) + o(u^{-\alpha})
.\label{eq: theorem 3.1 eq04}
\end{align}
From \eqref{eq: theorem 3.1 eq03}, \eqref{eq: theorem 3.1 eq04}, and Proposition~\ref{prop: determined part in the area under first return time}, 
\begin{align*}
u^{\alpha} \P ( \mathfrak B>u, T(u^\beta) < \tau_d) 
&\geq 
(1+\epsilon)^{-\alpha} C_\infty \E^\alpha [(Z^+)^\alpha \1_{\{\tau_d=\infty\}}] + o(1);
\\
u^{\alpha} \P ( \mathfrak B>u, T(u^\beta) < \tau_d)
&\leq 
(1-\epsilon)^{-\alpha} C_\infty \E^\alpha [(Z^+)^\alpha \1_{\{\tau_d=\infty\}}]  + o(1).
\end{align*}
Since $\epsilon$ is arbitrary, 
$$
u^{\alpha} \P ( \mathfrak B>u, T(u^\beta) < \tau_d) = C_\infty \E^\alpha [(Z^+)^\alpha \1_{\{\tau_d=\infty\}}] + o(1). 
$$
Along with \eqref{eq: theorem 3.1 eq01}, \eqref{eq: theorem 3.1 eq02}, this proves the first limit of Theorem~\ref{thm: one-sided large deviations} (1): 
$$u^{\alpha} \P ( \mathfrak B>u) = C_\infty \E^\alpha [(Z^+)^\alpha \1_{\{\tau_d=\infty\}}] + o(1).$$
We can use similar estimates to ``sandwich'' the quantity $\P(\mathfrak B<-u)$ and establish the second limit of Theorem~\ref{thm: one-sided large deviations} (1). 
\end{proof}

Now we move on to proving Theorem~\ref{thm: tail estimates for the area under first return time/regeneration cycle} (2). We first need the following propositions.

\begin{proposition}\label{prop: negligible parts in the area under regeneration cycle}
\linksinthm{prop: negligible parts in the area under regeneration cycle}
Let Assumptions~\ref{ass: regularity condition AR(1) process} and \ref{ass: minorization AR(1) process} hold. 
There exist $0<\gamma<\beta<1$ such that 
\[
\P \left( \left| \sum_{n=0}^{T(u^\beta)-1} X_n \right| > u, T(u^\beta) < r_1 \right)
\qquad \text{and}\qquad
\P \left( \left| \sum_{n=K_\beta^\gamma(u)}^{r_1-1} X_n \right| > u, T(u^\beta) < r_1 \right)
\]
are of order $o( u^{-\alpha} )$ as $u\to\infty$.
\end{proposition}

\begin{proposition}
\label{prop: determined part in the area under regeneration cycle}
\linksinthm{prop: determined part in the area under regeneration cycle}
Let Assumptions~\ref{ass: regularity condition AR(1) process} and \ref{ass: minorization AR(1) process} hold. 
There exist $0<\gamma<\beta<1$ (the same as in Proposition \ref{prop: negligible parts in the area under regeneration cycle}) such that 
\begin{align*}
&\lim_{u\to\infty} u^\alpha \P \left( \sum_{n=T(u^\beta)}^{K_\beta^\gamma(u)-1} X_n>u, T(u^\beta) < r_1 \right) 
= 
C_+\\
\quad\mbox{and}\quad
&\lim_{u\to\infty} u^\alpha \P \left( \sum_{n=T(u^\beta)}^{K_\beta^\gamma(u)-1} X_n < -u, T(u^\beta) < r_1 \right) 
= 
C_-,
\end{align*}
where 
$
C_+ = C_\infty \E^\alpha [(Z^+)^\alpha \1_{\{r_1=\infty\}}]
$
and
$
C_+ = C_\infty \E^\alpha [(Z^-)^\alpha \1_{\{r_1=\infty\}}]
$.
\end{proposition}

\begin{proof}[Proof of Theorem~\ref{thm: tail estimates for the area under first return time/regeneration cycle} (2)]
\linksinpf{part 2 of thm: tail estimates for the area under first return time/regeneration cycle}
Using similar arguments as in \eqref{eq: theorem 3.1 eq01} and \eqref{eq: theorem 3.1 eq02}, we can focus on 
$
\P( \pm \mathfrak R > u, T(u^\beta) < r_1 )
$. 
Combining the similar ``sandwich'' technique as in \eqref{eq: theorem 3.1 eq03} and \eqref{eq: theorem 3.1 eq04} with Proposition~\ref{prop: negligible parts in the area under regeneration cycle}, it remains to analyze 
\[
u^\alpha \P \left( \sum_{n=T(u^\beta)}^{K_\beta^\gamma(u)-1} X_n>u, T(u^\beta) < r_1 \right). 
\]
Using Proposition~\ref{prop: determined part in the area under regeneration cycle}, we conclude the proof. 
\end{proof}

Next we prove Proposition~\ref{prop: neglibigle parts in the area under first return time}. 
For this, we need the following lemma. 
Let $\{Y_n\}_{n\geq0}$ be the $\mathbb R_+$-valued Markov chain defined by $\notationdef{nota-Y-n}{Y_{n+1}} = A_{n+1}Y_n + |B_{n+1}|$, for $n\geq0$, and $\notationdef{nota-tau}{\tau} = \inf\{ n\geq1 \colon Y_n \leq d\}$.

\begin{lemma}
\label{lem: bound on moments of first return time}
\linksinthm{lem: bound on moments of first return time} 
Suppose that Assumptions~\ref{ass: regularity condition AR(1) process} and \ref{ass: minorization AR(1) process} hold.
Let $L>0$ be given, and let $\epsilon>0$ be such that $\lfloor \alpha-\epsilon \rfloor \geq 1$. 
Then there exists a positive constant $c$ such that, for sufficiently large $x$, 
\[
\E [ \tau^{\alpha+L} \,|\, Y_0=x ] \leq c x^{\lfloor \alpha-\epsilon \rfloor}.
\]
In particular $\E [ \tau^{\alpha+L} \,|\, Y_0=x ] = \mathcal O(x)$. 
\end{lemma}

\linkdest{proof of prop: neglibigle parts in the area under first return time}
\begin{proof}[Proof of Proposition~\ref{prop: neglibigle parts in the area under first return time}]
\linksinpf{prop: neglibigle parts in the area under first return time}
To begin with, note that 
\begin{align*}
\P \left( \left| \sum_{n=0}^{T(u^\beta)-1} X_n \right| > u, T(u^\beta) < \tau_d \right)
&\leq 
\P \left( \sum_{n=0}^{T(u^\beta)-1} |X_n|>u, T(u^\beta) < \tau_d \right)\\
&\leq 
\P ( u^\beta \tau_d >u ) 
= 
\P ( \tau_d >u^{1-\beta} ),
\end{align*}
which decays exponentially. 
It remains to show the second claim. 
Let \notationdef{nota-rho}{$\rho$} be a number such that \linkdest{cond-rho-btween-gamma-and-beta}$\rho\in (\gamma, \beta)$ and \linkdest{cond-beta-gamma-rho-greater-than-one}$\beta - \gamma + \rho > 1$, 
and define 
\[
\notationdef{nota-mathfrak-E-1}{\mathfrak E_1(u)} = \{ \exists\, n \text{ such that } K_\beta^\gamma(u)\leq n \leq \tau_d \text{ and } |X_n| \geq u^\rho \}.
\]
Note that 
\begin{align*}
\P \left( \left| \sum_{n=K_\beta^\gamma(u)}^{\tau_d-1} X_n \right| >u, T(u^\beta) < \tau_d \right)
&\leq 
\P \left( \sum_{n=K_\beta^\gamma(u)}^{\tau_d-1} |X_n| >u, T(u^\beta) < \tau_d, \mathfrak E_1(u) \right)\\ 
&\hspace{12.5pt}+ 
\P \left( \sum_{n=K_\beta^\gamma(u)}^{\tau_d-1} |X_n | >u, T(u^\beta) < \tau_d, (\mathfrak E_1(u))^c \right), 
\end{align*}
where the second term in the last equation is bounded by $\P ( \tau_d > u^{1-\rho} )$, and hence is of order $o(u^{-\alpha})$. 
It remains to analyze the first term, which is bounded by $\P ( T(u^\beta) < \tau_d, \mathfrak E_1(u) )$. 
Our goal here is to show that 
\begin{equation}\label{eq: proposition 4.1 eq01}
\P ( T(u^\beta) < \tau_d, \mathfrak E_1(u) ) = o(u^{-\alpha}),\quad\mbox{as}\ u\to\infty.
\end{equation}
To begin with, note that, 
on $\mathfrak E_1(u)$, where $K_\beta^\gamma(u)<\infty$ almost surely, we can define 
$\{Y_n'\}_{n\geq0}$ 
as follows
\[
\notationdef{nota-Y-n-prime}{Y_0'} = u^\gamma,\qquad Y_{n+1}' = A_{K_\beta^\gamma(u)+n+1} Y_n' + |B_{K_\beta^\gamma(u)+n+1}|,\qquad\mbox{for}\ n\geq0.
\]
so that 
$
|X_{K_\beta^\gamma(u)+n}| \leq Y_n'
$, 
for all $
n\geq0.
$
Let $\notationdef{nota-tau-prime}{\tau'} \triangleq \inf\{ n\geq1 \colon Y_n' \leq d\}$. 
Since $\1_{\{T(u^\beta)<\tau_d\}} \in m\mathcal F_{T(u^\beta)}$, and $Y_n'$ is well-defined on $\{T(u^\beta)<\tau_d\}$ (since $K_\beta^\gamma(u) < \infty$ $\P$-almost surely there), 
we have that 
\begin{align}
\P ( T(u^\beta) < \tau_d, \mathfrak E_1(u) ) 
&\leq
\P ( T(u^\beta) < \tau_d, \exists n \leq \tau' \text{ such that } Y_n' \geq u^\rho ) \nonumber\\
&= 
\P ( T(u^\beta) < \tau_d) \P(\exists\, n\leq \tau' \text{ such that } Y_n' \geq u^{\rho}\,\big|\,T(u^\beta) < \tau_d), \label{eq: qbnvuerwpio}
\end{align}
where $\P ( T(u^\beta) < \tau_d) = \mathcal O (u^{-\alpha\beta})$ (cf.\ Corollary~4.2 of \cite{collamore2016}). 
Since we have chosen $\beta$, $\gamma$, and $\rho$ in such a way that $\beta - \gamma + \rho > 1$, it remains to show that the second term on the r.h.s.\ is $\mathcal O(u^{-\alpha(\rho-\gamma)})$.
Recall the definition of $Y_n$ and $\tau$, and note that from the strong Markov property,
\begin{equation*}
\P(\exists\, n\leq \tau' \text{ such that } Y_n' \geq u^\rho\, \big|\,T(u^\beta) < \tau_d) 
= \P(\exists\, n\leq \tau \text{ such that } Y_n \geq u^\rho \,|\, Y_0=u^\gamma )
\end{equation*}
as $u\to \infty$. 
Recall Remark~\ref{a-remark-on-dual-change-of-measure} and consider
$
T = \inf \{ n\geq1 \colon Y_n \geq u^\rho \}
$.
We obtain
\begin{align*}
&\P(\exists\, n\leq \tau \text{ such that } Y_n \geq u^\rho \,|\, Y_0=u^\gamma )
\\
&=\P( T < \tau \,|\, Y_0 = u^\gamma ) 
=
\E^\alpha \left[ e^{-\alpha S_{T}} \1_{\{T<\tau\}} \,\middle|\, Y_0=u^\gamma \right]\\
&= 
u^{-\alpha(\rho-\gamma)} 
\E^\alpha \left[ 
\left( \frac{Y_{T}}{u^\rho} \right)^{-\alpha} 
\left( \frac{Y_{T}}{e^{S_{T}} u^\gamma} \right)^\alpha 
\1_{\{T<\tau\}} \,\middle|\, Y_0=u^\gamma \right]\\ 
&\leq 
u^{-\alpha(\rho-\gamma)}
\E^\alpha \left[ 
\left( \frac{Y_{T}}{e^{S_{T}} u^\gamma} \right)^\alpha 
\1_{\{T<\tau\}} \,\middle|\, Y_0=u^\gamma \right].
\end{align*}
Now it is sufficient to show that 
\begin{equation}\label{eq: proposition 4.1 eq02}
\varlimsup_{u\to\infty} \E^\alpha \left[ 
\left( \frac{Y_{T}}{e^{S_{T}} u^\gamma} \right)^\alpha 
\1_{\{T<\tau\}} \,\middle|\, Y_0=u^\gamma \right] < \infty,
\end{equation}
to prove that the r.h.s.\ of \eqref{eq: qbnvuerwpio} is $\mathcal O (u^{-\alpha(\rho - \gamma)})$, and hence,
$\P ( T(u^\beta) < \tau_d, \mathfrak E_1(u) ) = o(u^{-\alpha})$. 
To show \eqref{eq: proposition 4.1 eq02}, note that 
\begin{align*}
\frac{Y_{T}}{e^{S_{T}} u^\gamma} 
&= 
e^{-S_{T}} u^{-\gamma} \left(  e^{S_{T}} u^\gamma + e^{S_{T}} \sum_{k=1}^{T} |B_k| e^{-S_k} \right)
= 
1 + u^{-\gamma} \sum_{k=1}^{T} |B_k| e^{-S_k}.
\end{align*}
Thus, we have that 
\begin{align*}
\frac{Y_{T}}{e^{S_{T}} u^\gamma} 
\1_{\{T<\tau\}} 
&\leq
1 + u^{-\gamma} \sum_{k=1}^{T} |B_k| e^{-S_k} \1_{\{T<\tau\}}
\leq
1 + u^{-\gamma} \sum_{k=1}^\infty |B_k| e^{-S_k} \1_{\{k<\tau\}},
\end{align*}
and hence, 
\begin{align}
&\E^\alpha \left[ \left( \frac{Y_{T}}{e^{S_{T}} u^\gamma} \1_{\{T<\tau\}} \right)^\alpha \,\middle|\, Y_0=u^\gamma \right]^{1/\alpha}
\notag
\\
&\leq 
\E^\alpha \left[ \left( 1 + u^{-\gamma} \sum_{k=1}^\infty |B_k| e^{-S_k} \1_{\{k<\tau\}} \right)^\alpha \,\middle|\, Y_0=u^\gamma \right]^{1/\alpha}
\notag
\\
&\leq 
1 + \sum_{k=1}^\infty \E^\alpha \left[ u^{-\alpha\gamma} |B_k|^\alpha e^{-\alpha S_k} \1_{\{k<\tau\}} \,\middle|\, Y_0=u^\gamma \right]^{1/\alpha}
\label{eq: proposition 4.1 eq03}\\ 
\notag&= 
1 + u^{-\gamma} \sum_{k=1}^\infty \E^\alpha \left[ e^{-\alpha S_k} |B_k|^\alpha \1_{\{k<\tau\}} \,\middle|\, Y_0=u^\gamma \right]^{1/\alpha}\\
&= 
1 + u^{-\gamma} \sum_{k=1}^\infty \E \left[ |B_k|^\alpha \1_{\{k<\tau\}} \,\middle|\, Y_0=u^\gamma \right]^{1/\alpha}
\label{eq:6.10}\\
&\leq 
1 + u^{-\gamma} \sum_{k=1}^\infty (\E |B_k|^\alpha)^{1/\alpha} \P ( \tau\geq k \,|\, Y_0=u^\gamma )^{1/\alpha}  \label{eq:6.11}\\
&\leq 
1 + u^{-\gamma} (\E |B_1|^\alpha)^{1/\alpha} \big(\E[\tau^{\alpha+L} \,|\, Y_0=u^\gamma]\big)^{1/\alpha} \sum_{k=1}^\infty k^{-(\alpha+L)/\alpha},
\label{eq:6.12}
\end{align}
for some $L>0$, where \eqref{eq: proposition 4.1 eq03} is from  Fatou's lemma and Minkowski's inequality, \eqref{eq:6.10} is from Remark~\ref{a-remark-on-dual-change-of-measure} with $T = k$ and $R = |B_k|^\alpha \1_{\{k<\tau\}}$,  \eqref{eq:6.11} is from the fact that $\1_{\{k<\tau\}} \leq \1_{\{k\leq \tau\}}$ and  $\1_{\{k\leq \tau\}}\in m\mathcal F_{k-1}$ so that $\1_{\{k \leq \tau\}}$ and $|B_k|^\alpha$ are independent, and \eqref{eq:6.12} is from Markov's inequality. 
Using Lemma~\ref{lem: bound on moments of first return time} above, 
we prove \eqref{eq: proposition 4.1 eq02}, which, in turn, proves \eqref{eq: proposition 4.1 eq01}. This concludes the proof of Proposition~\ref{prop: neglibigle parts in the area under first return time}.
\end{proof}

Set 
\begin{equation}\label{eq: G+}
\notationdef{nota-scr-G-plus}{\mathscr G_+(u)} 
= 
u^{(1-\beta)\alpha} \P^{\mathscr D_{T(u^\beta)}^\alpha} \left( \sum_{n=T(u^\beta)}^{K_\beta^\gamma(u)-1} X_n>u \,\middle|\, \mathcal F_{T(u^\beta)} \right) 
\left( \frac{X_{T(u^\beta)}}{u^\beta} \right)^{-\alpha} \1_{\{Z_{T(u^\beta)}>0\}},
\end{equation}
and
\begin{equation}\label{eq: G-}
\notationdef{nota-scr-G-minus}{\mathscr G_-(u)}
= 
u^{(1-\beta)\alpha} \P^{\mathscr D_{T(u^\beta)}^\alpha} \left( \sum_{n=T(u^\beta)}^{K_\beta^\gamma(u)-1} X_n>u \,\middle|\, \mathcal F_{T(u^\beta)} \right)  
\left| \frac{X_{T(u^\beta)}}{u^\beta} \right|^{-\alpha} \1_{\{Z_{T(u^\beta)}\leq 0\}}. 
\end{equation}
Recall $C_\infty$ in \eqref{eq: tail asymptotics on the stationary distribution of AR(1) process}. 
The following two lemmas are useful in proving Proposition~\ref{prop: determined part in the area under first return time}.

\begin{lemma}
\label{lem: almost sure convergence of G+ and G-}
\linksinthm{lem: almost sure convergence of G+ and G-}
Suppose that Assumptions~\ref{ass: regularity condition AR(1) process} and \ref{ass: minorization AR(1) process} hold.
Under the measure $\P^\alpha$,
\[
\mathscr G_+(u) \toas C_\infty \1_{\{Z>0\}}
\quad\mbox{and}\quad
\mathscr G_-(u)\toas 0,
\quad\mbox{as}\ u\to\infty.
\]
Moreover, $\mathscr G_+(u)$ and $\mathscr G_-(u)$ are bounded in $u$ by some constants almost surely. 
\end{lemma}

Recall that $Z_n$, $\tau_d$, and $T(u)$ are defined in \eqref{eq: discrepancy process}, \eqref{eq: first return time}, and \eqref{eq: T(u) and K_beta^gamma(u)}, respectively. 
\begin{lemma}
\label{lem: integrability of Z bar}
\linksinthm{lem: integrability of Z bar} 
Suppose that Assumptions~\ref{ass: regularity condition AR(1) process} and \ref{ass: minorization AR(1) process} hold.
The random variables $Z_{T(u^\beta)}^+\1_{\{T(u^\beta)<\tau_d\}}$ and $Z_{T(u^\beta)}^-\1_{\{T(u^\beta)<\tau_d\}}$ are bounded by 
\[
\notationdef{nota-bar-Z}{\bar Z} = |X_0| + \sum_{n=1}^\infty |B_n| e^{-S_n} \1_{\{n<\tau_d\}}. 
\]
In addition, 
$\E^\alpha [\bar Z^\alpha] < \infty$. 
\end{lemma}

\linkdest{proof of prop: determined part in the area under first return time}
\begin{proof}[Proof of Proposition~\ref{prop: determined part in the area under first return time}]
\linksinpf{prop: determined part in the area under first return time}
We focus on deriving the first asymptotics, since the second one follows using similar arguments. 
Note that 
\begin{align}
\notag 
u^\alpha\P \left( \sum_{n=T(u^\beta)}^{K_\beta^\gamma(u)-1} X_n > u, T(u^\beta) < \tau_d \right) 
&=
u^\alpha \P \left( \sum_{n=T(u^\beta)}^{K_\beta^\gamma(u)-1} X_n > u, X_{T(u^\beta)}>0, T(u^\beta) < \tau_d \right)\\
\notag&\hspace{12.5pt}+ 
u^\alpha \P \left( \sum_{n=T(u^\beta)}^{K_\beta^\gamma(u)-1} X_n > u, X_{T(u^\beta)}<0, T(u^\beta) < \tau_d \right)\\ 
&= 
\textbf{(I.1)} + \textbf{(I.2)}.\label{eq: proposition 4.2 eq01}
\end{align}
We first consider \textbf{(I.1)}. 
Applying the dual change of measure ${\mathscr D_{T(u^\beta)}^\alpha}$ together with Result~\ref{res: connection between original and dual chang of measures}, we obtain that
\[
\textbf{(I.1)}
=
\E^{\mathscr D_{T(u^\beta)}^\alpha} [ g_{\tau_d-1}(X_0,\ldots,X_{\tau_d-1}) e^{-\alpha S_{T(u^\beta)}} \1_{\{T(u^\beta)<\tau_d\}} ],
\]
where
\[
g_{\tau_d-1}(X_0,X_1,\ldots, X_{\tau_d-1}) = 
\begin{cases}1 & \text{if } \sum_{n=T(u^\beta)}^{K_\beta^\gamma(u)-1} X_n>u \text{ and } X_{T(u^\beta)}>0
\\
0 & \text{otherwise}.
\end{cases}
\]
Recall the expression for $Z_n$ given in \eqref{eq: discrepancy process}.  
Note that
\begin{align}
\textbf{(I.1)}
\nonumber&= 
u^\alpha
\E^{\mathscr D_{T(u^\beta)}^\alpha} \left[ g_{\tau_d-1}(X_0,\ldots,X_{\tau_d-1})  e^{-\alpha S_{T(u^\beta)}} \1_{\{T(u^\beta)<\tau_d\}} \right]\\
\nonumber&= 
u^{\alpha}
\E^{\mathscr D_{T(u^\beta)}^\alpha} \Big[ g_{\tau_d-1}(X_0,\ldots,X_{\tau_d-1})\cdot | X_{T(u^\beta)}|^{-\alpha}
\cdot |X_{T(u^\beta)}|^\alpha \cdot e^{-\alpha S_{T(u^\beta)}} \1_{\{T(u^\beta)<\tau_d\}} \Big]\\
\label{eq: proposition 4.2 eq02}&= 
\E^{\mathscr D_{T(u^\beta)}^\alpha} \left[ \mathscr G_+(u) (Z_{T(u^\beta)}^+)^\alpha \1_{\{T(u^\beta)<\tau_d\}} \right],
\end{align}
for all $n\geq0$. 
Using Lemma~\ref{lem: almost sure convergence of G+ and G-}, Lemma~\ref{lem: integrability of Z bar}, the dominated convergence theorem
%
and the fact that $T(u^\beta)\to\infty$ as $u\to\infty$, we obtain that 
\begin{align*}
\lim_{u\to\infty} \textbf{(I.1)} 
&= 
\lim_{u\to\infty} \E^{\mathscr D_{T(u^\beta)}^\alpha} \left[ (Z_{T(u^\beta)}^+)^\alpha \1_{\{T(u^\beta)<\tau_d\}} \mathscr G_+(u) \right]
= 
\lim_{u\to\infty} \E^\alpha \left[ (Z_{T(u^\beta)}^+)^\alpha \1_{\{T(u^\beta)<\tau_d\}} \mathscr G_+(u) \right]\\
&= 
\E^\alpha \left[ \lim_{u\to\infty} (Z_{T(u^\beta)}^+)^\alpha \1_{\{T(u^\beta)<\tau_d\}} \mathscr G_+(u) \right]= 
\E^\alpha \left[ (Z^+)^\alpha \1_{\{\tau_d=\infty\}} C_\infty \right]\\ 
&=
C_\infty \E^\alpha \left[ (Z^+)^\alpha \1_{\{\tau_d=\infty\}} \right].
\end{align*} 
Analogously, we have that 
\begin{equation}
\textbf{(I.2)}
\label{eq: proposition 4.2 eq03}= 
\E^{\mathscr D_{T(u^\beta)}^\alpha} \left[ (Z_{T(u^\beta)}^-)^\alpha \1_{\{T(u^\beta)<\tau_d\}} \mathscr G_-(u) \right]\to 0,
\quad\mbox{as}\ u\to\infty,
\end{equation}
where
$\mathscr G_-(u) $ was defined in \eqref{eq: G-}.
Using \eqref{eq: proposition 4.2 eq01}, \eqref{eq: proposition 4.2 eq02}, and \eqref{eq: proposition 4.2 eq03}, we prove the first asymptotics in Proposition~\ref{prop: determined part in the area under first return time}. 
The second one can be shown analogously. 
\end{proof}

We need the following lemmas to prove Proposition~\ref{prop: negligible parts in the area under regeneration cycle}. 
Let $Y_{n+1} = A_{n+1}Y_n + |B_{n+1}|$ and let $r$ be the first time that $Y_n$ regenerates. 

\begin{lemma}
\label{lem: bound on moments of regeneration time}
\linksinthm{lem: bound on moments of regeneration time} 
Suppose that Assumptions~\ref{ass: regularity condition AR(1) process} and \ref{ass: minorization AR(1) process} hold. 
Let $\epsilon>0$, and let $L>0$ be such that $\lfloor \alpha-\epsilon \rfloor \geq 1$. 
Then there exists a positive constant $c$ such that, for sufficiently large $x$, 
\[
\E [ r^{\alpha+L} \,|\, Y_0=x ] \leq c x^{\lfloor \alpha-\epsilon \rfloor}.
\]
In particular, $\E [ r^{\alpha+L} \,|\, Y_0=x ] = \mathcal O(x)$. 
\end{lemma}

\begin{lemma}\label{lem: asymptotics of P(T(u^beta))<r_1}
\linksinthm{lem: asymptotics of P(T(u^beta))<r_1} Suppose that Assumptions~\ref{ass: regularity condition AR(1) process} and \ref{ass: minorization AR(1) process} hold. 
We have that 
\[
\lim_{u\to\infty} u^{\alpha}\P ( T(u) < r_1) = \E \left[ e^{-\alpha \mathfrak X} \right] \E^\alpha \left[ |Z|^\alpha \1_{\{r_1=\infty\}} \right],
\]
where $\mathfrak X$ is the positive random variable such that 
$
\log X_{T(u)} - \log u
$ converges in distribution to $\mathfrak X$ as $u\to\infty$ under $\P^\alpha$. 
\end{lemma}

\linkdest{proof of prop: negligible parts in the area under regeneration cycle}
\begin{proof}[Proof of Proposition~\ref{prop: negligible parts in the area under regeneration cycle}]
\linksinpf{prop: negligible parts in the area under regeneration cycle}
By replacing $\tau_d$ with $r_1$, the proposition can be shown using almost identical arguments as in the proof of Proposition~\ref{prop: neglibigle parts in the area under first return time}. 
Nonetheless, we need to show that 
\begin{itemize}
\item $\P(T(u^\beta)<r_1) \sim cu^{-\alpha\beta}$ for some constant $c$, and that
\item $\E [r^{\alpha+\epsilon}|Y_0=x] = \mathcal O(x)$, where $Y_{n+1} = A_{n+1}Y_n + |B_{n+1}|$ and $r-1$ is the first time that $(Y_n,\eta_n)$ returns to the set $[-d,d]\times\{1\}$. 
\end{itemize}
For this, we use Lemmas~\ref{lem: bound on moments of regeneration time} and \ref{lem: asymptotics of P(T(u^beta))<r_1} above. 
\end{proof}

\linkdest{proof of prop: determined part in the area under regeneration cycle}
\begin{proof}[Proof of Proposition~\ref{prop: determined part in the area under regeneration cycle}]
\linksinpf{prop: determined part in the area under regeneration cycle}
Using Lemma~\ref{lem: almost sure convergence of G+ and G-}, Lemma~\ref{lem: integrability of Z bar}, the dominated convergence theorem and the fact that $T(u^\beta)\to\infty$ as $u\to\infty$, one can prove the first asymptotics. 
The second one follows by a similar analysis. 
\end{proof}

Next we provide proofs of all lemmas in this section. 
To show Lemma~\ref{lem: bound on moments of first return time}, we introduce a result on bounding functionals of passage times for Markov chains. 
Let $\{V_n\}_{n\geq0}$ be an $\{\mathcal F_n\}$-adapted stochastic process taking values in an unbounded subset of $\mathbb R_+$. 
Let $\{U_n\}_{n\geq0}$ be another $\{\mathcal F_n\}$-adapted stochastic process taking values in an unbounded subset of $\mathbb R_+$ such that $U_n$ is integrable for all $n\geq0$. 
Let $\tau_b^V = \inf\{ n\geq 0\colon V_n \leq b\}$ be the first time $V_n$ returning to the set $[0,b]$. 

\begin{result}[Theorem 2$'$ of \cite{aspandiiarov1999}]
\label{bounds of functions of passage times for Markov chains}
\linksinthmwopf{bounds of functions of passage times for Markov chains}
Suppose there exists a positive number $d$ and functions $g$ and $h$ that are positive on $(b,\infty)$, 
\[
U_n\leq h(V_n), \qquad\forall n\geq0,
\]
and    
\[
\E [U_{n+1} - U_n \,|\, \mathcal F_n] \leq -g(V_n) \qquad \text{on }\quad \{\tau_b^V > n\}.
\]
Suppose in addition that $f$ is a function on $[0,\infty]$ such that 
\begin{itemize}
\item $f\in C^2$, $f>0$, $f'>0$, $\lim_{x\to\infty} f(x) = \infty$,
\item $f$ is convex for sufficiently large $x$,
\item $\log f'$ is concave for sufficiently large $x$,
\item $f$ satisfies
\[
\liminf_{y\to\infty} \frac{g(y)}{f'\circ f^{-1}\circ h(y)} > 0,
\]
\item there exists a positive constant $c_f$ such that
\[
\varlimsup_{y\to\infty} \frac{f(2y)}{f(y)} \leq c_f.
\]
\end{itemize}
Then there exists a positive constant $c$ such that, for all $x\geq b$ 
\[
\E [ f(\tau_b^V) \,|\, V_0=x ] \leq c h(x).
\]

\end{result}

\begin{proof}[Proof of Lemma~\ref{lem: bound on moments of first return time}]
\linksinpf{lem: bound on moments of first return time} 
We first apply Result~\ref{bounds of functions of passage times for Markov chains} with $f(y) = y^{\alpha+L}$, $g(y) = c_2 y^{\underline\alpha}$, $h(y) = y^{\underline\alpha}$ where $\underline\alpha = \lfloor \alpha-\epsilon \rfloor$, and $c_2$ is a constant that we construct below, $U_n = h(Y_n)$, $V_n = Y_n$, and $\mathcal F_n = \sigma(Y_i:\, i\leq n)$. 
From the binomial formula, we see that there exist positive constants $c_1$ that depends on the first $(\underline\alpha-1)$-st moments of $A_1$ and $B_1$ such that, on $\{Y_n\geq 1\}$,
\begin{align*}
\E [ U_{n+1} - U_n \,|\, \mathcal F_n ] 
&\leq
(\E [A_1^{\underline\alpha}]-1)Y_{n}^{\underline\alpha} + c_1 Y_{n}^{\underline\alpha-1}.
\end{align*}
Using the fact that $0<\underline\alpha<\alpha$ and the moment generating function of $\log A_1$ is strictly convex on $[0,\alpha]$, we have $\E [A_1^{\underline\alpha}] < 1$. 
Thus, there exists a sufficiently large constant $d'$ and sufficiently small constant $c_2$ such that, on $\{Y_n>d'\}$,
\begin{align*}
\E [ U_{n+1} - U_n \,|\, \mathcal F_n ] 
&\leq
(\E [A_1^{\underline\alpha}]-1)Y_{n}^{\underline\alpha} + c_1 Y_{n}^{\underline\alpha-1} 
\leq 
- c_2 Y_{n}^{\underline\alpha}.
\end{align*}
As mentioned at the beginning of the proof, we set $g(y) = c_2 y^{\underline\alpha} = c_2 h(y)$ so that
\begin{align*}
\E [ U_{n+1} - U_n \,|\, \mathcal F_n ] 
\leq 
- g (Y_{n}).
\end{align*}
It is now straightforward to check that $f$, $g$, and $h$ satisfy all the conditions in Result~\ref{bounds of functions of passage times for Markov chains}.
If we set $\tilde \tau = \inf \{ n\geq1 \colon Y_n\leq d' \}$, Result~\ref{bounds of functions of passage times for Markov chains} implies  that there exists a positive constant $c_3$ such that 
\begin{equation}\label{eq: lemma 4.1 eq01}
\E [ \tilde \tau^{\alpha+L} \,|\, Y_0=x ] \leq c_3 x^{\lfloor \alpha-\epsilon \rfloor}
\end{equation}
for all $x\geq d'$.
We assume w.l.o.g.\ that $d'\geq d$. 
Note that $Y_n$ satisfies the same set assumptions as $X_n$, and hence, Lemma~\ref{lem: exponential decay of regeneration time} applies to $Y_n$ as well, and $\tau$ is bounded by the regeneration time of $Y_n$. Therefore,
we can choose a $t$ so that 
\[
\sup_{y\in[0,d']} \E [ t^{\tau} \,|\, Y_0=y ]^{1/(\alpha+L)}<\infty. 
\]
Using Minkowski's inequality we obtain that
\begin{align*}
\E [ \tau^{\alpha+L} \,|\, Y_0=x ]^{1/(\alpha+L)}&= 
\E [ (\tilde \tau + \tau - \tilde \tau)^{\alpha+L} \,|\, Y_0=x ]^{1/(\alpha+L)}\\ 
&\leq 
\E [ \tilde \tau^{\alpha+L} \,|\, Y_0=x ]^{1/(\alpha+L)} 
+ 
\E [ (\tau - \tilde \tau)^{\alpha+L} \,|\, Y_0=x ]^{1/(\alpha+L)}\\ 
&\leq 
\E [ \tilde \tau^{\alpha+L} \,|\, Y_0=x ]^{1/(\alpha+L)} 
+ 
\sup_{y\in[0,d']} \E [ \tau^{\alpha+L} \,|\, Y_0=y ]^{1/(\alpha+L)}\\
&\leq 
\E [ \tilde \tau^{\alpha+L} \,|\, Y_0=x ]^{1/(\alpha+L)} 
+ 
\sup_{y\in[0,d']} \E [ t^{\tau} \,|\, Y_0=y ]^{1/(\alpha+L)}
+
c_4,
\\
&\leq 
\E [ \tilde \tau^{\alpha+L} \,|\, Y_0=x ]^{1/(\alpha+L)} 
+ 
c_5,
\end{align*}
for some $c_4,c_5>0$. 
Along with \eqref{eq: lemma 4.1 eq01}, this implies that there exists a $c>0$ such that $\E [ \tau^{\alpha+L} \,|\, Y_0=x ] \leq c x^{\lfloor \alpha-\epsilon \rfloor}$ for sufficiently large $x$.
\end{proof}


The following lemma is useful in proving Lemma~\ref{lem: almost sure convergence of G+ and G-}.
Define
\begin{equation}\label{eq: E_2^gamma(u)}
\notationdef{nota-mathfrak-E-2}{\mathfrak E_2(u)} = \{ | B_n | \leq u^\gamma,\forall n \in \{1,\ldots, K_\beta^\gamma(u)\} \}. 
\end{equation}

\begin{lemma}
\label{lem: bounding sign-switching event}
\linksinthm{lem: bounding sign-switching event}
Suppose that Assumptions~\ref{ass: regularity condition AR(1) process} and \ref{ass: minorization AR(1) process} hold.
Fix an arbitrary constant $v$ such that $|v|>1$. 
For any \linkdest{cond-beta-gamma-greater-than-one}$\beta+\gamma>1$ and any $\epsilon>0$ there exists a $u_0$ sufficiently large so that, for all $u \geq u_0$, 
\[
\P ( (\mathfrak E_2(u))^c \,|\, X_0 = v u^\beta) \leq \epsilon |v| u^{-(1-\beta)\alpha}.
\]
\end{lemma}

\begin{proof}[Proof of Lemma~\ref{lem: almost sure convergence of G+ and G-}]
\linksinpf{lem: almost sure convergence of G+ and G-}
We first prove the statements associated with $\mathscr G_+(u)$. 
As $\1_{\{Z_{T(u^\beta)}>0\}} \toas \1_{\{Z>0\}}$ under $\P^\alpha$, it is sufficient to show that
\[
\lim_{u\to\infty} 
u^{(1-\beta)\alpha} 
\P \left( \sum_{k=T(u^\beta)}^{K_\beta^\gamma(u)-1} X_k > u \,\middle|\, \frac{X_{T(u^\beta)}}{u^\beta}=v \right) 
= C_\infty v^\alpha,\qquad\mbox{for}\ v>1.
\]
Set 
$
S_n^{(u)} \triangleq \sum_{i=1}^n \log(A_i+u^{-\gamma}\cdot|B_i|)
$
and fix $v \geq 1$.
Note that since 
\[
\frac{X_n}{X_{n-1}}
\leq A_n + \frac{|B_n|}{|X_{n-1}|} < A_n + |B_n| u^{-\gamma}\qquad\mbox{on}\qquad T(u^\beta) < n < K_\beta^\gamma(u),
\]
we have that, conditional on $X_0 = vu^\beta$, 
$
|X_k| \leq vu^\beta e^{S_k^{(u)}}
$
for all $k < K_\beta^\gamma(u)$. 
Therefore, from the Markov property,
\begin{align}
\nonumber&
\P \left( \sum_{k=T(u^\beta)}^{K_\beta^\gamma(u)-1} X_k > u \,\middle|\, \frac{X_{T(u^\beta)}}{u^\beta}=v \right)
= \P \left( \sum_{k=0}^{K_\beta^\gamma(u)-1} X_k > u \,\middle|\, \frac{X_0}{u^\beta}=v \right)
\leq \P \left( \sum_{k=0}^{K_\beta^\gamma(u)-1} |X_k| > u \,\middle|\, \frac{X_0}{u^\beta}=v \right)
\\
\nonumber&\leq 
\P \left( \sum_{k=0}^{K_\beta^\gamma(u)-1} e^{S_k^{(u)}} > \frac{u^{1-\beta}}{v} \right)
\leq 
\P \left( \sum_{k=0}^\infty e^{S_k^{(u)}} > \frac{u^{1-\beta}}{v} \right)
=
\P \left( \sum_{k=0}^\infty e^{S_k} + \sum_{k=0}^\infty e^{S_k^{(u)}} - \sum_{k=0}^\infty e^{S_k} > \frac{u^{1-\beta}}{v} \right)\\
\label{eq: lemma 4.2 eq01}&\leq 
\P \left( \sum_{k=0}^\infty e^{S_k} > \frac{u^{1-\beta}}{v} - \delta \right)
+ 
\P \left( \sum_{k=0}^\infty e^{S_k^{(u)}} - \sum_{k=0}^\infty e^{S_k} > \delta \right)
\end{align}
for any $\delta>0$.
Note that from \eqref{eq: tail asymptotics on the stationary distribution of AR(1) process},
\begin{equation}\label{eq: lemma 4.2 eq02}
\P \left( \sum_{k=0}^\infty e^{S_k} > \frac{u^{1-\beta}}{v} - \delta \right) \sim C_\infty \left( \frac{u^{1-\beta}}{v} \right)^{-\alpha}.
\end{equation}
Moreover, using Markov's inequality and the fact that $S_n^{(u)}\geq S_n$ we obtain that 
\begin{align*}
u^{(1-\beta)\alpha} \P \left( \sum_{k=0}^\infty e^{S_k^{(u)}} - \sum_{k=0}^\infty e^{S_k} > \delta \right)
&\leq 
\delta^{-1} u^{(1-\beta)\alpha} \E \left[ \sum_{k=0}^\infty e^{S_k^{(u)}} - \sum_{k=0}^\infty e^{S_k} \right]\\ 
&= 
\delta^{-1} u^{(1-\beta)\alpha} \left( \sum_{k=0}^\infty \big(\E [A_1+|B_1|u^{-\gamma}]\big)^k - \sum_{k=0}^\infty \big(\E [A_1]\big)^k \right)\\ 
&= 
\delta^{-1} u^{(1-\beta)\alpha} \left( \frac{1}{1-\E A_1 - u^{-\gamma}\E|B_1|} - \frac{1}{1-\E A_1} \right)\\ 
&= 
\delta^{-1} u^{(1-\beta)\alpha} \left( \frac{u^{-\gamma}\E|B_1|}{(1-\E A_1 - u^{-\gamma}\E|B_1|)(1-\E A_1)} \right)\\ 
&= 
\mathcal O (u^{(1-\beta)\alpha-\gamma}),
\end{align*}
for sufficiently large $u$.
In the second equality, we used $\E A_1 < 1$, (which follows from $\alpha>1$) and $\E (A_1 + |B_1|u^{-\gamma}) < 1$ for sufficiently large $u$'s.
By choosing 
$\beta$ sufficiently close to $1$ so that 
\linkdest{cond-one-minus-beta-alpha-less-than-gamma}
$
(1-\beta)\alpha < \gamma
$,
we have that 
\begin{equation}\label{eq: lemma 4.2 eq03}
u^{(1-\beta)\alpha}\P \left( \sum_{k=0}^\infty e^{S_k^{(u)}} - \sum_{k=0}^\infty e^{S_k} > \delta \right) = o(1). 
\end{equation}
Using \eqref{eq: lemma 4.2 eq01}--\eqref{eq: lemma 4.2 eq03}, an upper bound is given by, for $v>1$ 
\begin{equation}\label{eq: lemma 4.2 eq04}
\varlimsup_{u\to\infty} u^{(1-\beta)\alpha} \P \left( \sum_{k=T(u^\beta)}^{K_\beta^\gamma(u)-1} X_k > u \,\middle|\, \frac{X_{T(u^\beta)}}{u^\beta} = v \right) \leq C_\infty v^\alpha.
\end{equation}
Next, we show the corresponding lower bound. 
By the Markov property we obtain that 
\begin{equation}\label{eq: lemma 4.2 eq05}
\P \left( \sum_{k=T(u^\beta)}^{K_\beta^\gamma(u)-1} X_k > u \,\middle|\, \frac{X_{T(u^\beta)}}{u^\beta}=v \right) 
= 
\P \left( \sum_{k=0}^{K_\beta^\gamma(u)-1} X_k > u \,\middle|\, \frac{X_0}{u^\beta}=v \right). 
\end{equation}
Set $\underline S_n^{(u)} = \sum_{i=1}^n \log (A_i - u^{-\gamma}\cdot|B_i|)^+$.
Note that, conditional on $X_0= vu^\beta$
\begin{equation*}
\left| \frac{X_n}{X_{n-1}} \right| \geq \left( A_n - \frac{|B_n|}{|X_{n-1}|} \right)^+ > (A_n - u^{-\gamma}|B_n|)^+,
\qquad \forall n \leq K_\beta^\gamma(u),
\end{equation*}
which, in turn, implies that
\begin{equation}\label{eq: lemma 4.2 eq06}
vu^\beta e^{\underline S_{k}^{(u)}} 
\leq \left| X_{k} \right|,
\qquad 
\forall k \leq K_\beta^\gamma(u). 
\end{equation}
In particular, 
$$
vu^\beta e^{\underline S_{K_\beta^\gamma(u)}^{(u)}} 
\leq \left| X_{K_\beta^\gamma(u)} \right|\leq u^\gamma.
$$
Therefore,
\begin{equation}\label{eq: lemma 4.2 eq07}
K_\beta^\gamma(u) \geq \inf \{ n\geq1 \colon \underline S_n^{(u)}\leq -\log v - (\beta-\gamma)\log u \} \triangleq  K'(u)
\end{equation}
conditional on $X_0 = vu^\beta$.
Recall that  $\mathfrak E_2(u) =  \{ | B_n | \leq u^\gamma,\forall n \in \{1,\ldots, K_\beta^\gamma(u)\} \}$, and note that on $\mathfrak E_2(u)$ and $X_0 = vu^\beta$, $X_k\geq 0$ for all $k<K_\beta^\gamma(u)$. 
To see this, let $\kappa \triangleq \inf\{k\geq0: X_k < 0\}$ and observe that $\kappa < K_\beta^\gamma(u)$ implies that $A_{\kappa-1}X_{\kappa-1} \geq 0$ and $X_\kappa < -u^\gamma$, and hence, $B_\kappa < -u^\gamma$, leading to a contradiction. 
In view of this, \eqref{eq: lemma 4.2 eq05}, and \eqref{eq: lemma 4.2 eq06}, we have
\begin{align*}
\P \left( \sum_{k=T(u^\beta)}^{K_\beta^\gamma(u)-1} X_k > u \,\middle|\, \frac{X_{T(u^\beta)}}{u^\beta}=v \right)
&= 
\P \left( \sum_{k=0}^{K_\beta^\gamma(u)-1} X_k > u \,\middle|\, \frac{X_0}{u^\beta}=v \right)
\\
&
\geq 
\P \left( \sum_{k=0}^{K_\beta^\gamma(u)-1} |X_k| > u,\, \mathfrak E_2(u) \,\middle|\, \frac{X_0}{u^\beta}=v \right)\\
&\geq 
\P \left( \sum_{k=0}^{K_\beta^\gamma(u)-1} e^{\underline S_k^{(u)}} > \frac{u^{1-\beta}}{v},\, \mathfrak E_2(u) \,\middle|\, \frac{X_0}{u^\beta}=v \right)\\  
&\geq  
\P \left( \sum_{k=0}^{K_\beta^\gamma(u)-1} e^{\underline S_k^{(u)}} > \frac{u^{1-\beta}}{v} \,\middle|\, \frac{X_0}{u^\beta}=v \right)
-
\P \left( \mathfrak E_2(u)^c \,\middle|\, \frac{X_0}{u^\beta}=v \right)\\ 
&= 
\P \left( \sum_{k=0}^{K_\beta^\gamma(u)-1} e^{\underline S_k^{(u)}} > \frac{u^{1-\beta}}{v} \,\middle|\, \frac{X_0}{u^\beta}=v \right) + o (u^{\alpha(1-\beta)}) |v|,
\end{align*}
where the last equality is from Lemma~\ref{lem: bounding sign-switching event}. 
On the other hand, from \eqref{eq: lemma 4.2 eq07},
\begin{align}
\nonumber\P \left( \sum_{k=0}^{K_\beta^\gamma(u)-1} e^{\underline S_k^{(u)}} > \frac{u^{1-\beta}}{v}  \,\middle|\, \frac{X_0}{u^\beta}=v\right)
&\geq 
\P \left( \sum_{k=0}^{K'(u)-1} e^{\underline S_k^{(u)}} > \frac{u^{1-\beta}}{v} \right)\\
&\geq 
\P \left( \sum_{k=0}^\infty e^{\underline S_k^{(u)}} > \frac{u^{1-\beta}}{v}+\delta \right) - \P \left( \sum_{k=K'(u)}^\infty e^{\underline S_k^{(u)}} >\delta \right)\label{eq: lemma 4.2 eq08}
\end{align}
for any $\delta>0$.
Note that 
\begin{align}
\P \left( \sum_{k=K'(u)}^\infty e^{\underline S_k^{(u)}} >\delta \right)
\nonumber&\leq 
\delta^{-1} \E \left[ \sum_{k=K'(u)}^\infty e^{\underline S_k^{(u)}} \right]
=
\delta^{-1} \E \left[ e^{\underline S_{K'(u)}^{(u)}}\sum_{k=0}^\infty e^{\underline S_{k+K'(u)}^{(u)}-\underline S_{K'(u)}^{(u)}} \right]\\ 
\nonumber&\leq 
\frac{u^{\gamma-\beta}}{\delta v} \E \left[ \sum_{k=0}^\infty e^{\underline S_k^{(u)}} \right]
\leq 
\frac{u^{\gamma-\beta}}{\delta} \E \left[ \sum_{k=0}^\infty e^{\underline S_k^{(u)}} \right]
\leq 
\frac{u^{\gamma-\beta}}{\delta} \E \left[ \sum_{k=0}^\infty e^{S_k} \right] = o(u^{\gamma - \beta}),
\end{align}
and hence,
\begin{equation}\label{eq: lemma 4.2 eq09}
u^{(1-\beta)\alpha} \P \left( \sum_{k=K'(u)}^\infty e^{\underline S_k^{(u)}} >\delta \right)
 = o(1), 
\qquad\mbox{for}\ 
\linkdest{cond-beta-greater-than-alpha-plus-gamma-over-alpha-plus-one}\beta>(\alpha+\gamma)/(\alpha+1),
\end{equation}
allowing us to choose a suitable value of $\beta$ since $(\alpha+\gamma)/(\alpha+1)<1$.
Therefore, it remains to show that the first term in \eqref{eq: lemma 4.2 eq08} is lower bounded by $C_\infty u^{-(1-\beta)\alpha}v^\alpha$. 
Note that 
\begin{align}
&\nonumber
\P \left( \sum_{k=0}^\infty e^{\underline S_k^{(u)}} > \frac{u^{1-\beta}}{v}+\delta \right)\\ 
&\geq \notag
\P \left( \sum_{k=0}^\infty e^{S_k} > \frac{u^{1-\beta}}{v}+2\delta \right) - \P \left( \sum_{k=0}^\infty e^{S_k} - \sum_{k=0}^\infty e^{\underline S_k^{(u)}} > \delta \right)\\
\nonumber&\geq 
\P \left( \sum_{k=0}^\infty e^{S_k} > \frac{u^{1-\beta}}{v}+2\delta \right) - 
\delta^{-1} \E \left[ \sum_{k=0}^\infty e^{S_n} - \sum_{k=0}^\infty e^{\underline S_n^{(u)}} \right]\\
\nonumber&= 
\P \left( \sum_{k=0}^\infty e^{S_k} > \frac{u^{1-\beta}}{v}+2\delta \right) - \delta^{-1} \left( \sum_{k=0}^\infty (\E A_1)^k - \sum_{k=0}^\infty (\E (A_1 - u^{-\gamma}|B_1^*|)^+)^k \right)\\
\nonumber&= 
\P \left( \sum_{k=0}^\infty e^{S_k} > \frac{u^{1-\beta}}{v}+2\delta \right)
 - \delta^{-1} \left( \frac{1}{1-\E A_1} - \frac{1}{1-\E (A_1 - u^{-\gamma}|B_1|)^+} \right)
 \end{align}\begin{align}
\nonumber&\geq 
\P \left( \sum_{k=0}^\infty e^{S_k} > \frac{u^{1-\beta}}{v}+2\delta \right) - \delta^{-1} \frac{\E A_1 - \E (A_1 - u^{-\gamma}|B_1|)^+}{(1-\E A_1)(1-\E (A_1 - u^{-\gamma}|B_1|)^+)}\\ 
&\geq 
\P \left( \sum_{k=0}^\infty e^{S_k} > \frac{u^{1-\beta}}{v}+2\delta \right) - \delta^{-1} \frac{u^{-\gamma}\E |B_1|}{(1-\E A_1)(1-\E (A_1 - u^{-\gamma}|B_1|)^+)}.\nonumber
\end{align}
Again from \eqref{eq: tail asymptotics on the stationary distribution of AR(1) process} along with the assumption that $(1-\beta)\alpha < \gamma$, we get
\begin{equation}
\P \left( \sum_{k=0}^\infty e^{\underline S_k^{(u)}} > \frac{u^{1-\beta}}{v}+\delta \right)
\geq C_\infty u^{-(1-\beta)\alpha}v^\alpha + o(u^{-(1-\beta)\alpha}).
\label{eq: lemma 4.2 eq10}
\end{equation}
In view of \eqref{eq: lemma 4.2 eq08}--\eqref{eq: lemma 4.2 eq10}, we have that 
\[
\liminf_{u\to\infty} u^{(1-\beta)\alpha} \P \left( \sum_{k=T(u^\beta)}^{K_\beta^\gamma(u)-1} X_k > u \,\middle|\, \frac{X_{T(u^\beta)}}{u^\beta}=v \right) \geq C_\infty v^\alpha.
\]
Combining this with \eqref{eq: lemma 4.2 eq04} we have that 
\[
\lim_{u\to\infty} u^{(1-\beta)\alpha} \P \left( \sum_{k=T(u^\beta)}^{K_\beta^\gamma(u)-1} X_k > u \,\middle|\, \frac{X_{T(u^\beta)}}{u^\beta} \right) \left( \frac{X_{T(u^\beta)}}{u^\beta} \right)^{-\alpha} = C_\infty,
\]
$\P^\alpha$-almost surely. 

Next we show  boundedness of $\mathscr G_+(u)$. 
Using \eqref{eq: lemma 4.2 eq04}, for $\epsilon>0$, and by separately considering $v\leq 1+\epsilon$ and $v\geq 1+\epsilon$, there exists $U(\epsilon)$ (independent of $v$) such that
\[
\left(\frac{u^{(1-\beta)}}{v}\right)^\alpha \P \left( \sum_{k=T(u^\beta)}^{K_\beta^\gamma(u)-1} X_k > u \,\middle|\, \frac{X_{T(u^\beta)}}{u^\beta} = v \right) \leq C_\infty + \epsilon, 
\]
for all $u^{(1-\beta)}\geq v U(\epsilon)$. 
Moreover, for all $0<u^{(1-\beta)}<v U(\epsilon)$,
\begin{equation}\label{eq: lemma 4.2 eq11}
u^{(1-\beta)\alpha} \P \left( \sum_{k=T(u^\beta)}^{K_\beta^\gamma(u)-1} X_k > u \,\middle|\, \frac{X_{T(u^\beta)}}{u^\beta} = v \right) \leq u^{(1-\beta)\alpha} \leq v^\alpha U(\epsilon)^\alpha.
\end{equation}
Thus 
\[
u^{(1-\beta)\alpha} \P \left( \sum_{k=T(u^\beta)}^{K_\beta^\gamma(u)-1} X_k > u \,\middle|\, \frac{X_{T(u^\beta)}}{u^\beta} = v \right) \leq \max \{ C_\infty + \epsilon, U(\epsilon)^\alpha \} v^\alpha, 
\]
for all $u>0$. 
This implies that $\mathscr G_+(u) \leq \max \{ C_\infty + \epsilon, U(\epsilon)^\alpha \}<\infty$. 

Finally, we show the statements involved with $\mathscr G_-$. 
By the Markov property, it is sufficient to show that, for any arbitrary $\epsilon>0$ and $v< -1$
\[
\varlimsup_{u\to\infty} 
u^{(1-\beta)\alpha} 
\P \left( \sum_{k=0}^{K_\beta^\gamma(u)-1} X_k > u \,\middle|\, \frac{X_0}{u^\beta}=v \right) 
\leq \epsilon |v|^\alpha.
\]
Recall 
$
\mathfrak E_2(u) = \{ | B_n | \leq u^\gamma,\forall 1\leq n< K_\beta^\gamma(u) \}
$
defined in \eqref{eq: E_2^gamma(u)}.
We have that
\begin{align*}
\P \left( \sum_{k=0}^{K_\beta^\gamma(u)-1} X_k > u \,\middle|\, \frac{X_0}{u^\beta}=v \right)
&\leq 
\P \left( \sum_{k=0}^{K_\beta^\gamma(u)-1} X_k > u,\, \mathfrak E_2(u) \,\middle|\, \frac{X_0}{u^\beta}=v \right)
+
\P \left( \mathfrak E_2(u)^c \,\middle|\, X_0=v u^\beta \right)\\
&= 
\P \left( \mathfrak E_2(u)^c \,\middle|\, X_0=v u^\beta \right) = o( u^{-(1-\beta)\alpha} ) |v|,
\end{align*}
thanks to Lemma~\ref{lem: bounding sign-switching event}. 
The boundedness of $\mathscr G_u^-$ follows using similar arguments as in \eqref{eq: lemma 4.2 eq11}. 
\end{proof}

\begin{remark}\label{rmk: tail asymptotics w.r.t. |X_n|}
Using similar arguments as in the proof of Lemma~\ref{lem: almost sure convergence of G+ and G-}, one can show that 
\[
\lim_{u\to\infty}  u^{(1-\beta)\alpha} \P^\mathscr D \left( \sum_{n=T(u^\beta)}^{K_\beta^\gamma(u)-1} |X_n| > u \,\middle|\, \mathcal F_{T(u^\beta)} \right) 
\left| \frac{X_{T(u^\beta)}}{u^\beta} \right|^{-\alpha} = C_\infty. 
\]
As a consequence of this result, we have that 
\[
u^\alpha \P ( \sum_{n=0}^{r_1-1} |X_n| > u ) \to C_\infty \E^\alpha [|Z|^\alpha \1_{\{r_1=\infty\}}], 
\hspace{1cm} u\rightarrow\infty.
\]
\end{remark}

\linkdest{proof of lem: integrability of Z bar}
\begin{proof}[Proof of Lemma~\ref{lem: integrability of Z bar}]
\linksinpf{lem: integrability of Z bar}
Note that $Z_{T(u^\beta)}^+\1_{\{T(u^\beta)<\tau_d\}}$ and $Z_{T(u^\beta)}^-\1_{\{T(u^\beta)<\tau_d\}}$ are bounded by $|Z_{T(u^\beta)}\1_{\{T(u^\beta)<\tau_d\}}|$, for which we have that 
\[
| Z_{T(u^\beta)} \1_{\{T(u^\beta)<\tau_d\}} |
\leq 
|X_0| + \sum_{n=1}^{T(u^\beta)} |B_n| e^{-S_n} \1_{\{T(u^\beta)<\tau_d\}}
\leq 
|X_0| + \sum_{n=1}^\infty |B_n| e^{-S_n} \1_{\{n<\tau_d\}} = \bar Z. 
\]
Moreover, using Minkowski's inequality we have that 
\begin{align*}
\big(\E^\alpha [\bar Z^\alpha]\big)^{1/\alpha}
&\leq 
\big(\E |X_0|^\alpha\big)^{1/\alpha}
+ 
\sum_{n=1}^\infty \big(\E^\alpha \left[ |B_n|^\alpha e^{-\alpha S_n} \1_{\{n<\tau_d\}} \right]\big)^{1/\alpha}
\\
&= 
\big(\E |X_0|^\alpha\big)^{1/\alpha}
+ 
\sum_{n=1}^\infty \big(\E \left[ |B_n|^\alpha \1_{\{n<\tau_d\}} \right]\big)^{1/\alpha}\\
&\leq 
\big(\E |X_0|^\alpha\big)^{1/\alpha}
+ 
(\E |B_1|^{\alpha+\epsilon} )^{1/(\alpha+\epsilon)} \sum_{n=1}^\infty \P (\tau_d>n)^{\epsilon /(\alpha(\alpha+\epsilon))}<\infty,
\end{align*}
where in the second last inequality we used H\"older's inequality, and the finiteness follows from the fact that
$\P(\tau_d>n)$ decays exponentially in $n$ uniformly in $|X_0| \leq d$, as established in Lemma~\ref{lem: exponential decay of regeneration time}. 
\end{proof}

\linkdest{proof of lem: bound on moments of regeneration time}
\begin{proof}
[Proof of Lemma~\ref{lem: bound on moments of regeneration time}]
\linksinpf{lem: bound on moments of regeneration time}
Recall that 
$\tau = \inf\{ n\geq1 \colon Y_n \leq d\}$.
Using Minkowski's inequality we obtain that
\begin{align*}
&
\E [ r^{\alpha+L} \,|\, Y_0=x ]^{1/(\alpha+L)}
= 
\E [ (\tau + r - \tau)^{\alpha+L} \,|\, Y_0=x ]^{1/(\alpha+L)}\\ 
&\leq 
\E [ \tau^{\alpha+L} \,|\, Y_0=x ]^{1/(\alpha+L)} 
+ 
\E [ (r - \tau)^{\alpha+L} \,|\, Y_0=x ]^{1/(\alpha+L)}\\ 
&\leq 
\E [ \tau^{\alpha+L} \,|\, Y_0=x ]^{1/(\alpha+L)} 
+ 
\sup_{y\in[0,d]} \E [ r^{\alpha+L} \,|\, Y_0=y ]^{1/(\alpha+L)}\\
&\leq 
\E [ \tau^{\alpha+L} \,|\, Y_0=x ]^{1/(\alpha+L)} 
+ 
\sup_{y\in[0,d]} \E [ t^r \,|\, Y_0=y ]^{1/(\alpha+L)} 
+ 
\mathcal O(1),
\end{align*}
as $x\to\infty$, 
where, by following the arguments as in the proof of Lemma~\ref{lem: exponential decay of regeneration time}, $t$ can be chosen such that $\sup_{y\in[0,d]} \E [ t^r \,|\, Y_0=y ]<\infty$. 
For this choice of $t$, we have that 
\begin{align*}
\E [ r^{\alpha+L} \,|\, Y_0=x ]^{1/(\alpha+L)}
&\leq 
\E [ \tau^{\alpha+L} \,|\, Y_0=x ]^{1/(\alpha+L)} 
+ 
\mathcal O(1),\qquad\mbox{as}\ x\to\infty.
\end{align*}
Finally, using Lemma~\ref{lem: bound on moments of first return time} above we have  
$
\E [ r^{\alpha+L} \,|\, Y_0=x ] \leq c x^{\lfloor \alpha-\epsilon \rfloor}
$ for sufficiently large $x$.
\end{proof}

\linkdest{proof of lem: asymptotics of P(T(u^beta))<r_1}
\begin{proof}[Proof of Lemma~\ref{lem: asymptotics of P(T(u^beta))<r_1}]
\linksinpf{lem: asymptotics of P(T(u^beta))<r_1}
Note that both $|Z_{T(u)}^\alpha| \1_{\{T(u) < r_1\}}$ and $|Z_n^\alpha| \1_{\{n \leq r_1\}}$ are bounded by 
\[
\bar Z = |X_0| + \sum_{n=1}^\infty |B_n| e^{-S_n} \1_{\{n<r_1\}},
\] 
whose $\alpha$-th moment is finite thanks to Lemma~\ref{lem: integrability of Z bar}. 
Moreover, note that $\{X_n\}_{n\geq0}$ is transient in the $\alpha$-shifted measure (cf.\ Lemma~\ref{lem: asymptotics of X_n and Z_n} above), and hence, $T(u)<\infty$ a.s. 
Applying a change of measure argument, we obtain that 
\begin{align*}
u^{\alpha}\P ( T(u) < r_1)
&= 
u^{\alpha} \E^\alpha [ e^{-\alpha S_{T(u)}} \1_{\{T(u) < r_1\}} ]
= 
\E^\alpha \left[ |Z_{T(u)}|^\alpha \left| \frac{X_{T(u)}}{u} \right|^{-\alpha} \1_{\{T(u) < r_1\}} \right]\\
&= 
\E^\alpha \left[ |Z_n|^\alpha \1_{\{n\leq T(u)\}} \1_{\{n \leq r_1\}} \E^\alpha \left[ \left| \frac{X_{T(u)}}{u} \right|^{-\alpha} \,\middle|\, \mathcal F_n \right] \right]\\
&\hspace{12.5pt}+
\E^\alpha \left[ \left( |Z_{T(u)}|^\alpha \1_{\{T(u) < r_1\}} - |Z_n|^\alpha \1_{\{n\leq T(u)\}} \1_{\{n \leq r_1\}} \right) \left| \frac{X_{T(u)}}{u} \right|^{-\alpha} \right]\\
&= 
\textbf{(II.1)} + \textbf{(II.2)},
\end{align*}
where $\{\mathcal F_n\}_{n\geq0}$ is the natural filtration.
Since $(X_{T(u)}/u)^{-\alpha}\leq1$ and $T(u)\to\infty$ a.s.\ as $u\to\infty$, 
\begin{align*}
\nonumber
\lim_{n\to\infty} \lim_{u\to\infty} \textbf{(II.2)}
&\leq \notag
\lim_{n\to\infty} \lim_{u\to\infty} 
\E^\alpha \left[ |Z_{T(u)}|^\alpha \1_{\{T(u) < r_1\}} - |Z_n|^\alpha \1_{\{n\leq T(u)\}} \1_{\{n \leq r_1\}} \right]\\
\nonumber&= 
\lim_{n\to\infty} 
\E^\alpha \left[ \lim_{u\to\infty} \left( |Z_{T(u)}|^\alpha \1_{\{T(u) < r_1\}} - |Z_n|^\alpha \1_{\{n\leq T(u)\}} \1_{\{n \leq r_1\}} \right) \right]\\
&= 
\lim_{n\to\infty} 
\E^\alpha \left[ |Z|^\alpha \1_{\{r_1=\infty\}} - |Z_n|^\alpha \1_{\{n \leq r_1\}} \right]
= 0.
\end{align*}
It remains to consider \textbf{(II.1)}. 
Note that, given $\mathcal F_n$, $n\leq T(u)$, the random variable $\log |X_{T(u)}| - \log u$ converges in distribution to some positive random variable $\mathfrak X$---which is independent of $\mathcal F_n$, $n\leq T(u)$---as $u\to\infty$, under the $\alpha$-shifted measure (cf.\ e.g.\ Theorem 3.8 of \cite{collamore2016}). 
Hence we have that 
\begin{align*}
\lim_{u\to\infty} \E^\alpha \left[ \left| \frac{X_{T(u)}}{u} \right|^{-\alpha} \,\middle|\, \mathcal F_n \right] \1_{\{n\leq T(u)\}} 
= 
\E \left[ e^{-\alpha \mathfrak X} \right].
\end{align*}
Moreover, using  dominated convergence and the fact 
$
\1_{\{n\leq T(u)\}} \E^\alpha [ | X_{T(u)}/u |^{-\alpha} \,|\, \mathcal F_n ] \leq 1,
$ 
we obtain 
\begin{align*}
\lim_{n\to\infty} \lim_{u\to\infty} \textbf{(II.1)}
&= 
\lim_{n\to\infty} \E^\alpha \left[ |Z_n|^\alpha \1_{\{n \leq r_1\}} \lim_{u\to\infty} \E^\alpha \left[ \left| \frac{X_{T(u)}}{u} \right|^{-\alpha} \,\middle|\, \mathcal F_n \right] \1_{\{n\leq T(u)\}} \right]\\ 
&= 
\E \left[ e^{-\alpha \mathfrak X} \right] \lim_{n\to\infty} \E^\alpha \left[ |Z_n|^\alpha \1_{\{n \leq r_1\}} \right]
= 
\E \left[ e^{-\alpha \mathfrak X} \right] \E^\alpha \left[ |Z|^\alpha \1_{\{r_1=\infty\}} \right],
\end{align*}
completing the proof
\end{proof}

\begin{proof}[Proof of Lemma~\ref{lem: bounding sign-switching event}]
\linksinpf{lem: bounding sign-switching event}
To begin with, 
we write, for some $\delta>0$,
\begin{align*}
\P ( (\mathfrak E_2(u))^c | X_0 = v u^\beta ) 
&= 
\P ( \exists n< K_\beta^\gamma(u) \colon |B_n| > u^\gamma | X_0 = v u^\beta )\\
&\leq 
\P ( \exists n< \tau_d \colon |B_n| > u^\gamma | X_0 = v u^\beta )\\
&\leq 
\P ( \exists n\leq u^{\delta} \colon |B_n| > u^\gamma )
+ 
\P ( \tau_d \geq u^{\delta} | X_0 = v u^\beta )\\ 
&= 
\textbf{(III.1)} + \textbf{(III.2)}. 
\end{align*}
To bound \textbf{(III.1)}, we have that 
\begin{align*}
\textbf{(III.1)} 
&\leq 
u^{\delta} \P( |B_1| > u^\gamma ) 
\leq 
u^{\delta-\alpha\gamma} \E |B_1|^\alpha = o(u^{-(1-\beta)\alpha}),
\end{align*}
for \linkdest{cond-one-minus-beta-alpha-plus-delta-minus-alpha-gamma-less-than-zero}$(1-\beta)\alpha + \delta - \alpha\gamma<0$. 
Since
$
\textbf{(III.2)} 
\leq 
u^{-(\alpha+L)\delta} \E [ \tau_d^{\alpha+L} | X_0 = v u^\beta ]
$, 
it is sufficient to bound $\E [ \tau_d^{\alpha+L} | X_0 = v u^\beta ]$. 
Recall $\{Y_n\}_{n\geq0}$ is the $\mathbb R_+$-valued Markov chain defined by $Y_{n+1} = A_{n+1}Y_n + |B_{n+1}|$, for $n\geq0$, and $\tau = \inf\{ n\geq1 \colon Y_n \leq d\}$. 
Note that 
$
\E [ \tau_d^{\alpha+L} | X_0 = v u^\beta ] 
\leq 
\E [ \tau^{\alpha+L} | Y_0 = |v|u^\beta ]
$. 
Combining this with Lemma~\ref{lem: bound on moments of first return time}, we conclude that there exist $c$ and $u_0$ such that 
\begin{align*}
\textbf{(III.2)} 
&\leq 
u^{-(\alpha+L)\delta} \E [ \tau^{\alpha+L} | Y_0 = |v|u^\beta ] 
\leq 
c |v| u^\beta u^{-(\alpha+L)\delta},\qquad\forall u\geq u_0.
\end{align*}
Note that we can set $L = L(\delta,\alpha,\beta)$ to be arbitrarily large. 
Combining the estimates above for \textbf{(III.1)} and \textbf{(III.2)}, we conclude the proof. 
\end{proof}

\section{Proofs of Sections~\ref{SECTION: ONE-SIDED LARGE DEVIATIONS}~and~\ref{SECTION: TWO-SIDED LARGE DEVIATIONS}}\label{section: proofs of section 3.2 and 3.3}
Again, we briefly describe our proof strategy before diving into the technicalities. 
Define $\bar X_n' = \{ \bar X_n'(t), t\in[0,1] \}$, where
\begin{equation}\label{eq: X_n' bar}
\notationdef{nota-X-bar-n-prime}
{\bar X_n'(t)} = \frac{1}{n}\sum_{i=1}^{N(nt)} X_i'
\quad\mbox{and}\quad
\notationdef{nota-X-i-prime}
{X_i'} = \sum_{j=r_{i-1}}^{r_i-1} X_j,
\end{equation}
where $\{r_i\}_{i\geq0}$ is the sequence of regeneration times as in Remark~\ref{rmk: regeneration scheme}, and 
\begin{equation}\label{eq: N(nt)}
\notationdef{nota-N}{N(s)} = \sup \{ j\geq 0 \colon r_j-1 \leq s \}. 
\end{equation}
Thanks to Theorem~4.1 of \cite{rheeblanchetzwart2016} and Theorem~\ref{thm: tail estimates for the area under first return time/regeneration cycle} above, we are able to establish an asymptotic equivalence between $\bar X_n'$ and some random walk $\bar W_n$ that will be specified below. 
This allows us to provide a large deviations result for $\bar X_n'$, using Lemma \ref{lem: asymptotic equivalence}.
In both the one-sided and the two-sided case, we will show that the residual process $\bar X_n - \bar X_n'$ is negligible in an asymptotic sense. 

We state here three lemmata that will play key roles in the proofs of Theorems~\ref{thm: one-sided large deviations} and \ref{thm:two-sided-large-deviations}. 
Let $\bar W_n = \{ \bar W_n(t), t\in[0,1] \}$ be such that 
\begin{equation}\label{eq: auxilary random walk}
\bar W_n(t) = \frac{1}{n} \sum_{i=1}^{\lfloor nt/\E r_1 \rfloor} X_i',
\end{equation}
where $X_i'$ is as in \eqref{eq: X_n' bar}. 
We begin with stating an asymptotic equivalence between $\bar X_n'$ and $\bar W_n$, however, w.r.t.\ the $J_1$-topology, which is stronger than the $M_1'$-topology introduced in the beginning of Section~\ref{SECTION: ONE-SIDED LARGE DEVIATIONS}. 
Let $d_{J_1}$ denote the Skorokhod $J_1$ metric on $\mathbb D$, which is defined by
\[
d_{J_1}(\xi_1,\xi_2)=\inf_{\lambda\in\Lambda} ||\lambda-id||_\infty\vee||\xi_1\circ\lambda-\xi_2||_\infty,\qquad \xi_1,\xi_2\in\mathbb D,
\]
where $id$ denotes the identity mapping, $||\cdot||_\infty$ denotes the uniform metric, that is, 
$
\|x\|_\infty = \sup_{t\in[0,1]} |x(t)|
$, 
and $\Lambda$ denotes the set of all strictly increasing, continuous bijections from $[0,1]$ to itself. 
Moreover, for $j\geq 0$, define
\[
\mathbb D_{\leqslant j}^\mu = \{ \xi\in\underline{\mathbb D}_{\leqslant j}^\mu \colon \xi(0) = 0 \}
\quad\mbox{and}\quad
\mathbb D_{\ll j}^\mu = \{ \xi\in\underline{\mathbb D}_{\ll j}^\mu \colon \xi(0) = 0 \}.
\]

\begin{lemma}
\label{lem: LD for X_n' bar}
\linksinthm{lem: LD for X_n' bar}
Consider the metric space $(\mathbb D, d_{J_1})$. 
Suppose that Assumptions~\ref{ass: regularity condition AR(1) process}~and~\ref{ass: minorization AR(1) process} hold.
For any $j\geq 1$, the following holds.
\begin{enumerate}
\item If $B_1\geq 0$ and $C_+$ as in Theorem~\ref{thm: tail estimates for the area under first return time/regeneration cycle} is strictly positive, 
then the stochastic process $\bar X_n'$ is asymptotically equivalent to $\bar W_n$ w.r.t.\ $n^{-j(\alpha-1)}$ and $\mathbb D_{\leqslant j-1}^\mu$. 
\item If $C_+$ and $C_-$ as in Theorem~\ref{thm: tail estimates for the area under first return time/regeneration cycle} satisfy $C_+C_->0$, 
then the stochastic process $\bar X_n'$ is asymptotically equivalent to $\bar W_n$ w.r.t.\ $n^{-j(\alpha-1)}$ and $\mathbb D_{\ll j}^\mu$. 
\end{enumerate}
\end{lemma}

\begin{proof}
\linksinpf{lem: LD for X_n' bar}
We only show \textit{part 2)}, since \textit{part 1)} can be proved by a similar argument. 
By Lemma~\ref{lem: asymptotic equivalence}, it is sufficient to show, for any $\delta>0$ and $\gamma > 0$, 
\begin{align}
\notag
&\varlimsup_{n\to\infty} n^{j(\alpha-1)} \P(\bar X_n'\in (\mathbb D\setminus\mathbb D_{\ll j}^\mu)^{-\gamma}, d_{J_1}(\bar X_n', \bar W_n)\geq \delta)\\
&=
\varlimsup_{n\to\infty} n^{j(\alpha-1)} \P(\bar W_n\in (\mathbb D\setminus\mathbb D_{\ll j}^\mu)^{-\gamma}, d_{J_1}(\bar X_n', \bar W_n)\geq \delta)
=0. \label{eq: lemma 5.7 eq01}
\end{align}
To prove \eqref{eq: lemma 5.7 eq01}, it is convenient to consider the space of paths on a longer time horizon $[0,2]$. 
Let $\bar W_n$ denote the stochastic process $\{ \bar W_n (t), t \in [0, 2] \}$ over the time horizon $[0, 2]$, and $\mathbb D_{\ll j}^{\mu;[0,2]}$ denote the space of step functions on $[0, 2]$ that corresponds to $\mathbb D_{\ll j}^\mu$. 
Let $d_{J_1}^{[0,2]}$ denote the Skorokhod $J_1$ metric on $\mathbb D^{[0,2]} = \mathbb D([0, 2], \R)$. 
Note that $d_{J_1} (\bar W_n,\mathbb D_{\ll j}^\mu) \geq \gamma$ implies that $d_{J_1}^{[0,2]} (\bar W_n^{[0,2]},\mathbb D_{\ll j}^{\mu;[0,2]}) \geq \gamma$, 
and $d_{J_1} (\bar X_n',\mathbb D_{\ll j}^\mu) \geq \gamma$ implies that either $d_{J_1}^{[0,2]} (\bar W_n^{[0,2]},\mathbb D_{\ll j}^{\mu;[0,2]}) \geq \gamma$ or $2n/\E r_1 \leq N(n)$. 
Therefore, \eqref{eq: lemma 5.7 eq01} is implied by
\begin{equation}\label{eq: lemma 5.7 eq02}
\varlimsup_{n\to\infty} n^{j(\alpha-1)} \P(d_{J_1}^{[0,2]} (\bar W_n^{[0,2]},\mathbb D_{\ll j}^{\mu;[0,2]}) \geq \gamma, d_{J_1}(\bar X_n', \bar W_n)\geq \delta)
=0.
\end{equation}
To prove \eqref{eq: lemma 5.7 eq02}, we adopt the construction of a piecewise linear nondecreasing homeomorphism $\bar \lambda_n$ from \cite[the proof of Theorem 4.1]{rheeblanchetzwart2016}. 
Let $t_0 = 0$ 
and $t_i$ be the $i$-th jump time of $N(n\cdot)$ and $t_L$ be the last jump time of $N(n\cdot)$. 
Let $L = (\lfloor n/\E r_1 \rfloor - 1)\wedge N(n)$. 
Define $\bar \lambda_n$ in such a way that $\bar \lambda_n(t) = \E r_1 N(nt) / n$ on $t_0,\ldots,t_L$, $\bar \lambda_n(1) = 1$, and $\bar \lambda_n$ is a piecewise linear interpolation in between. 
For such $\bar \lambda_n$, $\bar W_n(\bar\lambda_n(t)) = \bar X_n'(t)$ for all $t\in[0,t_L]$, and hence, 
$\| \bar W_n \circ \bar\lambda_n - \bar X_n' \|_\infty = \sup_{t\in[t_L,1]} | \bar W_n \circ \bar\lambda_n (t) - \bar X_n'(t) |$. 
Therefore, 
\begin{align}
d_{J_1} ( \bar W_n, \bar X_n' ) 
&= \notag
\inf_{\lambda \in \Lambda} \| \lambda - id \|_\infty \vee \| \bar W_n \circ \lambda - \bar X_n' \|_\infty\leq \notag
\| \bar \lambda_n - id \|_\infty \vee \| \bar W_n \circ \bar\lambda_n - \bar X_n' \|_\infty\\
&= 
\| \bar \lambda_n - id \|_\infty \vee \sup_{t\in[t_L,1]} | \bar W_n \circ \bar\lambda_n (t) - \bar X_n'(t) |.\label{eq: lemma 5.7 eq03}
\end{align}
The second term can be bounded (with high probability) as follows.
For an arbitrary $\epsilon > 0$, consider two cases: 
$\lfloor n/\E r_1-n\epsilon \rfloor < N(n) < \lfloor n/ \E r_1 \rfloor$ and 
$\lfloor n/ \E r_1 \rfloor \leq N(n) \leq \lfloor n/\E r_1+n\epsilon \rfloor$. 
Set 
\[
W_n = \sum_{i=1}^{\lfloor n/\E r_1\rfloor} X_i'.
\]
If $\lfloor n/\E r_1-n\epsilon \rfloor < N(n) < \lfloor n/ \E r_1 \rfloor$, by the construction of $\bar \lambda_n$,
\begin{align}
\sup_{t\in[t_L,1]} | \bar W_n \circ \bar\lambda_n (t) - \bar X_n'(t) |
&\leq
\sup_{s,t\in[1-\epsilon,1]} | \bar W_n(s) - \bar W_n(t) |.\label{eq: lemma 5.7 eq04}
\end{align}
On the other hand, if $\lfloor n/ \E r_1 \rfloor \leq N(n) \leq \lfloor n/\E r_1+n\epsilon \rfloor$, 
\begin{align}
\sup_{t\in[t_L,1]} | \bar W_n \circ \bar\lambda_n (t) - \bar X_n'(t) |
&\leq
\sup_{s,t\in[1,1+\epsilon]} | \bar W_n(s) - \bar W_n(t) |.\label{eq: lemma 5.7 eq05}
\end{align}
From \eqref{eq: lemma 5.7 eq04} and \eqref{eq: lemma 5.7 eq05}, we see that on the event $\{ \lfloor n/\E r_1-n\epsilon \rfloor < N(n) \leq \lfloor n/\E r_1+n\epsilon \rfloor \}$,
\begin{equation}\label{eq: lemma 5.7 eq06}
\sup_{t\in[t_L,1]} | \bar W_n \circ \bar\lambda_n (t) - \bar X_n'(t) | 
\leq 
\sup_{s,t\in[1-\epsilon,1+\epsilon]} | \bar W_n(s) - \bar W_n(t) |.
\end{equation}
Using \eqref{eq: lemma 5.7 eq03} and \eqref{eq: lemma 5.7 eq06}, we obtain that
\begin{align}
& \notag
\P(d_{J_1}^{[0,2]} (\bar W_n^{[0,2]},\mathbb D_{\ll j}^{\mu;[0,2]}) \geq \gamma, d_{J_1}(\bar X_n', \bar W_n)\geq \delta)\\
&\leq \notag
\P\left( d_{J_1}^{[0,2]} (\bar W_n^{[0,2]},\mathbb D_{\ll j}^{\mu;[0,2]}) \geq \gamma, \sup_{s,t\in[1-\epsilon,1+\epsilon]} | \bar W_n(s) - \bar W_n(t) | \geq \delta \right) \\
&\hspace{12.5pt}+
\P ( \{ \lfloor n/\E r_1-n\epsilon \rfloor < N(n) \leq \lfloor n/\E r_1+n\epsilon \rfloor \}^c )
+ 
\P ( \| \bar \lambda_n - id \|_\infty \geq \delta ).\label{eq: lemma 5.7 eq07}
\end{align}

Thanks to Cram\'er's theorem, the second term in \eqref{eq: lemma 5.7 eq07} decays geometrically. 
Moreover, using that $\bar \lambda_n(t) = \E r_1 N(nt) / n$ on $t_0,\ldots,t_L$ (and linearly interpolated in between), 
we can write the last term in \eqref{eq: lemma 5.7 eq07} as
$\P(\|N(n\,\cdot\,)/n-\,\cdot\,/\E r_1\|_\infty>\delta)$, which converges to 0 in view of the functional law of large numbers for renewal
processes (see e.g.\ Theorem 5.10 of \cite{ChenYao}).

For the first term in \eqref{eq: lemma 5.7 eq07}, we have that (see \cite[page 21]{rheeblanchetzwart2016})
\[
\varlimsup_{n\to\infty}
n^{j(\alpha-1)} \P\Bigg( d_{J_1}^{[0,2]} (\bar W_n^{[0,2]},\mathbb D_{\ll j}^{\mu;[0,2]}) \geq \gamma, \sup_{s,t\in[1-\epsilon,1+\epsilon]} | \bar W_n(s) - \bar W_n(t) | \geq \delta \Bigg)
\leq c \epsilon
\]
for some $c > 0$, 
where the intuition behind the asymptotics above is that, given the rare event takes place, 
the random walk $\bar W_n^{[0,2]}$ must have $j$ big jumps and one of them has to occur in the time interval $[1-\epsilon,1+\epsilon]$. 
Since the choice of $\epsilon>0$ was arbitrary, \eqref{eq: lemma 5.7 eq01} is proved by letting $\epsilon\to0$. 
\end{proof}

The next two lemmata are useful for future purposes. 

\begin{lemma}
\label{lem: comparing M_1' metric and J_1 metric}
\linksinthm{lem: comparing M_1' metric and J_1 metric}
For $\xi,\zeta\in\mathbb D$, we have that $d_{M_1'}(\xi,\zeta) \leq d_{J_1}(\xi,\zeta)$. 
\end{lemma}
\begin{proof} 
As explained in Section 2.1 of \cite{vysotsky}, $d_{M_1'}(\xi,\zeta) \leq d_{M_1}(\xi,\zeta)$. Furthermore, by Theorem 12.3.2 of \cite{whitt2002},  $d_{M_1}(\xi,\zeta) \leq d_{J_1}(\xi,\zeta)$.
\end{proof}

Recall that $\mbox{Disc}(\xi)$ is the set of discontinuities of $\xi\in\mathbb D$ and was defined in \eqref{eq: discontinuities}. 

\begin{lemma}\label{lem: implication of convergence w.r.t. M_1' metric}
If $d_{M_1'}(\xi_n,\xi) \to 0$ as $n\to\infty$, then, for each $t\in \mbox{Disc}(\xi)^c$
\[
\lim_{\delta\downarrow0}\, 
\varlimsup_{n\to\infty}\, 
\sup_{t_1\in \mathscr B_\delta(t)\cap [0,1]} | \xi_n(t_1) - \xi(t_1) | = 0.
\]
\end{lemma}

\begin{proof}
Let $t\in\mbox{Disc}(\xi)^c$. 
We first prove the statement for the case where $t\in(0,1)$. 
Let $\epsilon>0$ be fixed. 
Choose $\delta = \delta(\epsilon) >0$ such that 
\begin{equation}\label{eq: lemma 6.3 eq01}
|\xi(t_1)-\xi(t)|<\epsilon,\qquad\mbox{for}\ t_1\in\mathscr B_\delta(t)\subseteq(0,1).
\end{equation}
By the definition of $M_1'$ convergence, for given $\epsilon$, there exists $n_0$, such that 
$d_{M_1'}(\xi_n,\xi) < (\delta \wedge \epsilon)/8$ for all $n \geq n_0$. 
Moreover, for each fixed $n \geq n_0$, one can find $(u_n,v_n)\in\Gamma'(\xi_n)$ and $(u,v)\in\Gamma'(\xi)$ such that 
\begin{equation}\label{eq: lemma 6.3 eq02}
\| u_n - u \|_\infty \vee \| v_n - v \|_\infty < (\delta \wedge \epsilon)/4.
\end{equation}
Let $\underline s$, $s$, $\overline s$ be such that $v(\underline s) = t-\delta/2$, $v(s) = t$ and $v(\overline s) = t+\delta/2$. 
Moreover, by \eqref{eq: lemma 6.3 eq02} we have that $v_n(\underline s) <t-\delta/4$ and $v_n(\overline s)>t+\delta/4$. 
Thus, for all $t_1\in(t-\delta/4,t+\delta/4)$ 
there exists $s_n\in(\underline s,\overline s)$ such that $(u_n(s_n),v_n(s_n)) = (\xi_n(t_1),t_1)$. 
Combining this with \eqref{eq: lemma 6.3 eq01} and \eqref{eq: lemma 6.3 eq02}, we obtain that 
\begin{align*}
| \xi_n(t_1) - \xi(t_1) | 
&\leq 
| \xi_n(t_1) - \xi(t) | + | \xi(t_1) - \xi(t) |= 
| u_n(s_n) - u(s) | + | \xi(t_1) - \xi(t) |\\ 
&\leq 
| u_n(s_n) - u(s_n) | + | u(s_n) - u(s) | + \epsilon\\ 
&\leq 
(\delta\wedge\epsilon)/2 + \epsilon + \epsilon < 3\epsilon. 
\end{align*}
Finally, the case where $t\in\{0,1\}$ can be dealt with similarly. 
\end{proof}

The remainder of this section is split into two parts that deal with Theorems~\ref{thm: one-sided large deviations} and \ref{thm:two-sided-large-deviations}. 

\subsection{Proof of Theorem~\ref{thm: one-sided large deviations}}


We consider the case where $B_1$ is nonnegative. 
Let us give the ``roadmap'' of proving Theorem~\ref{thm: one-sided large deviations}. 

\begin{itemize}
\item In Corollary~\ref{corol: one-sided LD for X_n' bar} below we establish a sample-path large deviations result for the aggregated process $\bar X_n'$ (see \eqref{eq: X_n' bar} above) by considering a suitably defined random walk together with utilizing Theorem~4.1 of \cite{rheeblanchetzwart2016}. 
For the $\mathbb M$-convergence in Corollary~\ref{corol: one-sided LD for X_n' bar} we need Lemma~\ref{lem: closeness of D_<j} below. 
\item In Proposition~\ref{prop: one-sided asymptotic equivalence between X_n' and X_n} we show the asymptotic equivalence between the aggregated process $\bar X_n'$ and the original process $\bar X_n$. 
Again, one technical lemma, see Lemma~\ref{lem: help lemma for one-sided LD} below, is needed. 
\item Part 1) of Theorem~\ref{thm: one-sided large deviations} follows by combining Corollary~\ref{corol: one-sided LD for X_n' bar} with Proposition~\ref{prop: one-sided asymptotic equivalence between X_n' and X_n}. 
Part 2) is a direct consequence of part 1). 
\end{itemize}  

\begin{lemma}\label{lem: closeness of D_<j}
\linksinthm{lem: closeness of D_<j}
For all $j\geq0$ and all $z\in\R$, the set 
$\underline{\mathbb D}_{\leqslant j}^z$ is closed w.r.t.\ $(\mathbb D,d_{M_1'})$.
\end{lemma}

Recall that  $C_j^z$ was defined in \eqref{eq: C_j^z} for $z\in\R$. 

\begin{corollary}
\label{corol: one-sided LD for X_n' bar}
\linksinthm{corol: one-sided LD for X_n' bar}
Suppose that Assumptions~\ref{ass: regularity condition AR(1) process} and \ref{ass: minorization AR(1) process} hold. 
Moreover, let $B_1\geq 0$ and $C_+$ as in Theorem~\ref{thm: tail estimates for the area under first return time/regeneration cycle} be strictly positive. 
For any $j\geq 0$, 
\[
n^{j(\alpha-1)} \P(\bar X_n'\in\,\cdot\,)\to (C_+ \E r_1)^j C_j^{\mu}(\,\cdot\,),
\]
in $\mathbb M(\mathbb D\setminus\underline{\mathbb D}_{\leqslant j-1}^\mu)$ as $n\to\infty$. 
\end{corollary}

\begin{proposition}
\label{prop: one-sided asymptotic equivalence between X_n' and X_n}
\linksinthm{prop: one-sided asymptotic equivalence between X_n' and X_n}
Suppose that Assumptions~\ref{ass: regularity condition AR(1) process} and \ref{ass: minorization AR(1) process} hold. 
If $B_1\geq 0$ and $C_+$ as in Theorem~\ref{thm: tail estimates for the area under first return time/regeneration cycle} is strictly positive, 
then $\bar X_n$ is asymptotically equivalent to $\bar X_n'$ w.r.t.\ $(n\P(X_1'\geq n))^j$ and $\underline{\mathbb D}_{\leqslant j-1}^\mu$. 
\end{proposition}

\linkdest{proof of thm: one-sided large deviations}
\begin{proof}[Proof of Theorem~\ref{thm: one-sided large deviations}]
\linksinpf{thm: one-sided large deviations}
Part 1) follows by combining Corollary~\ref{corol: one-sided LD for X_n' bar} with Proposition~\ref{prop: one-sided asymptotic equivalence between X_n' and X_n}. 
Part 2) is a direct consequence of part 1). 
\end{proof}

\begin{proof}[Proof of Lemma~\ref{lem: closeness of D_<j}]
\linksinpf{lem: closeness of D_<j}
We give the proof for the case where $z=0$, while the proof for $z\neq 0$ follows using similar arguments. 
The statement is trivial for $\underline{\mathbb D}_{\leqslant 0}=\{0\}$; we focus on the case where $j\geq1$. 
Let $\xi_n,n\geq1$, be a sequence such that $\xi_n\in \underline{\mathbb D}_{\leqslant j}$, for all $n\geq1$, and $\lim_{n\to\infty} d_{M_1'}(\xi_n,\xi)=0$ for some $\xi\in\mathbb D$. 
Our goal is to prove that $\xi\in\underline{\mathbb D}_{\leqslant j}$. 
Note that by Lemma~\ref{lem: implication of convergence w.r.t. M_1' metric} above, for every $t\in \mbox{Disc}(\xi)^c\cup\{1\}$,
\begin{equation}\label{eq: lemma 6.4 eq01}
\lim_{n\to\infty} \xi_n(t) = \xi(t).
\end{equation} 
We first show that $\xi$ has at most $j$ discontinuity points. 
Assume that $| \mbox{Disc}(\xi) | \geq j+1$. 
Then there exists $0\leq t_{1,-} < t_{1,+} <\cdots< t_{j+1,-} <t_{j+1,+} \leq1$ such that $t_{i,-}$, $t_{i,+}\in \mbox{Disc}(\xi)^c\cup\{1\}$, and $|\xi(t_{i,-})-\xi(t_{i,+})| > 0$, for all $i\in\{1,\ldots,j+1\}$. 
By \eqref{eq: lemma 6.4 eq01}, there exists $N'$ such that $|\xi_{N'}(t_{i,-})-\xi_{N'}(t_{i,+})| > 0$ for all $i\in\{1,\ldots,j+1\}$. 
This leads to the contradiction that $|\mbox{Disc}(\xi_{N'})|\leq j$. 
Now let $\underline t < \overline t$ be two neighbouring discontinuity points of $\xi$. 
We claim that $\xi$ is constant on $(\underline t,\overline t)$. 
To see this, assume that the opposite statement holds. 
Then there exists $t_1<t_{j+2}$ such that $\underline t<t_1<t_{j+2}<\overline t$ and $\xi(t_1)\neq \xi(t_{j+2})$.
W.l.o.g.\ we assume that $\xi(t_1) < \xi(t_{j+2})$. 
Since $\xi$ is continuous on $(\underline t,\overline t)$, there exists $t_1<t_2<\cdots<t_{j+2}$ such that 
\begin{equation}\label{eq: lemma 6.4 eq02}
\xi(t_1)<\xi(t_2)<\cdots<\xi(t_{j+2}) 
\qquad\mbox{with}\quad
\epsilon' = \min_{i\in\{1,\ldots,j+1\}} \xi(t_{i+1}) - \xi(t_i).
\end{equation}
On the other hand, for any $\epsilon>0$, by \eqref{eq: lemma 6.4 eq01} there exists $N=N(\epsilon)$ such that 
\begin{equation}\label{eq: lemma 6.4 eq03}
\xi_N(t_i) \in (\xi(t_i)-\epsilon,\xi(t_i)+\epsilon), \qquad\mbox{for all}\ i\in\{1,\ldots,j+2\}.
\end{equation}
In view of \eqref{eq: lemma 6.4 eq02} and \eqref{eq: lemma 6.4 eq03}, by choosing $\epsilon<\epsilon'$ we conclude that $\xi_N$ has at least $j+1$ discontinuity points, which leads to the contradiction that  $|\mbox{Disc}(\xi_N)|\leq j$. 
Thus we conclude that $\xi$ is constant between any two neighbouring discontinuity points. 
Similarly one can show that $\xi(t^+)-\xi(t^-)>0$ for every $t\in\mbox{Disc}(\xi)$. 
\end{proof}

\linkdest{proof of corol: one-sided LD for X_n' bar}
\begin{proof}[Proof of Corollary~\ref{corol: one-sided LD for X_n' bar}]
\linksinpf{corol: one-sided LD for X_n' bar}
Note that $\underline{\mathbb D}_{\leqslant j}^\mu = \mathbb D_{\leqslant j}^\mu \cup \{ \xi\in\mathbb D \colon \xi(0)>0,\ \xi-\xi(0)\in\mathbb D_{\leqslant j-1}^\mu\}$.
In particular, $\mathbb D_{\leqslant j}^\mu \subseteq \underline{\mathbb D}_{\leqslant j}^\mu$.
Using Lemma~\ref{lem: comparing M_1' metric and J_1 metric}, Corollary~\ref{corol: one-sided LD for X_n' bar} is a consequence of Lemma~\ref{lem: LD for X_n' bar} and Theorem~4.1 in \cite{rheeblanchetzwart2016}. 
\end{proof}

The following lemma is essential in the proof of Proposition~\ref{prop: one-sided asymptotic equivalence between X_n' and X_n}.
Recall $\bar X_n'$ was defined in \eqref{eq: X_n' bar}.
Define 
\begin{equation}\label{eq: R_n}
R_n= \{R_n(t),t\in[0,1]\},
\quad\mbox{where}\ 
R_n(t) = \frac{1}{n} \sum_{i=r_{N(n)}}^{\lfloor nt \rfloor -1} X_i.
\end{equation}

\begin{lemma}
\label{lem: help lemma for one-sided LD}
\linksinthm{lem: help lemma for one-sided LD}
Suppose that Assumptions~\ref{ass: regularity condition AR(1) process}~and~\ref{ass: minorization AR(1) process} hold.
Moreover, let $B_1\geq 0$ and $C_+$ as in Theorem~\ref{thm: tail estimates for the area under first return time/regeneration cycle} be strictly positive. 
The following holds for any $\delta>0$, $\gamma>0$, and $j\geq 0$. 
\begin{enumerate}
\item First we have that 
\[
\P ( \bar X_n'\in (\mathbb D\setminus\underline{\mathbb D}_{\leqslant j-1}^\mu)^{-\gamma},R_n(1) \geq \delta ) 
= 
o ( (n\P(X_1'\geq  n))^{j+1} ),
\quad\mbox{as}\ n\to\infty.
\] 
\item Moreover, we have that 
\[
\P ( R_n\in(\mathbb D\setminus\underline{\mathbb D}_{\leqslant 1})^{-\gamma} )
= 
o ( (n\P(X_1'\geq  n))^j ), 
\quad\mbox{as}\ n\to\infty.
\]
\end{enumerate}
\end{lemma}


\linkdest{proof of prop: one-sided asymptotic equivalence between X_n' and X_n}
\begin{proof}[Proof of Proposition~\ref{prop: one-sided asymptotic equivalence between X_n' and X_n}]
\linksinpf{prop: one-sided asymptotic equivalence between X_n' and X_n}
To begin with, for $\epsilon>0$, define
\begin{equation}\label{eq: E_3^epsilon(n)}
\mathfrak E_3^\epsilon(n)
= \{ 
N_\epsilon^-(n)
< N(n) \leq 
N_\epsilon^+(n)
\},
\end{equation}
where $N_\epsilon^-(n) = \lfloor n/\E r_1-n\epsilon \rfloor$ and $N_\epsilon^+(n) = \lfloor n/\E r_1+n\epsilon \rfloor$. 
Using Cram\'er's theorem, it is easy to see that $\P(\mathfrak E_3^\epsilon(n)^c)$ decays exponentially to $0$ as $n\to\infty$.
Defining $\Delta_i = r_i-r_{i-1}$, we have that 
\begin{equation}\label{eq: proposition 6.1 eq01}
\{d_{M_1'}(\bar X_n,\bar X_n')\geq 2\delta\} \subseteq \left\{ \exists\, i\leq N(n)\ s.t.\ \Delta_i\geq n\delta \right\} \cup \left\{ R_n(1) \geq \delta \right\}. 
\end{equation}
First we show that for any $j\geq0$, $\delta>0$, and $\gamma>0$, 
\[
\lim_{n\to\infty} (n\P(X_1'\geq n))^{-j} \P(\bar X_n'\in(\mathbb D\setminus\underline{\mathbb D}_{\leqslant j-1}^\mu)^{-\gamma},d_{M_1'}(\bar X_n,\bar X_n')\geq 2\delta) = 0.
\]
By \eqref{eq: proposition 6.1 eq01} we have that
\begin{align}
\nonumber&\P(\bar X_n'\in(\mathbb D\setminus\underline{\mathbb D}_{\leqslant j-1}^\mu)^{-\gamma},d_{M_1'}(\bar X_n,\bar X_n')\geq 2\delta)\\
\nonumber&\leq 
\P(\exists\, i\leq N(n)\ s.t.\ \Delta_i\geq n\delta)
+
\P ( \bar X_n'\in(\mathbb D\setminus\underline{\mathbb D}_{\leqslant j-1}^\mu)^{-\gamma},\,R_n(1) \geq \delta )\\ 
\label{eq: proposition 6.1 eq02}&=
\P(\exists\, i\leq N(n)\ s.t.\ \Delta_i\geq n\delta) + o ( (n\P(X_1'\geq n))^j ),
\end{align}
where in \eqref{eq: proposition 6.1 eq02} we used Lemma~\ref{lem: help lemma for one-sided LD}~(1) above. 
It remains to analyze the first term in \eqref{eq: proposition 6.1 eq02}. 
Note that 
\begin{align*}
\nonumber\P( \exists\, i\leq N(n)\ s.t.\ \Delta_i\geq n\delta ) &\leq
\P( \exists\, i\leq N(n)\ s.t.\ \Delta_i\geq n\delta,  \mathfrak E_3^\epsilon(n)) 
+ 
\P(\mathfrak E_3^\epsilon(n)^c)\\ 
\nonumber&=
\P( \exists\, i\leq N(n)\ s.t.\ \Delta_i\geq n\delta,  \mathfrak E_3^\epsilon(n)) 
+ 
o( (n\P(X_1'\geq n))^j )\\ 
\nonumber&\leq 
\P\left( \exists\, i\leq \lfloor n/\E \tau_1+n\epsilon \rfloor\ s.t.\ \Delta_i\geq n\delta \right) 
+ 
o( (n\P(X_1'\geq n))^j )\\ 
\nonumber&\leq 
\lfloor n/\E r_1+n\epsilon \rfloor \P(r_1\geq n\delta)
+ 
o( (n\P(X_1'\geq n))^j )\\ 
\label{eq: proposition 6.1 eq03}&= 
o( (n\P(X_1'\geq n))^j ),
\end{align*}
for any $j\geq0$. 
Next we show that 
\[
\lim_{n\to\infty} (n\P(X_1'\geq n))^{-j} \P(\bar X_n\in(\mathbb D\setminus\underline{\mathbb D}_{\leqslant j-1}^\mu)^{-\gamma},d_{M_1'}(\bar X_n,\bar X_n')\geq 2\delta) = 0.
\]
In view of the estimation right above, it is sufficient to show that
\[
\lim_{n\to\infty} (n\P(X_1'\geq n))^{-j} \P ( \bar X_n\in(\mathbb D\setminus\underline{\mathbb D}_{\leqslant j-1}^\mu)^{-\gamma},\, \bar X_n'\in (\underline{\mathbb D}_{\leqslant j-1}^\mu)_\rho,\, R_n(1) \geq \delta ) = 0,
\] 
for some $\rho>0$.
Note that 
\begin{align*}
&\P ( \bar X_n\in(\mathbb D\setminus\underline{\mathbb D}_{\leqslant j-1}^\mu)^{-\gamma},\, \bar X_n'\in (\underline{\mathbb D}_{\leqslant j-1}^\mu)_\rho,\, R_n(1) \geq \delta )\\
&=
\P ( \bar X_n\in(\mathbb D\setminus\underline{\mathbb D}_{\leqslant j-1}^\mu)^{-\gamma},\, \bar X_n' \in (\underline{\mathbb D}_{\leqslant j-1}^\mu)_\rho \cap (\mathbb D\setminus\underline{\mathbb D}_{\leqslant j-2}^\mu)^{-\rho},\, R_n(1) \geq \delta )\\
&\hspace{12.5pt}+ 
\P ( \bar X_n\in(\mathbb D\setminus\underline{\mathbb D}_{\leqslant j-1}^\mu)^{-\gamma},\, \bar X_n' \in (\underline{\mathbb D}_{\leqslant j-1}^\mu)_\rho \cap (\underline{\mathbb D}_{\leqslant j-2}^\mu)_\rho,\, R_n(1) \geq \delta )\\
&\leq 
\P ( \bar X_n' \in (\mathbb D\setminus\underline{\mathbb D}_{\leqslant j-2}^\mu)^{-\rho},\, R_n(1) \geq \delta )\\
&\quad+ 
\P ( \bar X_n\in(\mathbb D\setminus\underline{\mathbb D}_{\leqslant j-1}^\mu)^{-\gamma},\, \bar X_n' \in (\underline{\mathbb D}_{\leqslant j-2}^\mu)_\rho )\\
&=
\P ( \bar X_n\in(\mathbb D\setminus\underline{\mathbb D}_{\leqslant j-1}^\mu)^{-\gamma},\, \bar X_n' \in (\underline{\mathbb D}_{\leqslant j-2}^\mu)_\rho )
+ 
o(n^{-j(\alpha-1)}).
\end{align*}
Thus, it remains to consider the first term in the last equation. 
Combining Lemma~\ref{lem: help lemma for one-sided LD}~(2) above with the fact that 
\[
\P ( \bar X_n\in(\mathbb D\setminus\underline{\mathbb D}_{\leqslant j-1}^\mu)^{-\gamma},\, \bar X_n'\in (\underline{\mathbb D}_{\leqslant j-2}^\mu)_\rho ) 
\leq 
\P ( R_n\in(\mathbb D\setminus\underline{\mathbb D}_{\leqslant 1})^{-\rho} ) 
+ 
o( n^{-j(\alpha-1)} ),
\]
for $\rho$ small enough, we conclude the proof. 
\end{proof}


\linkdest{proof of lem: help lemma for one-sided LD}
\begin{proof}[Proof of Lemma~\ref{lem: help lemma for one-sided LD}]
\linksinpf{lem: help lemma for one-sided LD}
\textit{Part 1):} We start showing the first equivalence. 
Defining 
$\bar X_{\leqslant k,n}' = \{\bar X_{\leqslant k,n}'(t),t\in[0,1]\}$ 
by 
$\bar X_{\leqslant k,n}'(t) = 1/n \sum_{i=1}^{N(nt)\wedge k} X_i'$, 
we have that 
\begin{align}
\nonumber&\P ( \bar X_n'\in (\mathbb D\setminus\underline{\mathbb D}_{\leqslant j-1}^\mu)^{-\gamma}, R_n(1) \geq \delta )\\ 
\nonumber&\leq 
\P \left( \bar X_n'\in (\mathbb D\setminus\underline{\mathbb D}_{\leqslant j-1}^\mu)^{-\gamma},\, \sum_{i=r_{N(n)}}^{r_{N(n)+1}-1} X_i \geq n\delta,\, \mathfrak E_3^\epsilon(n) \right) 
+ 
\P( \mathfrak E_3^\epsilon(n)^c )\\ 
\nonumber&= 
\sum_{k=N_\epsilon^-(n)}^{N_\epsilon^+(n)} \P ( \bar X_n'\in (\mathbb D\setminus\underline{\mathbb D}_{\leqslant j-1}^\mu)^{-\gamma},\, X_{N(n)+1}' \geq n\delta,\, N(n)=k ) +
o ( (n\P(X_1'\geq  n))^{j+1} )\\
\nonumber&= 
\sum_{k=N_\epsilon^-(n)}^{N_\epsilon^+(n)} \P ( \bar X_{\leqslant k,n}'\in (\mathbb D\setminus\underline{\mathbb D}_{\leqslant j-1}^\mu)^{-\gamma},\, X_{k+1}' \geq n\delta,\, N(n)=k )
+ 
o ( (n\P(X_1'\geq  n))^{j+1} )\\ 
\nonumber&\leq 
\sum_{k=N_\epsilon^-(n)}^{N_\epsilon^+(n)} \P ( \bar X_{\leqslant k,n}'\in (\mathbb D\setminus\underline{\mathbb D}_{\leqslant j-1}^\mu)^{-\gamma},\, X_{k+1}' \geq n\delta )
+ 
o ( (n\P(X_1'\geq  n))^{j+1} )\\
\nonumber&= 
\sum_{k=N_\epsilon^-(n)}^{N_\epsilon^+(n)} \P ( \bar X_{\leqslant k,n}'\in (\mathbb D\setminus\underline{\mathbb D}_{\leqslant j-1}^\mu)^{-\gamma} ) 
\P \left( X_{k+1}' \geq n\delta \right)
+ 
o ( (n\P(X_1'\geq  n))^{j+1} )\\ 
\nonumber&\leq 
\P( X_1' \geq n\delta ) \sum_{k=N_\epsilon^-(n)}^{N_\epsilon^+(n)} \P ( \bar X_n'\in (\mathbb D\setminus\underline{\mathbb D}_{\leqslant j-1}^\mu)^{-\gamma/2} )
+ 
o ( (n\P(X_1'\geq  n))^{j+1} )\\
\label{eq: lemma 6.5 eq01}&\leq
2\epsilon n \P( X_1' \geq n\delta ) \P ( \bar X_n'\in (\mathbb D\setminus\underline{\mathbb D}_{\leqslant j-1}^\mu)^{-\gamma/2} )
+ 
o ( (n\P(X_1'\geq  n))^{j+1} ).
\end{align}
It remains to consider the first term in \eqref{eq: lemma 6.5 eq01}. 
Using Corollary~\ref{corol: one-sided LD for X_n' bar}, we have that 
\begin{equation}\label{eq: lemma 6.5 eq02}
\varlimsup_{n\to\infty} (n\P(X_1'\geq n))^{-(j+1)}\, 2\epsilon n \P( X_1' \geq n\delta ) \P ( \bar X_n'\in (\mathbb D\setminus\underline{\mathbb D}_{\leqslant j-1}^\mu)^{-\gamma/2} ) 
\leq 
c \epsilon, 
\end{equation}
for some $c>0$ independent of $\epsilon$. 
Part (1) is proved using \eqref{eq: lemma 6.5 eq01} and \eqref{eq: lemma 6.5 eq02}, and letting $\epsilon\to0$.

\textit{Part 2):} Note that 
\begin{align*}
&\P\left( R_n\in(\mathbb D\setminus\underline{\mathbb D}_{\leqslant 1})^{-\gamma} \right)
\\
&= 
\P\left( R_n\in(\mathbb D\setminus\underline{\mathbb D}_{\leqslant 1})^{-\gamma},\, \frac{r_{N(n)}+1}{n}  >  \rho \right)
\P\left( R_n\in(\mathbb D\setminus\underline{\mathbb D}_{\leqslant 1})^{-\gamma},\, \frac{r_{N(n)}+1}{n}\leq \rho \right)
\end{align*}
where the first term equals zero for sufficiently large $\rho\in(0,1)$. 
Hence, it is sufficient to consider the second term which is bounded by 
\begin{align}
\nonumber\P\left( \frac{r_{N(n)}+1}{n}\leq \rho \right) 
&\leq 
\P\left( r_{N(n)}\leq n\rho \right) 
\leq 
\P\left( r_{N(n)}\leq n\rho,\, \mathfrak E_3^\epsilon(n) \right) 
+ 
\P\left( \mathfrak E_3^\epsilon(n)^c \right)\\ 
\nonumber&= 
\P\left( \sum_{i=1}^{N(n)} \Delta_i \leq n\rho,\, \mathfrak E_3^\epsilon(n) \right) 
+ 
o ( (n\P(X_1'\geq n))^j )\\ 
\nonumber&\leq 
\P\left( \sum_{i=1}^{N_\epsilon^-(n)} \Delta_i \leq n\rho \right) 
+ 
o ( (n\P(X_1'\geq n))^j )\\ 
&\leq \label{eq: lemma 6.5 eq03}
\P\left(  \sum_{i=1}^{N_\epsilon^-(n)} \frac{\Delta_i}{N_\epsilon^-(n)} \leq \frac{\rho}{1/\E r_1-\epsilon} \right) 
+ 
o ( (n\P(X_1'\geq n))^j ).
\end{align}
Note that, for every $\rho\in(0,1)$ there exists a sufficiently small $\epsilon>0$ such that 
$
\rho/(1/\E r_1 - \epsilon) < \E r_1
$. 
For this choice of $\epsilon$, the first term in \eqref{eq: lemma 6.5 eq03} decays exponentially thanks to Cram\'er's theorem. 
\end{proof}

\subsection{Proof of Theorem~\ref{thm:two-sided-large-deviations}}
\linksinpf{thm:two-sided-large-deviations} 
We consider the case where $B_1$ is a general random variable taking values in $\R$. 
The idea behind the proof of Theorem~\ref{thm:two-sided-large-deviations} is similar to the one in the one-sided case: 

\begin{itemize}
\item In Corollary~\ref{corol: two-sided LD for X_n' bar} below we establish a sample-path large deviations result for the aggregated process $\bar X_n'$ (see \eqref{eq: X_n' bar} above). 
\item In Proposition~\ref{prop: two-sided asymptotic equivalence between X_n' and X_n} we show the asymptotic equivalence between the aggregated process $\bar X_n'$ and the original process $\bar X_n$. 
In Lemma~\ref{lem: help lemma for two-sided LD} we deal with the technical issues appearing in Proposition~\ref{prop: two-sided asymptotic equivalence between X_n' and X_n}. 
\item Part 1) of Theorem~\ref{thm:two-sided-large-deviations} follows by combining Corollary~\ref{corol: two-sided LD for X_n' bar} with Proposition~\ref{prop: two-sided asymptotic equivalence between X_n' and X_n}. 
Part 2) is a direct consequence of part 1). 
\end{itemize} 

\begin{lemma}
\label{lem: closeness of D_<<j}
\linksinthm{lem: closeness of D_<<j}
For all $j \geq 0$ and all $z\in\R$, the set $\underline{\mathbb D}_{\ll j}^z$ is closed w.r.t.\ $(\mathbb D,d_{M_1'})$.
\end{lemma}
The proof of Lemma \ref{lem: closeness of D_<<j} is similar to the proof of Lemma \ref{lem: closeness of D_<j} and therefore omitted. 

Recall $C_{j,k}^z$ was defined in \eqref{eq: C_j,k^z}. 
Let $C_+$, $C_-$ be as in Theorem~\ref{thm: tail estimates for the area under first return time/regeneration cycle}.

\begin{corollary}
\label{corol: two-sided LD for X_n' bar}
\linksinthm{corol: two-sided LD for X_n' bar}
Suppose that Assumptions~\ref{ass: regularity condition AR(1) process}~and~\ref{ass: minorization AR(1) process} hold. 
If $C_+C_->0$,
then for any $j\geq1$
\[
n^{j(\alpha-1)} \P(\bar X_n'\in\,\cdot\,) 
\to 
(\E r_1)^j \sum\nolimits_{(l,m)\in I_{= j}} (C_+)^l (C_-)^m C_{l,m}^{\mu}(\,\cdot\,),
\]
in $\mathbb M(\mathbb D\setminus\underline{\mathbb D}_{\ll j}^\mu)$ as $n\to\infty$, 
where $I_{= j} = \{ (l,m)\in\mathbb Z_+^2 \colon l+m=j \}$. 
\end{corollary}


\begin{proposition}
\label{prop: two-sided asymptotic equivalence between X_n' and X_n}
\linksinthm{prop: two-sided asymptotic equivalence between X_n' and X_n}
Suppose that Assumptions~\ref{ass: regularity condition AR(1) process}~and~\ref{ass: minorization AR(1) process} hold. 
If $C_+C_->0$, then the following hold for all $j\geq 0$:
\begin{enumerate}
\item First 
\[
\lim_{n\to\infty} n^{j(\alpha-1)} \P(\bar X_n'\in(\mathbb D\setminus\underline{\mathbb D}_{\ll j}^\mu)^{-\gamma},\, d_{M_1'}(\bar X_n,\bar X_n') > \delta) = 0. 
\]
\item Assume additionally that $\E |B_1|^m<\infty$ for every $m\in\mathbb Z_+$. 
Then 
\[
\lim_{n\to\infty} n^{j(\alpha-1)} \P(\bar X_n\in(\mathbb D\setminus\underline{\mathbb D}_{\ll j}^\mu)^{-\gamma},\, d_{M_1'}(\bar X_n,\bar X_n') > \delta) = 0. 
\] 
In particular, $\bar X_n$ is asymptotically equivalent to $\bar X_n'$ w.r.t.\ $n^{-j(\alpha-1)}$ and $\underline{\mathbb D}_{\ll j}^\mu$. 
\end{enumerate}
\end{proposition}

We need the following lemma to prove Proposition~\ref{prop: two-sided asymptotic equivalence between X_n' and X_n}. 
Set
\[
\notationdef{nota-R-p-n}
{R_{p,n}(t)} = \frac{1}{n} \sum_{i=r_p}^{\lfloor r_{p+1} t \rfloor - 1} X_i. 
\]
Let $\notationdef{nota-T-1}{T_1(u)} = T(u) = \inf \{ n\geq0 \colon |X_n| > u \}$ and 
\[
\notationdef{nota-T-i}
{T_{i+1}(u)} = \inf \{ n\geq T_i(u)\colon -\mbox{sign}(X_{T_i}(u)) X_n > u \}, 
\quad i\geq1. 
\]
Define $\bar X_{i,n}= \{\bar X_{i,n}(t),t\in[0,1]\}$ and $\bar X_{i,n}'= \{\bar X_{i,n}'(t),t\in[0,1]\}$ by
\begin{equation}\label{eq: proposition 6.2 eq01}
\notationdef{nota-X-bar-i-n}
{\bar X_{i,n}(t)} = \frac{1}{n} \sum_{l=r_{i-1}}^{\lfloor nt \rfloor \wedge r_i - 1} X_l,
\quad\mbox{and}\quad 
\linkdest{nota-X-bar-i-n-prime}
\bar X_{i,n}'(t) = \frac{X_i'}{n} \1_{[r_i/n,1]}(t).
\end{equation}
respectively. 

\begin{lemma}\label{lem: help lemma for two-sided LD}
\linksinthm{lem: help lemma for two-sided LD}
Suppose that Assumptions~\ref{ass: regularity condition AR(1) process}~and~\ref{ass: minorization AR(1) process} hold. 
Moreover, assume that $\E |B_1|^m<\infty$ for every $m\in\mathbb Z_+$. 
Let $C_+$, $C_-$ be as in Theorem~\ref{thm: tail estimates for the area under first return time/regeneration cycle} such that $C_+C_->0$. 
\begin{enumerate}
\item For any $i\geq1$, $j\geq2$, $\epsilon>0$, and $\delta>0$, there exists $c_1$, $c_2$ and $n_1$, $n_2$  (independent of $i$) respectively such that 
\begin{alignat*}{2}
&\P ( d_{M_1'}(\bar X_{i,n},\bar X_{i,n}')\geq \delta ) \leq  c_1 n^{-(2-\epsilon)\alpha},
\qquad 
&&\mbox{for all}\ n\geq n_1,\\
\mbox{and}\quad &\P ( \bar X_{i,n} \in (\mathbb D\setminus\underline{\mathbb D}_{\ll j})^{-\delta}) \leq c_2 n^{-(j-\epsilon)\alpha},
&&\mbox{for all}\ n\geq n_2.
\end{alignat*}
\item For any $j\geq1$, $\hat X_n$ is asymptotically equivalent to $\bar X_n'$ w.r.t.\ $n^{-j(\alpha-1)}$ and $\underline{\mathbb D}_{\ll j}^\mu$. 
\item For any $i\in \{ N_\epsilon^-(n),\ldots,N_\epsilon^+(n) \}$, $j\geq1$, $\delta > 0$, and $\epsilon > 0$, there exists $c$ and $n_0$ (independent of $i$) such that 
\[
\P ( R_{i,n} \in(\mathbb D\setminus\underline{\mathbb D}_{\ll j})^{-\delta} ) \leq c n^{-(j-\epsilon)\alpha},
\qquad\mbox{for all}\ n\geq n_0.
\]
\end{enumerate}
\end{lemma}

\begin{remark}\label{rmk: uniform convergence}
Without the additional assumption $\E |B_1|^m<\infty$ for every $m\in\mathbb Z_+$, one can still show that $\P ( T_2(n^\beta)<r_1 ) = o(n^{-\alpha})$, by following the arguments as in the proof of Lemma~\ref{lem: help lemma for two-sided LD}. 
Hence, under Assumptions~\ref{ass: regularity condition AR(1) process}~and~\ref{ass: minorization AR(1) process}, uniformly in $i$,
\[
\lim_{n\to\infty} n^\alpha \P ( d_{M_1'}(\bar X_{i,n},\bar X_{i,n}')\geq \delta ) = 0.
\] 
\end{remark}

\linkdest{proof of prop: two-sided asymptotic equivalence between X_n' and X_n}
\begin{proof}[Proof of Proposition~\ref{prop: two-sided asymptotic equivalence between X_n' and X_n}]
\linksinpf{prop: two-sided asymptotic equivalence between X_n' and X_n}
To begin with, recall that, for $\epsilon>0$
\[
\mathfrak E_3^\epsilon(n)  
= \{ 
N_\epsilon^-(n)
\leq N(n) \leq 
N_\epsilon^+(n)
\},
\]
where $N_\epsilon^-(n) = n\lfloor 1/\E r_1-\epsilon \rfloor$ and $N_\epsilon^+(n) = n\lfloor 1/\E r_1+\epsilon \rfloor$. 
Moreover, $\P( (\mathfrak E_3^\epsilon(n))^c )$ decays exponentially to $0$ as $n\to\infty$. 
Let $R_n$ be as in \eqref{eq: R_n}.
Before we prove the two statements of the proposition, we first prove that
\begin{equation}\label{eq: proposition 6.2 eq02}
\{d_{M_1'}(\bar X_n,\bar X_n') > \delta\}
\subseteq 
\left\{ \exists\, i\leq N(n)\ \mbox{s.t.}\ d_{M_1'}(\bar X_{i,n},\bar X_{i,n}')\geq \delta \right\} \cup \left\{ \| R_n \|_\infty \geq \delta \right\}.
\end{equation}
To see \eqref{eq: proposition 6.2 eq02}, we assume that the opposite statement holds. 
Given that the event $\{ d_{M_1'}(\bar X_{i,n},\bar X_{i,n}') < \delta \}$ takes place, there exist $(u_1^i,v_1^i)\in\Gamma'(\bar X_{i,n})$ and $(u_2^i,v_2^i)\in\Gamma'(\bar X_{i,n}')$ such that $\| u_1^i-u_2^i \|_\infty \vee \| v_1^i-v_2^i \|_\infty < \delta + \eta$. 
W.l.o.g.\ we assume that 
\begin{equation}\label{eq: proposition 6.2 eq03}
\{ s\colon v_1^i(s) = r_{i-1}/n,\, u_1^i(s) = 0 \} \cap \{ s\colon v_2^i(s) = r_{i-1}/n,\, u_2^i(s) = 0 \} \neq \emptyset,
\end{equation}
as well as
\[
\{ s\colon v_1^i(s) = r_i/n,\, u_1^i(s) = X_i'/n \} \cap \{ s\colon v_2^i(s) = r_i/n,\, u_2^i(s) = X_i'/n \} \neq \emptyset.
\]
We give here the reasoning for \eqref{eq: proposition 6.2 eq03}, where the second equation can be obtained by following the same arguments. 
Let $s_1 \in \{ s\colon v_1^i(s) = r_{i-1}/n,\, u_1^i(s) = 0 \}$ and $s_2 \in \{ s\colon v_2^i(s) = r_{i-1}/n,\, u_2^i(s) = 0 \}$. 
When $s_1=s_2$, we are done.
We assume $s_1<s_2$, otherwise one can change the role of $s_1$ and $s_2$. 
Define a new parametric representation $(\bar u_2^i,\bar v_2^i) \in \Gamma'(\bar X_{i,n}')$ by
\[
\bar v_2^i(s) 
= 
\begin{cases}
v_1(s),\quad&\mbox{for}\ s\in[0,s_1],\\
v_1(s_1),&\mbox{for}\ s\in(s_1,s_2),\\
v_2(s),&\mbox{for}\ s\in[s_2,1],
\end{cases}
\qquad
\bar u_2^i(s) 
= 
\begin{cases}
0,\quad&\mbox{for}\ s\in[0,s_1],\\
0,&\mbox{for}\ s\in(s_1,s_2),\\
u_2(s),&\mbox{for}\ s\in[s_2,1].
\end{cases}
\]
It is easy to check that indeed $(\bar u_2^i,\bar v_2^i)$ is a parametric representation of $\Gamma'(\bar X_{i,n}')$. 
Moreover, $\| u_1^i - \bar u_2^i \|_\infty = \| u_1^i - u_2^i \|_\infty < \delta + \eta$,
\begin{align*}
| v_1^i(s) - \bar v_2^i(s) | 
&= 
| v_1^i(s) - v_1^i(s_1) |
\leq  
v_1^i(s_2) - v_1^i(s_1)
= 
v_1^i(s_2) - v_2^i(s_2)
<
\delta + \eta,
\end{align*}
for $s\in(s_1,s_2)$, 
and hence, $\| v_1^i - \bar v_2^i \|_\infty < \delta + \eta$. 
In view of the construction above, we can replace $v_2^i$ by $\bar v_2^i$, so that \eqref{eq: proposition 6.2 eq03} holds. 
Similarly, on the event $\{ \| R_n \|_\infty < \delta \}\subseteq \{ d_{M_1'}(R_n,0) < \delta \}$, there exist $(u_1^{N(n)+1},v_1^{N(n)+1})\in\Gamma'(R_n)$ and $(u_2^{N(n)+1},v_2^{N(n)+1})\in\Gamma'(0)$ such that 
\[
\| u_1^{N(n)+1}-u_2^{N(n)+1} \|_\infty \vee \| v_1^{N(n)+1}-v_2^{N(n)+1} \|_\infty < \delta + \eta, 
\]
and the intersection of 
\[
\{ s\colon v_1^{N(n)+1}(s) = r_{N(n)}/n,\, u_1^{N(n)+1}(s) = 0 \}
\] 
and 
\[
\{ s\colon v_2^{N(n)+1}(s) = r_{N(n)}/n,\, u_2^{N(n)+1}(s) = 0 \}
\] 
is an empty set. 
Now, we pick $s_-^1 = 0$, $s_+^{N(n)+1} = 1$,
\[
s_+^i \in \{ s\colon v_1^i(s) = r_i/n,\, u_1^i(s) = X_i'/n \} \cap \{ s\colon v_2^i(s) = r_i/n,\, u_2^i(s) = X_i'/n \},
\]
for $i\in\{1,\ldots,N(n)\}$,
and 
\[
s_-^i \in \{ s\colon v_1^i(s) = r_i/n,\, u_1^i(s) = 0 \} \cap \{ s\colon v_2^i(s) = r_i/n,\, u_2^i(s) = 0 \},
\]
for $i\in\{2,\ldots,N(n)+1\}$. 
W.l.o.g.\ we assume that $s_+^i = s_-^{i+1}$; otherwise one can apply a strictly increasing, continuous bijection from $[0,1]$ to itself to the corresponding parametric representation, which preserves the uniform distance between parametric representations. 
Finally, we define parametric representations $(u_1,v_1)\in\Gamma'(\bar X_n)$ and $(u_2,v_2)\in\Gamma'(\bar X_n')$ by 
$
v_i(s) = v_i^j(s)
$, 
and 
$
u_i(s) = u_i^j(s) + \sum_{k=1}^{j-1} X_k'
$, 
for $s\in[s_-^j,s_+^j]$, $j\in\{1,\ldots,N(n)+1\}$, and $i\in\{1,2\}$. 
It is easy to check that $\| u_1-u_2 \|_\infty \vee \| v_1-v_2 \|_\infty < \delta + \eta$, 
and hence, $d(\bar X_n,\bar X_n') \leq \| u_1-u_2 \|_\infty \vee \| v_1-v_2 \|_\infty < \delta + \eta$. 
Letting $\eta\to0$ leads to the contradiction of $d_{M_1'}(\bar X_n,\bar X_n') > \delta$. We therefore conclude that (\ref{eq: proposition 6.2 eq02}) indeed holds. 

We are now ready to prove the two statements in the proposition. 

\textit{Part 1):}
For $\gamma>0$ and $j\geq 1$, define 
\begin{equation}\label{eq:D>=j,k}
\mathcal D_{\geqslant j}^{\gamma} 
= 
\{ \xi\in \mathbb D \colon |\mbox{Disc}_{\gamma}(\xi)| \geq j \},
\quad  
\mbox{Disc}_{\gamma}(\xi) = \{t\in\mbox{Disc}(\xi)\colon |\xi(t)-\xi(t^-)| \geq \gamma\}.
\end{equation}
Note that (cf.\ the proof of Lemma~2 in \cite{chen2019}), for any $L>0$, there exists a $\bar \gamma=\bar \gamma(\gamma,L)>0$ sufficiently small such that 
\begin{equation}\label{eq: proposition 6.2 eq04}
\P ( \bar X_n'\in (\mathbb D\setminus\underline{\mathbb D}_{\ll j}^\mu)^{-\gamma} \cap (\mathcal D_{\geqslant j}^{\bar \gamma}  )^c ) = o( n^{-L} ).
\end{equation}
Thus, it suffices to show that for any $j\geq1$ and any $\delta>0$ 
\[
\lim_{n\to\infty} n^{j(\alpha-1)} \P(\bar X_n'\in\mathcal D_{\geqslant j}^{\bar \gamma},\, d_{M_1'}(\bar X_n,\bar X_n')\geq 2\delta) = 0.
\]
By \eqref{eq: proposition 6.2 eq02} we have that
\begin{align*}
\P(\bar X_n'\in \mathcal D_{\geqslant j}^{\bar \gamma},\, d_{M_1'}(\bar X_n,\bar X_n')\geq 2\delta)
&\leq 
\P ( \bar X_n'\in \mathcal D_{\geqslant j}^{\bar \gamma},\, \exists\, i\leq N(n)\ s.t.\ d_{M_1'}(\bar X_{i,n},\bar X_{i,n}')\geq \delta )\\
&\hspace{12.5pt}+
\P ( \bar X_n'\in \mathcal D_{\geqslant j}^{\bar \gamma},\, \|R_n\|_\infty \geq \delta )= 
\textbf{(IV.1)} + \textbf{(IV.2)},
\end{align*}
where  
\[
\textbf{(IV.1)}
= 
\P ( \bar X_n'\in \mathcal D_{\geqslant j}^{\bar \gamma},\, \mathfrak E_3^{\epsilon}(n) ,\, \exists\, i\leq N(n)\ s.t.\ d_{M_1'}(\bar X_{i,n},\bar X_{i,n}')\geq \delta )
+ o( n^{-j(\alpha-1)} ).
\]
For $p\in\mathbb Z_+$, let $\mathcal P(E,p)$ denote the set of all $p$-permutations of a discrete set $E$. 
Using Lemma~\ref{lem: help lemma for two-sided LD}~(1) and the fact that the blocks $\{X_{r_{i-1}},\ldots,X_{r_i}\}$, $i\geq1$ are mutually independent, we obtain that
\begin{align*}
&\P ( \bar X_n'\in \mathcal D_{\geqslant j}^{\bar \gamma},\, \mathfrak E_3^{\epsilon}(n) ,\, \exists\, i\leq N(n)\ \mbox{s.t.}
d_{M_1'}(\bar X_{i,n},\bar X_{i,n}')\geq \delta )\\
&\leq 
\P ( \exists (i_1,\ldots,i_j)\in \mathcal P(\{1,\ldots,N_\epsilon^+(n)\},j) \ \mbox{ s.t. }
d_{M_1'}(\bar X_{i_1,n},\bar X_{i_1,n}')\geq \delta,\, |X_{i_p}'|\geq n\bar \gamma,\forall 2\leq p\leq j )\\
&= 
\mathcal O(n^j) \P ( d_{M_1'}(\bar X_{i_1,n},\bar X_{i_1,n}')\geq \delta ) \P ( |X_{i_p}'|\geq n\bar \gamma )^{j-1}= 
\mathcal O(n^j) o(n^{-\alpha}) \mathcal O(n^{-(j-1)\alpha}) 
= 
o(n^{-j(\alpha-1)}), 
\end{align*}
where $\P ( d_{M_1'}(\bar X_{i_1,n},\bar X_{i_1,n}')\geq \delta )$ is of order $o(n^{-\alpha})$ thanks to Remark~\ref{rmk: uniform convergence}. 
Recalling
$$
\bar X_{\leqslant m,n}' = \left\{ \frac1n \sum_{i=1}^{N(nt)\wedge m} X_i',\ t\in[0,1] \right\}
,$$
we have that 
\begin{align*}
\textbf{(IV.2)}
&\leq 
\P \left( \bar X_n'\in\mathcal D_{\geqslant j}^{\bar \gamma},\, \sum_{i=r_{N(n)}}^{r_{N(n)+1}-1} |X_i| \geq n\delta,\, \mathfrak E_3^\epsilon(n) \right) 
+ 
\P( \mathfrak E_3^\epsilon(n)^c )\\ 
&= 
\sum_{m=N_\epsilon^-(n)}^{N_\epsilon^+(n)} 
\P \left( \bar X_n'\in\mathcal D_{\geqslant j}^{\bar \gamma},\, \sum_{i=r_{N(n)}}^{r_{N(n)+1}-1} |X_i| \geq n\delta,\, N(n)=m \right)+
o ( n^{-j(\alpha-1)} )\\
&\leq 
\sum_{m=N_\epsilon^-(n)}^{N_\epsilon^+(n)} 
\P \left( \bar X_{\leqslant m,n}'\in\mathcal D_{\geqslant j}^{\bar \gamma},\, \sum_{i=r_m}^{r_{m+1}-1} |X_i| \geq n\delta \right) 
+ 
o ( n^{-j(\alpha-1)} )\\
&= 
\P \left( \sum_{i=0}^{r_1-1} |X_i| \geq n\delta \right) \sum_{m=N_\epsilon^-(n)}^{N_\epsilon^+(n)} 
\P ( \bar X_{\leqslant m,n}'\in\mathcal D_{\geqslant j}^{\bar \gamma} ) 
+ 
o ( n^{-j(\alpha-1)} )\\
&\leq 
\P \left( \sum_{i=0}^{r_1-1} |X_i| \geq n\delta \right) \sum_{m=N_\epsilon^-(n)}^{N_\epsilon^+(n)} 
\P ( \bar X_n'\in\mathcal D_{\geqslant j}^{\bar \gamma} ) 
+ 
o ( n^{-j(\alpha-1)} )\\
&\leq 
\P \left( \sum_{i=0}^{r_1-1} |X_i| \geq n\delta \right) 2\epsilon n 
\P ( \bar X_n'\in\mathcal D_{\geqslant j}^{\bar \gamma} ) 
+ 
o ( n^{-j(\alpha-1)} )\\
&= 
2\epsilon n \mathcal O(n^{-\alpha}) \mathcal O(n^{-j(\alpha-1)}) = o(n^{-j(\alpha-1)}),
\end{align*} 
where $\P ( \sum\nolimits_{i=0}^{r_1-1} |X_i| \geq n\delta )$ is of order $\mathcal O(n^{-\alpha})$ due to Remark~\ref{rmk: tail asymptotics w.r.t. |X_n|}. 

\textit{Part 2):}
In view of Part 1), it is sufficient to show that 
\[
\P ( \bar X_n\in(\mathbb D\setminus\underline{\mathbb D}_{\ll j}^\mu)^{-\gamma},\, \bar X_n' \in (\underline{\mathbb D}_{\ll j}^\mu)_{\rho/3} ) = o( n^{-j(\alpha-1)}), 
\]
for some $\rho>0$. 
Noting
$
\hat X_n(t) = (1/n) \sum_{i=0}^{(\lfloor nt \rfloor \wedge r_{N(n)})-1} X_i
$ for $t\in[0,1]$, 
we have that 
\begin{align*}
& \{ \bar X_n\in(\mathbb D\setminus\underline{\mathbb D}_{\ll j}^\mu)^{-\gamma},\, \bar X_n' \in (\underline{\mathbb D}_{\ll j}^\mu)_{\rho/3} \} \\
 \hspace{5pt} &\subseteq 
\{ \bar X_n' \in (\underline{\mathbb D}_{\ll j}^\mu)_{\rho/3},\, \hat X_n \in (\mathbb D\setminus\underline{\mathbb D}_{\ll j}^\mu)^{-\rho} \}\cup 
\{ \bar X_n\in(\mathbb D\setminus\underline{\mathbb D}_{\ll j}^\mu)^{-\gamma},\, \hat X_n \in (\underline{\mathbb D}_{\ll j}^\mu)_\rho \}\\
&\subseteq 
\{ \bar X_n' \in (\underline{\mathbb D}_{\ll j}^\mu)_{\rho/3},\, \hat X_n \in (\mathbb D\setminus\underline{\mathbb D}_{\ll j}^\mu)^{-\rho} \} \cup
\{ \bar X_n\in(\mathbb D\setminus\underline{\mathbb D}_{\ll j}^\mu)^{-\gamma},\, \hat X_n \in (\underline{\mathbb D}_{\ll j-1}^\mu)_\rho \}\\   
 & \hspace{15pt}\cup 
\{ \bar X_n\in(\mathbb D\setminus\underline{\mathbb D}_{\ll j}^\mu)^{-\gamma},\, \hat X_n \in (\underline{\mathbb D}_{\ll j}^\mu)_\rho \cap (\mathbb D\setminus\underline{\mathbb D}_{\ll j-1}^\mu)^{-\rho} \}.
\end{align*}
Iterating this procedure $j+k$ times, we obtain that 
\begin{align}
\nonumber&
\{ \bar X_n\in(\mathbb D\setminus\underline{\mathbb D}_{\ll j}^\mu)^{-\gamma},\, \bar X_n' \in (\underline{\mathbb D}_{\ll j}^\mu)_{\rho/3} \}\\
\nonumber&\subseteq
\{ \bar X_n' \in (\underline{\mathbb D}_{\ll j}^\mu)_{\rho/3},\, \hat X_n \in (\mathbb D\setminus\underline{\mathbb D}_{\ll j}^\mu)^{-\rho} \}
\cup 
\{ \bar X_n\in(\mathbb D\setminus\underline{\mathbb D}_{\ll j}^\mu)^{-\gamma},\, \hat X_n \in (\underline{\mathbb D}_{0}^\mu)_\rho \}\\
\label{eq: proposition 6.2 eq05}&\hspace{12.5pt}\cup 
\bigcup_{i=1}^{j+k-1} 
\{ \bar X_n\in(\mathbb D\setminus\underline{\mathbb D}_{\ll j}^\mu)^{-\gamma},\, \hat X_n \in (\underline{\mathbb D}_{\ll j+1-i}^\mu)_\rho \cap (\mathbb D\setminus\underline{\mathbb D}_{\ll j-i}^\mu)^{-\rho} \}.
\end{align}
Now, note that 
\begin{equation} \{ \bar X_n' \in (\underline{\mathbb D}_{\ll j}^\mu)_{\rho/3},\, \hat X_n \in (\mathbb D\setminus\underline{\mathbb D}_{\ll j}^\mu)^{-\rho} \}\subseteq 
\{ \hat X_n \in (\mathbb D\setminus\underline{\mathbb D}_{\ll j}^\mu)^{-\rho}, d_{M_1'} ( \bar X_n',\hat X_n) \geq \rho/3 \}. \label{eq: proposition 6.2 eq06}
\end{equation}
Moreover, for $\rho>0$ sufficiently small, we have that
\begin{align}
\{ \bar X_n\in(\mathbb D\setminus\underline{\mathbb D}_{\ll j}^\mu)^{-\gamma},\, \hat X_n \in (\underline{\mathbb D}_{0}^\mu)_\rho \} 
\subseteq 
\{ R_n \in (\mathbb D\setminus\underline{\mathbb D}_{\ll j})^{-\rho} \}, \label{eq: proposition 6.2 eq07}
\end{align}
and that 
\begin{equation}
\notag\{ \bar X_n\in(\mathbb D\setminus\underline{\mathbb D}_{\ll j}^\mu)^{-\gamma},\, \hat X_n \in (\underline{\mathbb D}_{\ll j+1-i}^\mu)_\rho \cap (\mathbb D\setminus\underline{\mathbb D}_{\ll j-i}^\mu)^{-\rho} \} \subseteq 
\{ \hat X_n \in (\mathbb D\setminus\underline{\mathbb D}_{\ll j-i}^\mu)^{-\rho},\, R_n\in(\mathbb D\setminus\underline{\mathbb D}_{\ll i})^{-\rho} \}, \label{eq: proposition 6.2 eq08}
\end{equation}
for all $i\in\{1,\ldots,j+k-1\}$. 
In view of \eqref{eq: proposition 6.2 eq05}--\eqref{eq: proposition 6.2 eq08}, we have that 
\begin{align}
\nonumber&\P ( \bar X_n\in(\mathbb D\setminus\underline{\mathbb D}_{\ll j}^\mu)^{-\gamma},\, \bar X_n' \in (\underline{\mathbb D}_{\ll j}^\mu)_{\rho/3} )\\
\nonumber&\leq 
\P ( \hat X_n \in (\mathbb D\setminus\underline{\mathbb D}_{\ll j}^\mu)^{-\rho}, d_{M_1'} ( \bar X_n',\hat X_n) \geq \rho/3 ) 
+ 
\P ( R_n \in (\mathbb D\setminus\underline{\mathbb D}_{\ll j})^{-\rho} )\\
&\hspace{12.5pt}+
\sum_{i=1}^{j+k-1} \P ( \hat X_n \in (\mathbb D\setminus\underline{\mathbb D}_{\ll j-i}^\mu)^{-\rho},\, R_n\in(\mathbb D\setminus\underline{\mathbb D}_{\ll i})^{-\rho} ),\label{eq: proposition 6.2 eq09}
\end{align}
where the first term in the previous inequality is of order $o( n^{-j(\alpha-1)} )$ due to Lemma~\ref{lem: help lemma for two-sided LD}~(2) above. 
Turning to estimating the summation in \eqref{eq: proposition 6.2 eq09}, we define 
$R_{p,n} = \{ R_{p,n}(t), t\in[0,1] \}$ by 
\[
R_{p,n}(t) = \frac{1}{n} \sum_{i=r_p}^{\lfloor r_{p+1} t \rfloor - 1} X_i.
\] 
Using the facts that $R_{N(n),n}(t) = R_n(r_{N(n)+1}t/n)$ and $r_{N(n)+1}/n > 1$ a.s., we have that
\begin{equation}\label{eq: proposition 6.2 eq10}
R_n \in(\mathbb D\setminus\underline{\mathbb D}_{\ll i})^{-\rho}\ \Rightarrow\ R_{N(n),n} \in(\mathbb D\setminus\underline{\mathbb D}_{\ll i})^{-\rho/2}. 
\end{equation}
Define
$
\bar X_{\leqslant p,n} = \{\bar X_{\leqslant p,n}(t),t\in[0,1]\}
$ 
by
$
\bar X_{\leqslant p,n}(t) = (1/n) \sum_{i=0}^{(\lfloor nt \rfloor \wedge r_{N(n)\wedge p})-1} X_i
$. 
In view of \eqref{eq: proposition 6.2 eq10}, we have that 
\begin{align*}
&\P ( \hat X_n \in (\mathbb D\setminus\underline{\mathbb D}_{\ll j-i}^\mu)^{-\rho},\, R_n\in(\mathbb D\setminus\underline{\mathbb D}_{\ll i})^{-\rho} )\\
&\leq
\P ( \hat X_n \in (\mathbb D\setminus\underline{\mathbb D}_{\ll j-i}^\mu)^{-\rho},\, R_{N(n),n} \in(\mathbb D\setminus\underline{\mathbb D}_{\ll i})^{-\rho/2} )\\
&\leq 
\P ( \hat X_n \in (\mathbb D\setminus\underline{\mathbb D}_{\ll j-i}^\mu)^{-\rho},\, R_{N(n),n} \in(\mathbb D\setminus\underline{\mathbb D}_{\ll i})^{-\rho/2},\, \mathfrak E_3^\epsilon(n) ) 
+ 
\P( \mathfrak E_3^\epsilon(n)^c )
\\
&=
\sum_{p=N_\epsilon^-(n)}^{N_\epsilon^+(n)} \P ( \hat X_n \in (\mathbb D\setminus\underline{\mathbb D}_{\ll j-i}^\mu)^{-\rho},\, R_{N(n),n} \in(\mathbb D\setminus\underline{\mathbb D}_{\ll i})^{-\rho/2},\, N(n) = p )
+ 
o( n^{-j(\alpha-1)} )\\ 
&=
\sum_{p=N_\epsilon^-(n)}^{N_\epsilon^+(n)} \P ( \bar X_{\leqslant p,n} \in (\mathbb D\setminus\underline{\mathbb D}_{\ll j-i}^\mu)^{-\rho},\, R_{p,n} \in(\mathbb D\setminus\underline{\mathbb D}_{\ll i})^{-\rho/2},\, N(n) = p )
+ 
o( n^{-j(\alpha-1)} )\end{align*}\begin{align*} 
&\leq 
\sum_{p=N_\epsilon^-(n)}^{N_\epsilon^+(n)} \P ( \bar X_{\leqslant p,n} \in (\mathbb D\setminus\underline{\mathbb D}_{\ll j-i}^\mu)^{-\rho} ) \P ( R_{p,n} \in(\mathbb D\setminus\underline{\mathbb D}_{\ll i})^{-\rho/2} ) 
+ 
o( n^{-j(\alpha-1)} )\\
&\leq 
\P ( \hat X_n \in (\mathbb D\setminus\underline{\mathbb D}_{\ll j-i}^\mu)^{-\rho/2} ) \sum_{p=N_\epsilon^-(n)}^{N_\epsilon^+(n)} \P ( R_{p,n} \in(\mathbb D\setminus\underline{\mathbb D}_{\ll i})^{-\rho/2} ) 
+ 
o( n^{-j(\alpha-1)} )\\ 
&= 
\mathcal O( n^{-(j-i)(\alpha-1)} ) 2\epsilon \mathcal O( n^{-i(\alpha-1)} ) + o( n^{-j(\alpha-1)} ),
\end{align*}
where in the final step we use Lemma~\ref{lem: help lemma for two-sided LD}~(2)--(3). 
Letting $\epsilon\to0$, we prove that the summation in \eqref{eq: proposition 6.2 eq09} is of order $o ( n^{-j(\alpha-1)} )$. 
Similarly, it can be shown that $\P ( R_n \in (\mathbb D\setminus\underline{\mathbb D}_{\ll j})^{-\rho} )$, and hence, $\P ( \bar X_n\in(\mathbb D\setminus\underline{\mathbb D}_{\ll j}^\mu)^{-\gamma},\, \bar X_n' \in (\underline{\mathbb D}_{\ll j}^\mu)_{\rho/3} )$ are of order $o ( n^{-j(\alpha-1)} )$. 
\end{proof}

%

\linkdest{proof of lem: help lemma for two-sided LD}
\begin{proof}[Proof of Lemma~\ref{lem: help lemma for two-sided LD}]
\linksinpf{lem: help lemma for two-sided LD}
Let $\mathbb D^s$ denote the set of all step functions in $\mathbb D$. 
Let $\mathbb D^{s,\uparrow}$ denote the set of all nondecreasing step functions in $\mathbb D$. 
Define the mapping $\Psi^\uparrow \colon \mathbb D^s\to\mathbb D^{s,\uparrow}$ by $\zeta = \Psi^\uparrow (\xi)$ and
\begin{equation}\label{eq:Psi+}
\zeta(t) = \inf \{ \zeta'(t)\in\mathbb R\colon \zeta'\in\mathbb D^{s,\uparrow},\, \zeta' \geq \xi \}, 
\qquad\mbox{for all}\ t\in[0,1].
\end{equation}
Basically, $\Psi^\uparrow (\xi)$ is the least possible nondecreasing step function such that $\Psi^\uparrow (\xi) \geq \xi$.

\textit{Part 1):}
First we show that $\P ( d_{M_1'}(\bar X_{i,n},\bar X_{i,n}') \geq \delta ) \leq \P ( T_2(n^\beta)<r_1 ) + o(n^{-(2-\epsilon)\alpha})$, for any $\beta\in(0,1)$. 
To begin with, setting $\beta^0 = (1-\beta)/2$ we have that
\begin{align*}
\P ( d_{M_1'}(\bar X_{i,n},\bar X_{i,n}') \geq \delta ) 
&\leq 
\P ( d_{M_1'}(\bar X_{i,n},\bar X_{i,n}') \geq \delta,\, r_i-r_{i-1} \leq n^{\beta_0} )+ 
\P ( r_i-r_{i-1} > n^{\beta_0} )\\
&= 
\P ( d_{M_1'}(\bar X_{i,n},\bar X_{i,n}') \geq \delta,\, r_i-r_{i-1} \leq n^{\beta_0} ) 
+ o(n^{-(2-\epsilon)\alpha}).
\end{align*}
Hence, it is sufficient to show that 
\begin{equation}\label{eq: lemma6.7 eq01}
\P ( d_{M_1'}(\bar X_{i,n},\bar X_{i,n}') \geq \delta,\, r_i-r_{i-1} \leq n^{\beta_0} ) \leq \P ( T_2(n^\beta)<r_1 ).
\end{equation}
Note that $d_{M_1'}(\bar X_{i,n},\bar X_{i,n}') \geq \delta$ implies $\| \bar X_{i,n}-\bar X_{i,n}' \|_\infty \geq \delta$, and hence,
\[
\delta 
\leq 
\sup_{k\leq r_i\wedge n} \left| \frac{1}{n} \sum_{j=r_{i-1}}^{k-1} X_j \right| 
\leq 
\sup_{k\leq r_i} \left| \frac{1}{n} \sum_{j=r_{i-1}}^{k-1} X_j \right|.
\]
It is sufficient to show that 
\[
\left\{ \sup_{k\leq r_i} \left| \frac{1}{n} \sum_{j=r_{i-1}}^{k-1} X_j \right| \geq \delta,\, d_{M_1'}(\bar X_{i,n},\bar X_{i,n}') \geq \delta,\, r_i-r_{i-1} \leq n^{\beta_0} \right\}
\]
is a subset of $\{ T_2(n^\beta) < r_1 \}$. 
We distinguish between the cases 1)  $\sup_{k\leq r_i} \frac{1}{n} \sum_{j=r_{i-1}}^{k-1} X_j \geq \delta$, and 2) $\inf_{k\leq r_i} \frac{1}{n} \sum_{j=r_{i-1}}^{k-1} X_j \leq -\delta$.
We focus on 1), since 2) can be dealt with by replacing $X_i$ by $-X_i$. 
Note that 
\[
\sup_{k\leq r_i\wedge n} \sum_{j=r_{i-1}}^{k-1} X_j \geq \delta n,\ r_i-r_{i-1} \leq n^{\beta_0}
\] 
implies the existence of $k_1\in\{r_{i-1},\ldots,r_i-1\}$ such that $X_{k_1} > n^{1-\beta_0} > n^\beta$. 
Now, suppose that $X_k \geq -n^\beta $ for all $k\in\{r_{i-1},\ldots,r_i-1\}$. 
Then the following statements hold.
\begin{itemize}
\item[(i)] For $n$ sufficiently large, we have
\[
\sup_{t\in[0,1]} \Psi^\uparrow(\bar X_{i,n})(t) - \sup_{t\in[0,1]} \bar X_{i,n}'(t)
\leq
n^{-1} (r_i - r_{i-1}) n^\beta \leq n^{\beta+\beta_0-1} \leq \delta/3,
\]
and hence, 
\[
\sup_{t\in[0,1]} \bar X_{i,n}'(t) 
\geq
\sup_{t\in[0,1]} \Psi^\uparrow(\bar X_{i,n})(t)  - \delta/3 \geq 2/3\delta>0.
\]
Moreover, both $\Psi^\uparrow(\bar X_{i,n})\in\mathbb D^{s,\uparrow}$ and $\bar X_{i,n}'\in\mathbb D^{s,\uparrow}$ are nonnegative functions in $\mathbb D$. 
Combining these with $r_i-r_{i-1} \leq n^{\beta_0}$, we have that, for sufficiently large $n$, 
\[
d_{M_1'}( \Psi^\uparrow(\bar X_{i,n}),\bar X_{i,n}' )\leq 
\left\{ \sup_{t\in[0,1]} \Psi^\uparrow(\bar X_{i,n})(t) - \sup_{t\in[0,1]} \bar X_{i,n}'(t) \right\} \vee (r_i - r_{i-1})/n 
\leq 
\delta/3.
\]
\item[(ii)] For $n$ sufficiently large,  
\[
d_{M_1'}( \Psi^\uparrow(\bar X_{i,n}),\bar X_{i,n} )
\leq 
\| \Psi^\uparrow(\bar X_{i,n}) - \bar X_{i,n} \|_\infty
\leq
n^{-1} (r_i - r_{i-1}) n^\beta \leq n^{\beta+\beta_0-1} \leq \delta/3. 
\]
\end{itemize}
In view of (i) and (ii), we have that 
\[
d_{M_1'}( \bar X_{i,n},\bar X_{i,n}' ) 
\leq 
d_{M_1'}( \bar X_{i,n},\Psi^\uparrow(\bar X_{i,n}) ) + d_{M_1'}( \Psi^\uparrow(\bar X_{i,n}),\bar X_{i,n}' )
\leq 
2\delta/3
,\] 
which leads to the contradiction of $d_{M_1'}( \bar X_{i,n},\bar X_{i,n}' ) \geq \delta$. 
Hence, we prove \eqref{eq: lemma6.7 eq01}. 

Next we show that $\P ( \bar X_{i,n} \in (\mathbb D\setminus\underline{\mathbb D}_{\ll k})^{-\delta} ) = \P( T_k(n^\beta) < r_1 ) + o(n^{-(k-\epsilon)\alpha})$ for any $\beta\in(0,1)$. 
First we claim that
\begin{equation}  
d(\xi, \underline{\mathbb D}_{\ll k}) > \delta
\ \Rightarrow
\exists (t_0,\ldots,t_k)\ \mbox{s.t.}\ 0\leq t_0 <  \cdots < t_k \leq 1,\, |\xi(t_i) - \xi(t_{i-1})| > \delta,\, i = 1,\ldots,k. \label{eq: lemma6.7 eq15}
\end{equation}
To see this, assume that the opposite holds. 
Set $s_0=0$ and 
\[
s_i = \sup \{ t \in (s_{i-1},1] \colon |\xi(t) - \xi(s_{i-1})| \leq \delta \},
\] 
for $i = 1,\ldots,k$.
Define $\zeta\in\mathbb D$ by $\zeta(t) = \xi(s_i)$ for $s_i\leq t < s_{i+1}$. 
Due to the assumption, we have $\zeta \in \underline{\mathbb D}_{\ll k}$, $d(\xi,\zeta) \leq \delta$, and hence, $d(\xi,\underline{\mathbb D}_{\ll k}) \leq \delta$. 
This leads to the contradiction of $d(\xi,\underline{\mathbb D}_{\ll k}) > \delta$. 
Thus, we proved \eqref{eq: lemma6.7 eq15}. 
Using the fact that $\P(r_1 > n\delta/2)$ decays exponentially, we are able to restrict ourselves to the case where $r_1 \leq n\delta/2$. 
Let $(t_0,\ldots,t_k)$ be as in the r.h.s.\ of \eqref{eq: lemma6.7 eq15}.
Using the fact that, under the $M_1'$ topology, jumps with the same sign ``merge'' into one jump in case they are ``close'', we conclude that $\mbox{sign}(\xi(t_i))\mbox{sign}(\xi(t_{i-1})) = -1$ for $i\in\{1,\ldots,k\}$. 
Combining this with the fact that $\P(r_1 > n^{(1-\beta)}) = o(n^{-(k-\epsilon)\alpha})$ we obtain that 
\begin{align}
\notag
\P ( \bar X_{i,n} \in (\mathbb D\setminus\underline{\mathbb D}_{\ll k})^{-\delta} )&= 
\P ( \bar X_{i,n} \in (\mathbb D\setminus\underline{\mathbb D}_{\ll k})^{-\delta},\, r_1 \leq n^{(1-\beta)} ) + \P(r_1 > n^{(1-\beta)})\\ 
&= 
\P( T_k(n^\beta) < r_1 ) + o(n^{-(k-\epsilon)\alpha})\label{eq: lemma6.7 eq02}
\end{align} 
for any $\beta\in(0,1)$.

Now, it remains to show that $\P ( T_k(u^\beta) < r_1 ) = \mathcal O( u^{-(k-\epsilon)\alpha} )$ as $u\to\infty$. 
We prove this by induction in $k$. 
For the base case we need to show $\P ( T_2(n^\beta)<r_1 ) = \mathcal O (n^{-(2-\epsilon)\alpha})$. 
Recalling  
$K_\beta^\gamma(u) = \inf \{ n>T(u^\beta)\colon |X_n| \leq u^\gamma \}$, 
we have that 
\begin{align}
\P ( T_2( u^\beta ) < r_1 ) 
\nonumber&= 
\P ( T_2( u^\beta ) < K_\beta^\gamma(u)  ) 
+ 
\P ( T_1( u^\beta ) < K_\beta^\gamma(u) <  T_2( u^\beta ) < r_1 )\\ 
\label{eq: lemma6.7 eq03}&= 
\P ( T_2( u^\beta ) < K_\beta^\gamma(u)  ) 
+ 
\mathcal O( u^{-(2\beta-\gamma)\alpha} ),
\end{align}
where $\P (T_1( u^\beta ) < K_\beta^\gamma(u) <  T_2( u^\beta ) < r_1) = \mathcal O( u^{-(2\beta-\gamma)\alpha} )$ can be deduced by following the arguments as in the proof of Proposition~\ref{prop: neglibigle parts in the area under first return time}. 
Applying the dual change of measure $\mathscr D$ over the time interval $[0,T_1(u^\beta)]$, we obtain that
\begin{align}
&
u^{(2\beta-\gamma)\alpha} \P ( T_2( u^\beta ) < K_\beta^\gamma(u)  ) 
\nonumber
\\
&= 
u^{(2\beta-\gamma)\alpha} \E^\mathscr D \left[ e^{-\alpha S_{T(u^\beta)}} \1_{\{ T(u^\beta) < r_1 \}} \P^\mathscr D ( T_2( u^\beta ) < K_\beta^\gamma(u) \,|\, \mathcal F_{T(u^\beta)} ) \right]
\nonumber
\\ 
\label{eq: lemma6.7 eq04}&= 
\E^\mathscr D \left[ \1_{\{T(u^\beta) < r_1\}} u^{(\beta-\gamma)\alpha} \P^\mathscr D ( T_2( u^\beta ) < K_\beta^\gamma(u) \,|\, \mathcal F_{T(u^\beta)} ) \left| \frac{X_{T(u^\beta)}}{u^\beta Z_{T(u^\beta)}} \right|^{-\alpha} \right].
\end{align}
Recalling $\mathfrak E_2(u) = \{ | B_n | \leq u^\gamma,\forall 1\leq n < K_\beta^\gamma(u) \}$, 
we have that, for $|v|\geq1$
\begin{align}
\nonumber&
\P^\mathscr D ( T_2( u^\beta ) < K_\beta^\gamma(u) \,|\, X_{T(u^\beta)} = v u^\beta )\\
\nonumber&\leq 
\P^\mathscr D ( | B_n | \leq u^\gamma,\forall T(u^\beta) < n < r_1,\, T_2( u^\beta ) < K_\beta^\gamma(u) \,|\, X_{T(u^\beta)} = v u^\beta )\\
\nonumber&\hspace{12.5pt}+ 
\P^\mathscr D ( \exists T(u^\beta) < n < r_1\ s.t.\ | B_n | > u^\gamma \,|\, X_{T(u^\beta)}=v u^\beta )\\ 
&= 
\P ( (\mathfrak E_2(u))^c \,|\, X_0 = v u^\beta )
= o (u^{-(\beta-\gamma)\alpha}) v,
\label{eq: lemma6.7 eq05}
\end{align}
where the tail estimate in \eqref{eq: lemma6.7 eq05} is obtained by following the arguments in the proof of Lemma~\ref{lem: bounding sign-switching event} and taking advantage of the additional assumption that $\E |B_1|^m<\infty$ for every $m\in\mathbb Z_+$. 
Plugging \eqref{eq: lemma6.7 eq05} into \eqref{eq: lemma6.7 eq04} and using the dominated convergence theorem, we obtain that 
\begin{align}
u^{(2\beta-\gamma)\alpha} \P ( T_2( u^\beta ) < K_\beta^\gamma(u) ) 
= 
o (1).\label{eq: lemma6.7 eq06}
\end{align}
In view of \eqref{eq: lemma6.7 eq01}, \eqref{eq: lemma6.7 eq03}, and \eqref{eq: lemma6.7 eq06}, 
\[
\P ( T_2(n^\beta)<r_1 ) = \mathcal O (n^{-(2\beta-\gamma)\alpha}) = \mathcal O (n^{-(2-\epsilon)\alpha}),
\] 
by choosing $\beta = 1 - \epsilon/3$ and $\gamma = \epsilon/3$. 
Turning to the inductive step, suppose that $\P(T_k(u^\beta) < r_1) = \mathcal O(u^{-(k-\epsilon)\alpha})$. 
Note that 
\begin{align*}
\P(T_{k+1}(u^\beta)<r_1) = 
\P(T_k(u^\beta)<K_\beta^\gamma(u)<T_{k+1}(u^\beta)<r_1) + \P(T_{k+1}(u^\beta)<K_\beta^\gamma(u)),
\end{align*}
where for the first term in the previous sum we have that 
\begin{align*}
\P(T_k(u^\beta)<K_\beta^\gamma(u)<T_{k+1}(u^\beta)<r_1)&\leq
\P(T_k(u^\beta)< r_1) \P(T(u^\beta)<r_1 | X_0 = u^\gamma)\\
&= 
\mathcal O( u^{-(k-\epsilon')\alpha} ) \mathcal O( u^{-(\beta-\gamma)\alpha} )
= 
\mathcal O( u^{-(k+1-\epsilon)\alpha} ),
\end{align*}
for suitable choice of $\beta$ and $\gamma$. 
Hence, it remains to bound $\P(T_{k+1}(u^\beta)<K_\beta^\gamma(u))$. 
Applying the dual change of measure $\mathscr D$ over the time interval $[0,T_1(u^\beta)]$, we obtain that
\begin{align}
\nonumber&u^{((k+1)\beta-\gamma)\alpha}\P(T_{k+1}(u^\beta)<K_\beta^\gamma(u))\\
\label{eq: lemma6.7 eq07}&= 
\E^\mathscr D \left[ \1_{\{T(u^\beta) < r_1\}} u^{(k\beta-\gamma)\alpha} \P^\mathscr D ( T_{k+1}( u^\beta ) < K_\beta^\gamma(u) \,|\, \mathcal F_{T(u^\beta)} ) \left| \frac{X_{T(u^\beta)}}{u^\beta Z_{T(u^\beta)}} \right|^{-\alpha} \right].
\end{align}
Moreover, we have that, for $|v|\geq 1$,
\begin{align}
\nonumber&
\P^\mathscr D ( T_{k+1}( u^\beta ) < K_\beta^\gamma(u) \,|\, X_{T(u^\beta)} = v u^\beta )\\ 
\nonumber&\leq 
\P^\mathscr D ( \exists T(u^\beta) < n_1<\cdots<n_k < r_1\ s.t.\ | B_{n_i} | > u^\gamma,\,\forall i\leq k \,|\, X_{T(u^\beta)}=v u^\beta )\\
\nonumber&=
\P ( \exists 0 < n_1<\cdots<n_k < r_1\ s.t.\ | B_{n_i} | > u^\gamma,\,\forall i\leq k \,|\, X_0=v u^\beta )\\
\label{eq: lemma6.7 eq08}&= 
\P ( \exists 0 < n_1<\cdots<n_k < r_1\ s.t.\ | B_{n_i} | > u^\gamma,\,\forall i\leq k )
= 
o(u^{-(k\beta-\gamma)\alpha}),
\end{align}
where the tail estimate in \eqref{eq: lemma6.7 eq08} is obtained by following the arguments in the proof of Lemma~\ref{lem: bounding sign-switching event} and taking advantage of the additional assumption that $\E |B_1|^m<\infty$ for every $m\in\mathbb Z_+$. 
Combining \eqref{eq: lemma6.7 eq07} and \eqref{eq: lemma6.7 eq08} with the fact that $| X_{T(u^\beta)}/u^\beta | \leq 1$ we obtain that $\P(T_{k+1}(u^\beta)<K_\beta^\gamma(u))$, and hence, $\P(T_{k+1}(u^\beta)<r_1) $ are of order $\mathcal O(u^{-(k+1-\epsilon)\alpha})$. 

\textit{Part 2):} By a similar reasoning as in proving part (1) of Proposition~\ref{prop: two-sided asymptotic equivalence between X_n' and X_n}, we have that
\begin{align*}
&\P( \bar X_n'\in(\mathbb D\setminus\underline{\mathbb D}_{\ll j}^\mu)^{-\gamma},\, d_{M_1'}(\bar X_n',\hat X_n)\geq\delta )\\ 
&\leq 
\P ( \bar X_n'\in \mathcal D_{\geqslant j}^{\bar \gamma},\, \exists\, i\leq N(n)\ s.t.\ d_{M_1'}(\bar X_{i,n},\bar X_{i,n}')\geq \delta ) + o ( n^{-j(\alpha-1)} )\\
&= 
o ( n^{-j(\alpha-1)} ),
\end{align*}
where $\mathcal D_{\geqslant j}^{\bar \gamma}$ is defined as in \eqref{eq:D>=j,k}. 
It remains to show that, for any $j\geq1$, $\gamma>0$, and $\delta>0$, there exists some $\rho>0$ so that 
\[
\P( \hat X_n\in(\mathbb D\setminus\underline{\mathbb D}_{\ll j}^\mu)^{-\gamma},\, \bar X_n'\in(\underline{\mathbb D}_{\ll j}^\mu)_\rho,\ d_{M_1'}(\bar X_n',\hat X_n)\geq\delta ) = o ( n^{-j(\alpha-1)} ),
\] 
as $n\to\infty$. 
Recall, for $\gamma>0$ and $j\geq 1$,
$
\mathcal D_{\geqslant j}^{\gamma} 
= 
\{ \xi\in \mathbb D \colon |\mbox{Disc}_{\gamma}(\xi)| \geq j \}
$, 
where
$
\mbox{Disc}_{\gamma}(\xi) = \{t\in\mbox{Disc}(\xi)\colon |\xi(t)-\xi(t^-)| \geq \gamma\}
$. 
Defining $\mathcal D_{= j}^{\rho} = \{ \xi\in \mathbb D \colon |\mbox{Disc}_{\gamma}(\xi)| = j \}$ for $j\in\mathbb Z$ and $\rho>0$, we have 
\begin{align}
\nonumber&
\P( \hat X_n\in(\mathbb D\setminus\underline{\mathbb D}_{\ll j}^\mu)^{-\gamma},\, \bar X_n'\in(\underline{\mathbb D}_{\ll j}^\mu)_\rho,\, d_{M_1'}(\bar X_n',\hat X_n)\geq\delta )\\
\nonumber&\leq 
\P( \bar X_n'\in \mathcal D_{\geqslant j-1}^{\rho_0} ,\, d_{M_1'}(\bar X_n',\hat X_n)\geq\delta )+
\P( \hat X_n\in(\mathbb D\setminus\underline{\mathbb D}_{\ll j}^\mu)^{-\gamma},\, \bar X_n'\in (\mathcal D_{\geqslant j-1}^{\rho_0})^c )\\
\nonumber&\leq 
\P( \bar X_n'\in \mathcal D_{\geqslant j-1}^{\rho_0},\, d_{M_1'}(\bar X_n',\hat X_n)\geq\delta )+\notag
\sum_{i=1}^{j-1} \P( \hat X_n\in(\mathbb D\setminus\underline{\mathbb D}_{\ll j}^\mu)^{-\gamma},\, \bar X_n'\in \mathcal D_{= j-i-1}^{\rho_0} )\\ 
\label{eq: lemma6.7 eq09}&= 
\P( \bar X_n'\in \mathcal D_{\geqslant j-1}^{\rho_0},\, d_{M_1'}(\bar X_n',\hat X_n)\geq\delta )
+
\sum_{i=1}^{j-1} \P ( E_j(i) ).
\end{align}
Note that 
\begin{align}
\nonumber&
\P( \bar X_n'\in \mathcal D_{\geqslant j-1}^{\rho_0},\, d_{M_1'}(\bar X_n',\hat X_n)\geq\delta )\\ 
\nonumber&\leq 
\P ( \bar X_n'\in \mathcal D_{\geqslant j-1}^{\rho_0},\, \exists\, i\leq N(n)\ s.t.\ d_{M_1'}(\bar X_{i,n},\bar X_{i,n}')\geq \delta ) + o ( n^{-j(\alpha-1)} )\\
\nonumber&= 
\P ( \bar X_n'\in \mathcal D_{\geqslant j-1}^{\rho_0},\, \mathfrak E_3^{\epsilon}(n) ,\, \exists\, i\leq N(n) s.t.\ d_{M_1'}(\bar X_{i,n},\bar X_{i,n}')\geq \delta ) + o( n^{-j(\alpha-1)} )\\
\nonumber&\leq 
\P ( \exists (i_0,\ldots,i_{j-2})\in \mathcal P(\{1,\ldots,N_\epsilon^+(n)\},j-1) \ \mbox{s.t.}\\ 
\nonumber&\hspace{25pt} 
d_{M_1'}(\bar X_{i_0,n},\bar X_{i_0,n}')\geq \delta,\, |X_{i_p}'|\geq n \rho_0,\forall 1\leq p\leq j-2 )\\
\label{eq: lemma6.7 eq10}&= 
\mathcal O ( n^{j-1} n^{-(2-\epsilon)\alpha} n^{-(j-2)\alpha} ) + o ( n^{-j(\alpha-1)} )
= 
o ( n^{-j(\alpha-1)} ),
\end{align}
where in \eqref{eq: lemma6.7 eq10} we use Lemma~\ref{lem: help lemma for two-sided LD}~(1) together with the fact that the blocks $\{X_{r_{i-1}},\ldots,X_{r_i}\}$, $i\geq1$, are mutually independent, and the final equivalence is obtained by setting $\epsilon<1/\alpha$. 
In view of the above computation, it remains to analyze $\P ( E_j(k) )$, $k\in\{1,\ldots,j-1\}$ as in \eqref{eq: lemma6.7 eq09}. 

Let $I^* = \{ i\leq N(n) \colon d_{M_1'}(\bar X_{i,n},\bar X_{i,n}')\geq \rho_1 \}$. 
Note that 
\begin{align*}
&\P ( \hat X_n\in(\mathbb D\setminus\underline{\mathbb D}_{\ll j}^\mu)^{-\gamma},\, \bar X_n'\in \mathcal D_{=  j-k-2}^{\rho_0} )\\
&= 
\P ( \hat X_n\in(\mathbb D\setminus\underline{\mathbb D}_{\ll j}^\mu)^{-\gamma},\, \bar X_n'\in \mathcal D_{=  j-k-2}^{\rho_0},\, |I^*| \geq (k+2)\wedge(j-k-2) )\\
&\hspace{12.5pt}+ 
\P ( \hat X_n\in(\mathbb D\setminus\underline{\mathbb D}_{\ll j}^\mu)^{-\gamma},\, \bar X_n'\in \mathcal D_{=  j-k-2}^{\rho_0},\, 1 \leq |I^*| < (k+2)\wedge(j-k-2) )\\
&\hspace{12.5pt}+ 
\P ( \hat X_n\in(\mathbb D\setminus\underline{\mathbb D}_{\ll j}^\mu)^{-\gamma},\, \bar X_n'\in \mathcal D_{=  j-k-2}^{\rho_0},\, |I^*| = 0 )\\
&= 
\textbf{(V.1)} + \textbf{(V.2)} + \textbf{(V.3)}.
\end{align*}
Suppose that $k \leq j/2 - 2$, where the case $k > j/2 - 2$ can be dealt with similarly. 
Note that 
\begin{align*}
\textbf{(V.1)}
&\leq 
\P ( \bar X_n'\in \mathcal D_{=  j-k-2}^{\rho_0},\, |I^*| \geq k+2 ,\, \mathfrak E_3^{\epsilon}(n) ) + o( n^{-j(\alpha-1)} )\\
&\leq 
\P ( \exists (i_1,\ldots,i_{j-k-2})\in \mathcal P(\{1,\ldots,N_\epsilon^+(n)\},j-k-2) \ \mbox{s.t.}\\ 
&\hspace{25pt} 
d_{M_1'}(\bar X_{i_p,n},\bar X_{i_p,n}')\geq \rho,\forall 1\leq p\leq k+2,\\ 
&\hspace{25pt} 
|X_{i_q}'|\geq n\rho_0,\forall k+3\leq q\leq j-k-2 )\\
\nonumber&\hspace{12.5pt}+ 
o( n^{-j(\alpha-1)} )\\
&= 
\mathcal O ( n^{j-k-2} n^{-(k+2)(2-\epsilon)\alpha} n^{-(j-2k-4)\alpha} ) + o ( n^{-j(\alpha-1)} )\\
&= 
\mathcal O ( n^{-j(\alpha-1)} n^{-(k+2)+(k+2)\epsilon\alpha} ) + 
o ( n^{-j(\alpha-1)} )
=o ( n^{-j(\alpha-1)} ),
\end{align*}
if $\epsilon < 1/\alpha$. 
Moreover, we have that $\textbf{(V.3)} = o ( n^{-j(\alpha-1)} )$ for $\rho_0$ sufficiently small. 
Let $I' = \{ i\leq N(n)\colon \bar X_{i,n}' \geq \rho_0 \}$. 
Turning to bounding \textbf{(V.2)} we have that 
\begin{align*}
\textbf{(V.2)}
&= 
\P ( \hat X_n\in(\mathbb D\setminus\underline{\mathbb D}_{\ll j}^\mu)^{-\gamma},\, \bar X_n'\in \mathcal D_{=  j-k-2}^{\rho_0},\, 1 \leq |I^*| \leq k+1 )\\
&= 
\sum_{k_1=1}^{k+1} \sum_{k_2=0}^{k_1} \P ( \hat X_n\in(\mathbb D\setminus\underline{\mathbb D}_{\ll j}^\mu)^{-\gamma},\, \bar X_n'\in \mathcal D_{=  j-k-2}^{\rho_0},\\ 
&\hspace{65pt} 
|I^*| = k_1,\, |I' \cap I^*| = k_2,\, \mathfrak E_3^{\epsilon}(n) )\\ 
&\hspace{12.5pt}+ 
o(n^{-j(\alpha-1)}).
\end{align*}
Defining $J = \{(l_1',\ldots,l_{k_1}')\colon \mathbf{1}^{T}(l_1',\ldots,l_{k_1}') < k+2+k_2\}$, it is now sufficient to consider 
\begin{align*}
&\P ( \hat X_n\in(\mathbb D\setminus\underline{\mathbb D}_{\ll j}^\mu)^{-\gamma},\, \bar X_n'\in \mathcal D_{=  j-k-2}^{\rho_0},\, |I^*| = k_1,\, |I' \cap I^*| = k_2,\, \mathfrak E_3^{\epsilon}(n) )\\
&\leq 
\P \Bigg ( \exists (i_1,\ldots,i_{j-k-2-k_2+k_1})\in \mathcal P(\{1,\ldots,N_\epsilon^+(n)\},j-k-2-k_2+k_1) \ \mbox{s.t.}\\ 
&\hspace{37.5pt} 
(\bar X_{i_1,n},\ldots,\bar X_{i_{k_1},n}) \in \Big( \bigcup_{(l_1,\ldots,l_{k_1})\in J} \prod_{p=1}^{k_1} \underline{\mathbb D}_{l_{i_p}} \Big)^{-\rho_2},\\  
&\hspace{37.5pt} 
|X_{i_q}'|\geq n\rho_0,\, \forall k_1+1\leq q\leq j-k-2-k_2+k_1 \Bigg)\\
&\hspace{12.5pt}+
\P \Bigg ( \hat X_n\in(\mathbb D\setminus\underline{\mathbb D}_{\ll j}^\mu)^{-\gamma},\, \bar X_n'\in \mathcal D_{=  j-k-2}^{\rho_0},\, |I^*| = k_1,\, |I' \cap I^*| = k_2,\, \mathfrak E_3^{\epsilon}(n),\\ 
&\hspace{50pt}
\exists (i_1,\ldots,i_{j-k-2-k_2+k_1})\in \mathcal P(\{1,\ldots,N_\epsilon^+(n)\},j-k-2-k_2+k_1)\\ 
&\hspace{50pt}\mbox{s.t.}\ 
(\bar X_{i_1,n},\ldots,\bar X_{i_{k_1},n}) \in \Big( \bigcup_{(l_1,\ldots,l_{k_1})\in J} \prod_{p=1}^{k_1} \underline{\mathbb D}_{l_{i_p}} \Big)_{\rho_2},\\
&\hspace{50pt} 
|X_{i_q}'|\geq n\rho_0,\, \forall k_1+1\leq q\leq j-k-2-k_2+k_1 \Bigg)\\ 
&=
\textbf{(V.2.a)} + \textbf{(V.2.b)}.
\end{align*}
Since $0\leq k_2\leq k_1 \leq k+1$ we have that 
\begin{align*}
\textbf{(V.2.a)} 
&\leq 
\mathcal O(n^{j-1}) \mathcal O(n^{-(k+2+k_2-k_1\delta)\alpha}) \mathcal O(n^{-(j-k-2-k_2)\alpha})\\
&= 
\mathcal O(n^{-j(\alpha-1)} n^{k_1\delta\alpha - 1})
= 
o (n^{-j(\alpha-1)}),
\end{align*}
for $\delta < 1/((k+1)\alpha)$. 
It remains to show that $\textbf{(V.2.b)} = o(n^{-j(\alpha-1)})$. 
To see this, for $\epsilon>0$ there exists 
\begin{equation}\label{eq: lemma6.7 eq11}
(\zeta_1,\ldots,\zeta_{k_1}) \in \bigcup_{(l_1,\ldots,l_{k_1})\in J} \prod_{p=1}^{k_1} \underline{\mathbb D}_{l_{i_p}}
\end{equation}
such that $d(\bar X_{i_p,n},\zeta_{i_p}) \leq \rho_2+\epsilon$, for all $1\leq p\leq k_1$. 
Hence, we have that 
\begin{equation}\label{eq: lemma6.7 eq12}
d\left( \hat X_n, \bar X_n'- \sum_{i\in I'\cap\{i_1,\ldots,i_{k_1}\}} \bar X_{i,n}' + \sum_{p=1}^{k_1} \zeta_{i_p} \right) \leq \rho_1\vee(\rho_2+\epsilon). 
\end{equation}
For any $c > 0$, define $\Phi_c \colon \mathbb D \to \mathbb D$ by 
\begin{equation}\label{eq: Phi_c}
\Phi_c(\xi)(t)
=
\sum_{s\in[0,t]\cap\mbox{Disc}(\xi,c)} ( \xi(s)-\xi(s^-) ), \quad \text{for } t\in[0,1],
\end{equation}
where $\mbox{Disc}(\xi,c) \{ t\in\mbox{Disc}(\xi)\colon \xi(t) - \xi(t^-) \geq c \}$. 
Now we claim that 
\begin{equation}\label{eq: lemma6.7 eq13}
\| \bar X_n' - \Phi_{\rho_0}(\bar X_n') - \mu\cdot\mbox{id} \|_\infty > \rho_3.
\end{equation}
To see this, suppose $\| \Phi_{\rho_0}(\bar X_n') - \mu\cdot\mbox{id} \|_\infty \leq \rho_3$. 
Hence, 
\begin{align}
\nonumber&d\left( \bar X_n'- \sum_{i\in I'\cap\{i_1,\ldots,i_{k_1}\}} \bar X_{i,n}' + \sum_{p=1}^{k_1} \zeta_{i_p}, \mu\cdot\mbox{id} + \sum_{p=1}^{k_1} \zeta_{i_p} + \sum_{i\in I'\setminus\{i_1,\ldots,i_{k_1}\}} \bar X_{i,n}' \right)\\
&\leq 
\left\|
\bar X_n' - \sum_{i\in I'} \bar X_{i,n}' - \mu\cdot\mbox{id} \right\|_\infty
=
\| \bar X_n' - \Phi_{\rho_0}(\bar X_n') - \mu\cdot\mbox{id} \|_\infty \leq \rho_3.\label{eq: lemma6.7 eq14}
\end{align}
In view of \eqref{eq: lemma6.7 eq12} and \eqref{eq: lemma6.7 eq14} we obtain that 
\[
d\left( \hat X_n, \mu\cdot\mbox{id} + \sum_{p=1}^{k_1} \zeta_{i_p} + \sum_{i\in I'\setminus\{i_1,\ldots,i_{k_1}\}} \bar X_{i,n}' \right)
\leq 
\rho_1 \vee (\rho_2+\epsilon) + \rho_3,
\]
where 
\[
\mu\cdot\mbox{id} + \sum_{p=1}^{k_1} \zeta_{i_p} + \sum_{i\in I'\setminus\{i_1,\ldots,i_{k_1}\}} \bar X_{i,n}' \in \underline{\mathbb D}_{\ll j}^\mu
\] 
due to \eqref{eq: lemma6.7 eq11}.
This leads to the contradiction of $\hat X_n\in(\mathbb D\setminus\underline{\mathbb D}_{\ll j}^\mu)^{-\gamma}$ by choosing $\rho_1$, $\rho_2$ and $\rho_3$ small enough. 
In view of \eqref{eq: lemma6.7 eq13} we have that 
\begin{align*}
\textbf{(V.2.b)}
&\leq 
\P\left(
\bar X_n'
\in
\left\{ \xi\in \mathbb D\colon \xi(t) - \sup_{t\in[0,1]} \Bigg| \Phi_{\rho_0}(\xi)(t) -\mu t \Bigg| > \rho_3 \right\}
\right)\\
&=
o(n^{-j(\alpha-1)}),
\end{align*}
by choosing $\rho_0$ and $\rho_3$ such that $\rho_3/\rho_0 \notin \mathbb Z$ and $\lceil \rho_3/\rho_0 \rceil > j$. 

\textit{Part 3):} 
Since
\[
\P ( r_{i+1}-r_i > r_i \delta ) \leq \P ( r_{i+1}-r_i > (n-\epsilon') \delta ) + \P ( r_i \geq n-\epsilon' ),
\] 
$\P ( r_{i+1}-r_i > r_i \delta )$ decays exponentially, for $i \in \{ N_\epsilon^-(n),\ldots,N_\epsilon^+(n) \}$. 
Combining this with \eqref{eq: lemma6.7 eq15}, we are able to utilize the argument as in \eqref{eq: lemma6.7 eq02} and obtain that 
\[
\P ( R_{i,n} \in(\mathbb D\setminus\underline{\mathbb D}_{\ll j})^{-\delta} ) 
= 
\P( T_j(n^\beta) < r_1 ) + o(n^{-(j-\epsilon)\alpha})
\]
for any $\beta \in (0,1)$. 
Since $\P ( T_j(u^\beta) < r_1 ) = \mathcal O( u^{-(j-\epsilon)\alpha} )$ for a suitable choice of $\beta$, the proof is completed. 
\end{proof}


\bibliographystyle{plain}
\bibliography{bibliography}

\newcommand{\noop}[1]{}
\begin{thebibliography}{10}

\bibitem{aspandiiarov1999}
S.~Aspandiiarov and R.~Iasnogorodski.
\newblock General criteria of integrability of functions of passage-times for
  nonnegative stochastic processes and their applications.
\newblock {\em Theory of Probability \& Its Applications}, 43(3):343--369,
  1999.

\bibitem{athreya1978}
K.~B. Athreya and P.~Ney.
\newblock A new approach to the limit theory of recurrent {M}arkov chains.
\newblock {\em Transactions of the American Mathematical Society},
  245:493--501, 1978.

\bibitem{Basrak2018}
B.~Basrak, H.~Planini\'{c}, and P.~Soulier.
\newblock An invariance principle for sums and record times of regularly
  varying stationary sequences.
\newblock {\em Probab. Theory Related Fields}, 172(3-4):869--914, 2018.

\bibitem{Mihail2020}
M.~Bazhba, J.~Blanchet, C.-H. Rhee, and B.~Zwart.
\newblock Sample path large deviations for {L}\'{e}vy processes and random
  walks with {W}eibull increments.
\newblock {\em The Annals of Applied Probability}, 30(6):2695--2739, 2020.

\bibitem{billingsley2013}
P.~Billingsley.
\newblock {\em Convergence of probability measures}.
\newblock John Wiley \& Sons, 2013.

\bibitem{borovkov2008}
A.~A. Borovkov and K.~A. Borovkov.
\newblock {\em Asymptotic analysis of random walks: Heavy-tailed
  distributions}.
\newblock Encyclopedia of Mathematics and its Applications. Cambridge
  University Press, 2008.

\bibitem{buraczewskidamekmikosch2016}
D.~Buraczewski, E.~Damek, and T.~Mikosch.
\newblock {\em Stochastic models with power-law tails: The equation X=AX+B}.
\newblock Springer Series in Operations Research and Financial Engineering.
  Springer International Publishing, 1st edition, 2016.

\bibitem{buraczewski2013}
D.~Buraczewski, E.~Damek, T.~Mikosch, and J.~Zienkiewicz.
\newblock Large deviations for solutions to stochastic recurrence equations
  under {K}esten's condition.
\newblock {\em The Annals of Probability}, 41(4):2755--2790, 2013.

\bibitem{chenthesis}
B.~Chen.
\newblock {\em Heavy tails: asymptotics, algorithms, applications}.
\newblock PhD thesis,
  https://research.tue.nl/en/publications/heavy-tails-asymptotics-algorithms-applications,
  2019.

\bibitem{chen2019}
B.~Chen, J.~Blanchet, C.-H. Rhee, and B.~Zwart.
\newblock Efficient rare-event simulation for multiple jump events in regularly
  varying random walks and compound {P}oisson processes.
\newblock {\em Mathematics of Operations Research}, 44(3):919--942, 2019.

\bibitem{ChenYao}
H.~Chen and D.~D. Yao.
\newblock {\em Fundamentals of queueing networks}, volume~46 of {\em
  Applications of Mathematics (New York)}.
\newblock Springer-Verlag, New York, 2001.
\newblock Performance, asymptotics, and optimization, Stochastic Modelling and
  Applied Probability.

\bibitem{collamore2007}
J.~F. Collamore and A.~H{\"o}ing.
\newblock Small-time ruin for a financial process modulated by a {H}arris
  recurrent {M}arkov chain.
\newblock {\em Finance and Stochastics}, 11(3):299--322, Jul 2007.

\bibitem{collamore2016}
J.~F. Collamore and S.~Mentemeier.
\newblock Large excursions and conditioned laws for recursive sequences
  generated by random matrices.
\newblock {\em The Annals of Probability}, 46(4):2064--2120, 2018.

\bibitem{collamore2013}
J.~F. Collamore and A.~N. Vidyashankar.
\newblock Tail estimates for stochastic fixed point equations via nonlinear
  renewal theory.
\newblock {\em Stochastic Processes and their Applications}, 123(9):3378--3429,
  2013.

\bibitem{denisov2008}
D.~Denisov, A.~B. Dieker, and V.~Shneer.
\newblock Large deviations for random walks under subexponentiality: {T}he
  big-jump domain.
\newblock {\em The Annals of Probability}, 36(5):1946--1991, 2008.

\bibitem{donsker1975asymptotic}
M.~Donsker and S.~Varadhan.
\newblock Asymptotic evaluation of certain markov process expectations for
  large time, ii.
\newblock {\em Communications in Pure and Applied Mathematics}, 28(2):279--301,
  1975.

\bibitem{DV_3}
M.~Donsker and S.~Varadhan.
\newblock Asymptotic evaluation of certain markov process expectations for
  large time—iii.
\newblock {\em Communications in Pure and Applied Mathematics}, 29(4):389--461,
  1976.

\bibitem{foss2007}
S.~Foss, T.~Konstantopoulos, and S.~Zachary.
\newblock Discrete and continuous time modulated random walks with heavy-tailed
  increments.
\newblock {\em Journal of Theoretical Probability}, 20(3):581--612, Sep 2007.

\bibitem{fosskorshunov2012}
S.~Foss and D.~Korshunov.
\newblock On large delays in multi-server queues with heavy tails.
\newblock {\em Mathematics of Operations Research}, 37(2):201--218, 2012.

\bibitem{fosskorshunovzachary2013}
S.~Foss, D.~Korshunov, and S.~Zachary.
\newblock {\em An introduction to heavy-tailed and subexponential
  distributions}, volume~38 of {\em Springer Series in Operations Research and
  Financial Engineering}.
\newblock Springer-Verlag New York, 2nd edition, 2013.

\bibitem{goldie1991}
C.~M. Goldie.
\newblock Implicit renewal theory and tails of solutions of random equations.
\newblock {\em The Annals of Applied Probability}, 1(1):126--166, 1991.

\bibitem{guivarch}
Y.~Guivarc'H and E.~Le~Page.
\newblock On the homogeneity at infinity of the stationary probability for an
  affine random walk.
\newblock In {\em Recent trends in ergodic theory and dynamical systems},
  volume 631 of {\em Contemp. Math.}, pages 119--130. Amer. Math. Soc.,
  Providence, RI, 2015.

\bibitem{hult2007}
H.~Hult and F.~Lindskog.
\newblock Extremal behavior of stochastic integrals driven by regularly varying
  {L}{\'e}vy processes.
\newblock {\em The Annals of Probability}, 35(1):309--339, 2007.

\bibitem{hult2005}
H.~Hult, F.~Lindskog, T.~Mikosch, and G.~Samorodnitsky.
\newblock Functional large deviations for multivariate regularly varying random
  walks.
\newblock {\em The Annals of Applied Probability}, 15(4):2651--2680, 2005.

\bibitem{kesten1973}
H.~Kesten.
\newblock Random difference equations and renewal theory for products of random
  matrices.
\newblock {\em Acta Mathematica}, 131:207--248, 1973.

\bibitem{KM_2}
I.~Kontoyiannis and S.~Meyn.
\newblock Large deviations asymptotics and the spectral theory of
  multiplicatively regular markov processes.
\newblock {\em Electronic Journal of Probability}, 10(3):61--123, 2005.

\bibitem{kontoyiannis2003spectral}
I.~Kontoyiannis and S.~P. Meyn.
\newblock Spectral theory and limit theorems for geometrically ergodic markov
  processes.
\newblock {\em The Annals of Applied Probability}, 13(1):304--362, 2003.

\bibitem{lindskogresnickroy2014}
F.~Lindskog, S.~Resnick, and J.~Roy.
\newblock Regularly varying measures on metric spaces: {H}idden regular
  variation and hidden jumps.
\newblock {\em Probability Surveys}, 11:270--314, 2014.

\bibitem{meyn2009}
S.~Meyn and R.~L. Tweedie.
\newblock {\em Markov chains and stochastic stability}.
\newblock Cambridge University Press, New York, NY, USA, 2nd edition, 2009.

\bibitem{mikosch2000}
T.~Mikosch and G.~Samorodnitsky.
\newblock Ruin probability with claims modeled by a stationary ergodic stable
  process.
\newblock {\em The Annals of Probability}, 28(4):1814--1851, Oct 2000.

\bibitem{mikosch2013}
T.~Mikosch and O.~Wintenberger.
\newblock Precise large deviations for dependent regularly varying sequences.
\newblock {\em Probability Theory and Related Fields}, 156(3):851--887, Aug
  2013.

\bibitem{mikosch2016}
T.~Mikosch and O.~Wintenberger.
\newblock A large deviations approach to limit theory for heavy-tailed time
  series.
\newblock {\em Probability Theory and Related Fields}, 166(1):233--269, Oct
  2016.

\bibitem{nagaev1969}
A.~V. Nagaev.
\newblock Integral limit theorems taking large deviations into account when
  {C}ram{\'e}r’s condition does not hold. {I}.
\newblock {\em Theory of Probability \& Its Applications}, 14(1):51--64, 1969.

\bibitem{nagaev1978}
A.~V. Nagaev.
\newblock On a property of sums of independent random variables.
\newblock {\em Theory of Probability \& Its Applications}, 22(2):326--338,
  1978.

\bibitem{rheeblanchetzwart2016}
C.-H. Rhee, J.~Blanchet, and B.~Zwart.
\newblock Sample path large deviations for lévy processes and random walks
  with regularly varying increments.
\newblock {\em Annals of Probability}, 6(47):3551--3605, 2019.

\bibitem{vysotsky}
V.~Vysotsky.
\newblock Contraction principle for trajectories of random walks and
  {C}ram\'{e}r's theorem for kernel-weighted sums.
\newblock {\em ALEA Lat. Am. J. Probab. Math. Stat.}, 18(2):1103--1125, 2021.

\bibitem{whitt2002}
W.~Whitt.
\newblock {\em Stochastic-process limits}.
\newblock Springer Series in Operations Research and Financial Engineering.
  Springer-Verlag New York, 1st edition, 2002.

\bibitem{zwart2004}
B.~Zwart, S.~Borst, and M.~Mandjes.
\newblock Exact asymptotics for fluid queues fed by multiple heavy-tailed
  on–off flows.
\newblock {\em The Annals of Applied Probability}, 14(2):903--957, May 2004.

\end{thebibliography}






\begin{acks}
We are grateful to Remco van der Hofstad and a referee for providing many useful suggestions for improvement. 
\end{acks}


\end{document}